\documentclass[11pt]{article}
\usepackage{bbm}
\usepackage{geometry}
\usepackage{amsfonts}
\usepackage{latexsym, amssymb, amsmath, amscd, amsthm, amsxtra}
\usepackage{mathtools}
\usepackage{enumerate}
\usepackage[all]{xy}
\usepackage{mathrsfs}
\usepackage{fancyhdr}
\usepackage{listings}
\usepackage{hyperref}
\usepackage{wrapfig}
\usepackage{soul}

\hypersetup{
     colorlinks=true
}
\def\P{\mathbb{P}}
\def\Pr{\mathrm{\mathbf{P}}}
\def\Ex{\mathrm{\mathbf{E}}}
\def\E{\mathbb{E}}

\def\Z{\mathbb{Z}}
\def\R{\mathbb{R}}

\def\Q{\mathbb{Q}}

\def\11{\mathbbm{1}}

\def\K{\mathcal{K}}

\def\D{\mathfrak{D}}
\def\sd{\mathcal D}

\def\calC{\mathcal{C}}
\def\calV{\mathcal{V}}
\def\calW{\mathcal{W}}
\def\gpt{\mathcal G_x}

\def\ft{\mathsf{T}}

\def\ob{\mathcal{O}}
\def\bba{\mathbf{\Lambda}}

\def\tuu{{2/3}}
\def\tuuovertwo{{1/3}}

\def\fuu{{1/10}}

\def\rad{\varrho_n}

\def\cregion{{\mathcal V_\lambda}}
\def\tempset{A}

\def\cgpt{2d}

\def\locE{\mathsf{E}}
\def\locV{\mathsf{V}}

\def\for{\quad\text{for}\quad}
\def\AND{\quad \text{and}\quad}
\def\epl{\iota}

\def\epm{{\epsilon}}
\def\cucc{C_0}

\def\emp{\mathcal E}

\def\fregion{{\mathcal U}}

\def\tOmega{{\widetilde \Omega_{\epm}}}

\def\hbn{{\widehat B_n}}

\def\mabox{\mathbb K}
\def\mac{\mathscr C}
\def\mbfi{\mathbf{i}}
\def\mbfj{\mathbf{j}}

\def\ck{\mathcal C_\mathbb K}

\def\Cbar{\bar C}
\def\cul{\underline c}

\newcommand\dif{\mathop{}\!\mathrm{d}}

\newtheorem{thm}{Theorem}[section]

\newtheorem{proposition}[thm]{Proposition}

\newtheorem{lemma}[thm]{Lemma}
\newtheorem{cor}[thm]{Corollary}
\newtheorem{defn}[thm]{Definition}

\theoremstyle{definition}
\newtheorem{remark}[thm]{Remark}

\numberwithin{equation}{section}
\usepackage{environ}
\begin{document}

\title{Localization for random walks among random obstacles in a single Euclidean ball}

\author{ Jian Ding\thanks{Partially supported by NSF grant DMS-1757479 and an Alfred Sloan fellowship.} \\ University of Pennsylvania \and Changji Xu\footnotemark[1]  \\
University of Chicago
}
\date{}

\maketitle

\begin{abstract}
Place an obstacle with probability $1-p$ independently at each vertex of $\mathbb Z^d$, and run a simple random walk until hitting one of the obstacles. For $d\geq 2$ and $p$ strictly above the critical threshold for site percolation, we condition on the environment where the origin is contained in an infinite connected component free of obstacles, and we show that for environments with probability tending to one as $n\to \infty$ there exists a unique discrete Euclidean ball of volume $d \log_{1/p} n$ asymptotically such that the following holds: conditioned on survival up to time $n$ we have that at any time $t \in [\epsilon_n n,n]$ (for some $\epsilon_n\to_{n\to \infty} 0$) with probability tending to one the simple random walk is in this ball. This work relies on and substantially improves a previous result of the authors on  localization in a region of volume poly-logarithmic in $n$ for the same problem.
\end{abstract}

\section{Introduction}

For $d\geq 2$, we consider a random environment where  each vertex of $\mathbb Z^d$  is placed with an obstacle independently with probability  $1- p$. On this random environment, we then consider a discrete-time simple random walk $(S_t)_{t \in \mathbb N}$ started at the origin and killed at time $\tau$ when the random walk hits an obstacle for the first time. In this paper, we study the quenched behavior of the random walk conditioned on survival for a large time. For convenience of notation, throughout the paper we use $\P$ (and $\E$) for the probability measure with respect to the random environment, and use $\Pr$ (and $\Ex$) for the probability measure with respect to the random walk. Our main result in the present article is the following.  
\begin{thm}
\label{ballthm}
For any fixed $d\geq 2$ and $p > p_c(\mathbb Z^d)$ (the critical threshold for site percolation), we condition on the event that the origin is in an infinite cluster (i.e., an infinite connected component) free of obstacles. Then 
  there exists a constant $C = C(d,p)$ and
a $\P$-measurable discrete ball $B_n \subseteq \Z^d$ of cardinality at most $(1+\epsilon_n)d \log_{1/p} n$ (where $\epsilon_n$ tends to 0 as $n\to \infty$) such that the following holds: for any $t\in [Cn (\log n)^{-2/d}, n]$
\begin{equation}
\label{eq:ball - localization}
    \Pr(S_t  \in B_n \mid \tau>n) \to 1 \mbox{ in } \P\mbox{-probability} \,.
\end{equation}
Here a discrete ball (in $\mathbb Z^d$) is the set that contains all the lattice points of some Euclidean ball (in $\mathbb R^d$). 
\end{thm}

A key step toward proving Theorem~\ref{ballthm} is the following result which states that conditioned on survival the random walk is localized in a single local neighborhood; this is sometimes known as the one-city theorem in literature. 

\begin{thm}\label{onecity}
  For any fixed $d\geq 2$ and $p > p_c(\mathbb Z^d)$, we condition on the event that the origin is in an infinite cluster free of obstacles. Let $v_*$ be the maximizer of the variational problem \eqref{eq:variational}, which is measurable with respect to the environment. Then there exist constants $\kappa, C$ only depending on $(d,p)$ such that the following holds: Let $T$ be the first time that the random walk visits $B_{(\log n)^{\kappa/2}}(v_*)$ (i.e., a ball centered at $v_*$ of radius $(\log n)^{\kappa/2}$). Then
   \begin{equation}
\Pr\Big(T \leq C |v_*|, \{S_t: t \in [T,n]\} \subseteq B_{(\log n)^{\kappa/2}}(v_*) \mid \tau > n\Big) \to 1 \mbox{ in } \P\mbox{-probability}\,.
   \end{equation} 

\end{thm}  

\subsection{Background and related results}

Random walks among random obstacles has been studied extensively in literature. In the annealed case, the logarithmic asymptotics for survival probabilities are closely related to the large deviation estimates for the range of the random walk/Wiener sausage \cite{DR75, DR79, Sznitman90, Sznitman93}.  The localization problem has also been studied in the annealed case, where in \cite{Bolthausen94,Sznitman91} it was proved that in dimension two the range of the random walk/Wiener sausage is asymptotically a ball and in \cite{Povel99} it was shown that in dimension three and higher the range of the Brownian motion is contained in a ball; in both cases the asymptotics of the radii for the balls were determined. 

The same problem in the quenched case was far more challenging: in the Brownian setting it was studied first in \cite{Sznitman93b, Sznitman97} and its logarithmic asymptotics for the survival probability was first derived in \cite{Sznitman93b} (see also the celebrated monograph \cite{Sznitman98}). In \cite{Fukushima09} a simple argument for the quenched asymptotics of the survival probability was given using the Lifshitz
tail effect.  In the random walk setting, the logarithmic asymptotics of survival probability was computed in \cite{Antal95} which built upon methods developed in the Brownian setting.  Stating the result of \cite{Antal95} in our setting, we have the following: conditioned on the origin being in the infinite open cluster, one has that with $\P$-probability tending to one as $n\to \infty$
\begin{equation}\label{eq:000-1}
\Pr(\tau > n) = \exp\{-c_*n(\log n)^{-2/d}(1 + o(1))\}\,.
\end{equation}
Here $c_* = \mu_{B} (\frac{\omega_d \log (1/p)}{d})^{2/d}$, $B$ is a unit ball in $\R^d$, $\omega_d$ is the volume of  $B$ and $\mu_{B}$ is the first eigenvalue of the Dirichlet-Laplacian of $B$ which is formally defined as follows. For an open set $\Omega\subseteq \mathbb R^d$ with finite measure, 
\begin{equation}\label{eq-def-first-eigenvalue}
\mu_\Omega = \frac{1}{2d}\min_{u\in W_0^{1, 2}(\Omega)} \Big\{\int_\Omega |\nabla u|^2 \dif x: \|u\|_{L^2(\Omega) =1}\Big\}\,,
\end{equation}
where $W_0^{1, 2}(\Omega)$ is the closure of $C_0^\infty(\Omega)$ in the norm $\|u\|_{W_0^{1, 2}(\Omega)} = (\int_\Omega |\nabla u|^2 dx)^{1/2}$. 


One strategy to obtain the lower bound in \eqref{eq:000-1} is for the random walk to travel in minimal possible number of steps to the largest ball in $[-n^{1+o(1)}, n^{1+o(1)}]^d$  which is free of obstacles and connected to the origin,  and then stays within that ball afterwards. By a straightforward computation, the radius of such ball should be asymptotic to
\begin{equation}
\label{eq:def-r}
	\rad := \lfloor (\omega_d^{-1}d\log_{1/p} n)^{1/d} \rfloor\,.
\end{equation}
 This suggests that the random walk will be localized in a ball of radius asymptotically $\rad$ (or equivalently of cardinality asymptotic to $d\log_{1/p} n$), as stated in Theorem \ref{ballthm}. 


An analogous localization result to Theorem~\ref{ballthm} was obtained in \cite{Sznitman96} (see also \cite{Sznitman98}) with an upper bound of $t^{o(1)}$  on the volume of localized region (a counterpart of $|B_n|$ in Theorem~\ref{ballthm}) and with a possibility that this region is split into $t^{o(1)}$ many local regions (or cities in view of the name of ``one-city theorem'' in literature) --- we remark that many other results have been established in \cite{Sznitman98} such as Lyapunov exponents and many deep connections to variational problems. In the random walk setting, the logarithmic asymptotics of survival probability was computed in \cite{Antal95} which built upon methods developed in the Brownian setting. Recently, in \cite{DX17} we have shown that ever since $\epsilon_n n$ steps (for some $\epsilon_n\to_{n\to \infty} 0$) the random walk is localized in a region of volume that is poly-logarithmic in $n$ in the same context of the present paper.

Thus far, we have been discussing random walks with Bernoulli obstacles, i.e., at each vertex we kill the random walk with probability either 0 or a certain fixed number. More generally, one may place i.i.d.\ random potentials $\{W_v: v\in \mathbb Z^d\}$ (where $W_v$ follows a general distribution) and one assigns a random walk path probability proportional to $\exp(\sum_{i=0}^n W_{S_i})$. The case of Bernoulli obstacles is a prominent example in this family. Previously, there has been a huge amount of work devoted to the study of various localization phenomenon when the potential distribution exhibits some tail behavior ranging from heavy tail to doubly-exponential tail.  
See  \cite{Wolfgang16} for an almost up-to-date review on this subject, also known as the parabolic Anderson model. See also \cite{ADS17} for a review on random walk among mobile/immobile random traps. 

For a very partial review, in the works of \cite{GMS83, AS07, AS08, HMS08, KLMS09, LM12, ST14}, much progress has been made for heavy-tailed potential where in particular they proved localization in a single lattice point. We note that by localization in a single lattice point we meant for a single large $t$, as considered in the present article; one could alternatively consider the behavior for all large $t$ simultaneously as in \cite{KLMS09}, in which case they showed that the random walk is localized in two lattice points, almost surely as $t\to \infty$.   In a few recent papers \cite{BK16, BKS16} (which improved upon \cite{GKM07}), the case of doubly-exponential potential was tackled where detail behavior on leading eigenvalues and eigenfunctions, mass concentration as well as aging were established. In particular, it was proved that in the doubly-exponential case the mass was localized in a bounded neighborhood of a site that achieves an optimal compromise between the local Dirichlet eigenvalue of the Anderson Hamiltonian and the distance to the origin.

\subsection{From poly-logarithmic localization to sharp localization}\label{sec-poly-to-sharp}

The main result in \cite{DX17} is that the random walk is confined in at most poly-logarithmic in $n$ many \emph{islands} (where an \emph{island} is a connected subset in $\Z^d$) during time $[o(n),n]$ (that is, during time $[\epsilon_n n, n]$ for some $\epsilon_n\to_{n\to \infty} 0$) and each island has diameter at most poly-logarithmic in $n$ --- we will refer to each such island as a \emph{pocket island}. The present article is closely related to our previous work \cite{DX17}: we rely both on the results and techniques in \cite{DX17} (we note that our proof is otherwise self-contained and in particular does not rely on results in \cite{Sznitman98}).  Provided with \cite{DX17}, our  proof of Theorem~\ref{ballthm} is naturally divided into two essentially separate parts, as we describe in what follows.
\vspace{3pt}

\noindent{\bf One-city theorem.} The first step is to show that one of these pocket islands will stand out and dominate the union of the rest of them, as incorporated in Sections~\ref{sec:one-city} and \ref{sec:prop3.3}. To this end, we consider the probability cost in order for the random walk to travel from the origin to a faraway island, and we will show that the logarithm of this probability grows (roughly speaking) linearly in the distance from the origin to the island provided that the angle is fixed.  Thus, this probability cost has a large fluctuation since pocket islands occur more or less uniformly in the box under consideration. Due to fluctuation, the best  pocket island will substantially dominate all the others. A refined version of ``logarithm of the probability cost grows linearly in distance'' as incorporated in Proposition \ref{phi-and-g}, is a major ingredient in proving Theorem~\ref{onecity}. 

\smallskip

\noindent{\bf Intermittent island.} While the random walk is confined in a pocket island during time $[o(n),n]$, we will show that at each given time $t \in [o(n),n]$ it is in a much smaller region (called \emph{intermittent island}) inside the pocket island with high probability. In addition, we will show that the intermittent island is asymptotically a ball. These are contained in Sections~\ref{section:Asymptotic Ball}  and \ref{section:Localization on intermittent island}.
 An important observation here is that the total volume for regions with low obstacle density  in any pocket island is at most $d \log_{1/p} n(1 + o(1))$. This observation, combined with the celebrated Farber--Krahn inequality (see Section~\ref{sec:ingredients} below), then implies that the asymptotic shape of the intermittent island is a ball. 
 
 We would like to add a remark that in this paper, we have used the term \emph{intermittent island} in a manner that is not completely precise. For instance, we have referred to $\emp$ in Definition 
 \ref{empty-set}, $\Omega_\epsilon$ in Definition \ref{def-eign-f}, $B_\epsilon$ in Lemma~\ref{Asymptotic shape - Ball}, $\hbn$ in \eqref{eq:def-hbn} as intermittent island in informal discussions. The abuse of the terminology is justified by the fact that all of these sets have negligible pair-wise symmetric differences (and of course our mathematical statements are always precisely formulated).

\subsection{Two important proof ingredients}\label{sec:ingredients}

In Section~\ref{sec-poly-to-sharp} we described the high-level structure of our proof in two essentially separate steps;
in this subsection we will discuss two important proof ingredients, which provides  a glance at some highlights of our proof.\ Discussions on more detailed proof ideas can be found at the beginning of  Sections~\ref{sec:one-city}, \ref{sec:prop3.3}, \ref{section:Asymptotic Ball}, \ref{section:Localization on intermittent island}.

\noindent{\bf Convergence rates for sub-additive functions.} Rate of convergence for sub-additive functionals has received much attention in the past. See \cite{Alexander93, Kesten93, Alexander97, AZ13, ADH15} for progresses on bounds for rate of convergence for sub-additive functionals with prominent application in first-passage percolation.  In particular, a general theory was given in \cite{Alexander97} via the ingenious convex hull approximation property which applies to several processes on lattices including first-passage percolation. For instance, it was shown that in first-passage percolation the expected length of the shortest path connecting $0$ and $x$  can be approximated by a function of $x$ that is convex and homogeneous of order 1, with approximate error upper bounded by $O(|x|^\nu)$ for $\nu<1$.
Our proof in Section~\ref{section:Convergence rate} follows the framework developed in \cite{Alexander97} and is dedicated to verifying the convexity hull condition in \cite{Alexander97} for our log-weighted Green's functions defined as in \eqref{eq:phi-def} --- an incorrect but heuristically useful interpretation of  log-weighted Green's function is the logarithmic of the probability for the random walk to travel from one point to another point without hitting an obstacle (in fact, this is simply the logarithmic of the Green's function without reweighting, which converge to Lyapunov exponents \cite{Sznitman98}). While this resembles the first-passage percolation problem (as already noted in \cite{Sznitman98}), our context is more complicated since our function in a vague sense takes average over many (not necessarily self-avoiding) paths rather than takes the length of the single shortest path as in first-passage percolation. Furthermore, the real definition of log-weighted Green's function is even more complicated: for instance, it has to take into account  the travel time for the random walk as well as to incorporate the requirement that the random walk has to avoid certain regions. These incur substantial challenges in implementing the proof framework in \cite{Alexander97}.

\medskip

\noindent{\bf Faber--Krahn inequality.}  
 A classic result, known as the celebrated Faber--Krahn inequality, states that among sets with given volume balls are the only sets which minimize the first eigenvalue (we remark that Faber--Krahn inequality is the fundamental reason behind the phenomenon that the localization occurs in a ball). Various versions of quantitative Faber--Krahn inequality have been proved in the past \cite{HN94, Melas92, Bha01, FMP09, BDV15}. 
In particular, the  following sharp quantitative Faber--Krahn inequality was proved in \cite[Main Theorem]{BDV15} (here ``sharp'' means that the lower bound is achievable (up to constant) for some choice of $\Omega$; recall \eqref{eq-def-first-eigenvalue} for the definition of $\mu_\Omega$)
\begin{equation}
\label{eq:FB}
	|\Omega|^{2/d}\mu_\Omega - |B|^{2/d} \mu_B \geq \sigma_d \mathcal (\mathcal A(\Omega))^2 \mbox{ for a constant } \sigma_d >0 \mbox{ depending only on } d\,,
\end{equation}
where $B$ is an Euclidean ball, $|\Omega|$ denotes for the volume of $\Omega$, and $\mathcal A(\Omega)$ is the \textit{Fraenkel asymmetry} defined as (below $\bigtriangleup$ denotes the symmetric difference)
\begin{equation} \label{eq-def-mathcal-A}
\mathcal A(\Omega) = \inf \left\{ \frac{|\Omega \bigtriangleup B|}{|B|}: B \text{ is a ball such that }|B| = |\Omega| \right\}\,.	
\end{equation}
 Our proof that the localization region is asymptotically a ball uses \eqref{eq:FB}. In fact, for the purpose of our proof, we do not need the full power of the sharp inequality as in \eqref{eq:FB} --- the inequalities in \cite{Bha01, FMP09} would suffice. 

We remark that ours is not the first application of Faber--Krahn type of inequality in the study of localization of random walks. For instance: in \cite{Bolthausen94} a key ingredient was a version of this type of inequality in two-dimensions which was proved in the same paper; in \cite{Sznitman97} another quantitative version of Faber--Krahn was proved independently;  in \cite{Povel99} a quantitative version of isoperimetric inequality (related to Faber--Krahn inequality) from \cite{Hall92} was a key ingredient in the proof.

\subsection{Future directions and open problems}\label{sec:open problems}
We say a vertex is open if no obstacle is placed there, and thus each vertex is open with probability $p$. In this paper,  we have chosen $p>p_c(\mathbb Z^d)$ to make sure that with positive probability the open cluster containing the origin is infinite.    A variation of the model is to place obstacles on edges rather than on vertices. In addition, in this paper the obstacles are chosen to be hard, i.e., trapping a random walk with probability 1. One could alternatively consider soft obstacles. That is, every time the random walk hits an obstacle it has a certain fixed probability (which is strictly less than 1) to be killed. Furthermore, one could also consider the continuous analogue, i.e., Brownian motion with Poissonian obstacles as in \cite{Sznitman98}. While we believe our methods useful in these settings, we leave these for future study.  In what follows, we wish to emphasize a number of open problems in our context for which we believe serious efforts are required and deserved. 

\medskip

\noindent{\bf Refined results on localization.} While from Theorem~\ref{onecity} we see that conditioned on survival ever since $o(n)$ steps the simple random walk is localized in a ball of volume poly-logarithmic in $n$, we should not expect that ever since $o(n)$ steps the simple random walk stays within $B_n$ (as defined in Theorem~\ref{ballthm}). The reason is that, due to entropy the random walk will occasionally have excursions of order $\log n$ away from $B_n$.  It would be interesting to get a more refined description on the path localization.

Another direction is to prove a lower bound on the localization. We expect that for any $\P$-measurable set $A_n$ with $\frac{A_n}{d\log_{1/p} n} \leq a<1$ for a fixed constant $a$  and any fixed time $t_n$ one should have that $\Pr(S_{t_n}\in A_n \mid \tau > n)$ is strictly bounded away from 1. In addition, if $\frac{A_n}{d\log_{1/p} n} \to 0$, we expect that $\Pr(S_{t_n}\in A_n \mid \tau > n)\to 0$ as $n\to \infty$.

In addition, it would be very interesting to determine the correct order of the Euclidean distance between the origin and $B_n$.

\medskip

\noindent{\bf Scaling limit of  principal eigenvalues.} In both the present article and \cite{DX17}, we did not obtain precise limiting law of the principal eigenvalue, let alone the order statistics of eigenvalues near the extremum. We note that in \cite{BK16} the authors managed to prove a Poisson convergence for the upper order statistics of eigenvalues when the random potential in the environment has doubly-exponential tails, and this serves as an input in \cite{BKS16} for the proof of localization (as well as other properties such as aging) in this context. One may argue that one advantage of our technique is to prove localization without a full understanding of the extremal principal eigenvalues; but admittedly it is a disadvantage of our work that we did not develop techniques to understand the precise behavior of principal eigenvalues. We think it is a very interesting open question to determine the scaling limit of the extremal process of principal eigenvalues.

\medskip

\noindent {\bf Geometric aspects of $B_n$.} It is interesting to study more geometric aspects of the the set $B_n$ as in Theorem~\ref{ballthm}. While we know that it is asymptotically an Euclidean ball, some fine details are missing. For instance, does $B_n$ contain obstacles inside away from its boundary? Questions of this type has been studied in the annealed case: it was shown in \cite{Bolthausen94}  that ``no obstacle exists in the bulk of the random walk range'' for $d=2$; recently, a similar result for $d\geq 3$ was established in \cite{DFSX18}. 
It remains to be an open question in the quenched case (either continuous or discrete). 

Finally, we would like to mention a challenging open problem, which is on the surface fluctuation of the localization ball. The formulation of the problem in the quenched case requires some thinking (as the random walk does have excursions away from the localization ball which is rather long in comparison of the diameter of the ball), but in the annealed case one can simple ask for the surface fluctuation of the random walk range. At the moment, we do not even have a clear physics picture on what is the surface fluctuation exponent, say in the case of two-dimensions.

\subsection{Organization}
The remaining sections of the paper are organized as follows. In Section~\ref{sec:prelim} we review results from \cite{DX17} and also record a few useful lemmas. In Section \ref{sec:one-city}, we give a proof of Theorem \ref{onecity} assuming a major ingredient as incorporated in Proposition \ref{phi-and-g}.  Section \ref{sec:prop3.3} is devoted to the proof of Proposition \ref{phi-and-g}. In Section \ref{section:Asymptotic Ball}, we prove that there is a ball in the pocket island of cardinality asymptotically $d\log_{1/p} n$ such that the principal eigenvalue of this ball is close to that of the pocket island. Finally, we prove in Section \ref{section:Localization on intermittent island} that the random walk will be localized in the intermittent island and hence complete the proof of Theorem \ref{ballthm}.

\subsection{Notation convention} \label{sec:notation}
For $v \in \Z^d$, we recall that $v$ is open if there is no obstacle placed at $v$. We define the $\ell_2$-norm $|\cdot|$ by $|v| = (\sum_{i=1}^d v_i^2)^{1/2}$, the $\ell_1$-norm $|\cdot|_1$ by $|v|_1 = \sum_{i=1}^d |v_i|$, and  the $\ell_\infty$-norm $|\cdot|_\infty$ by $|v|_\infty = \max_{1\leq i\leq d} |v_i|$ . For $r>0,v\in\Z^d$, we define $B_r(v) = \{x \in \Z^d:|x-v|\leq r\}$ and $K_r(v) = \{x \in \Z^d : |x-v|_\infty \leq r\}$.  For $A \subseteq \Z^d$, $|A|$ denotes the cardinality of $A$. Write $\partial A = \{x\in A^c: y\sim x \mbox{ for some } y\in A\}$, where $x\sim y$ means that $x$ is a neighbor of $y$ (i.e. $|x-y|_1 = 1$) and $\partial_i A = \{x\in A: y\sim x \mbox{ for some } y\in A^c\}$. We let $\lambda_{A}$ denote the principal eigenvalue (i.e., the largest eigenvalue)  of $P|_{A}$, which is the transition matrix of simple random walk on $\Z^d$ killed upon exiting $A$. For Lebesgue measurable set $A$ in $\R^d$,  we use $|A|$ to denote the Lebesgue measure of $A$. For $x, y\in A\subseteq \mathbb Z^d$, we define $D_A(x, y)$ to be the length of the shortest path which stays within $A$ and joins $x$ and $y$.

For $x\in \mathbb Z^d$, we denote by $\Pr^x$ and $\Ex^x$ for probability and expectation for random walks with respect to starting point $x$. When $x$ is omitted from the superscript, by default the random walk starts from the origin (This convention has been used already earlier in the introduction).

We denote by $\ob$ the collection of all obstacles (some times referred as closed vertices). For $v\in \Z^d$, we denote by $\calC(v)$ the open cluster containing $v$ and by $\calC(\infty)$ the infinite open cluster. If $v$ is closed, then $\calC(v) = \varnothing$. We denote by $\xi_A = \inf\{t \geq 0: S_t \not\in A\}$ the first time for the random walk to exit from $A$, and by $\tau_A = \inf\{t \geq 0: S_t\in A\}$ the hitting time to $A$. In particular, we denote $\tau_x = \tau_{\{x\}}$ for $x \in \Z^d$. As having appeared earlier, we let $\tau = \tau_{\ob}$  be the survival time of the random walk. For a subset of non-negative integers $I$, we denote $S_I = \{S_t: t\in I\}$.


Throughout the rest of the paper, $C, c$ denote  positive constants depending only on $(d, p)$ whose numerical values may vary from line to line (and we do not introduce them anymore). We have in mind that $C$ is a large constant while $c$ is a small constant. For constants with decorations such as $c_*$, $C_1$ or $\kappa$ (which also depend only $(d, p)$), their values will stay the same in the whole paper. 
\begin{center}
\begin{tabular}{ llllll } 
 \hline
 $\mu_\Omega$&  \eqref{eq-def-first-eigenvalue}&
 $\rad$   & \eqref{eq:def-r}&
 $k_n,R,\calC_R(v)$& \textsc{Sec.} \ref{sec:prelim}\\
  $\D_*,\lambda(v),\lambda_*$& \textsc{Sec.} \ref{sec:prelim}& 
  $\locV$ & \textsc{Def.} \ref{locV-def}&
 $\ft_v,\locE_v,\sd_\lambda$ & \textsc{Def.} \ref{def-Ev} \\
 $\varphi_\star$&\textsc{Def.} \ref{def-phistar}&
 $g$& \eqref{eq:phi-g-g-g}\eqref{eq:g-def-lim}&
 $v_*$& \eqref{eq:variational}\\
 $\fregion$& \eqref{eq:def-fregion}&
 $G_A(x,y;\lambda)$& \textsc{Def.} \ref{def-green}& 
 $r,\cregion,\varphi,\varphi_*$ & \textsc{Def.} \ref{def-phi}\\
 $L,\mabox_{\mathbf{i}}$& \textsc{Def.} \ref{def-box}&
white/black, $\mac_\mathbf{i}$ & \textsc{Def.} \ref{de-kwhite} & 
&\\
  \multicolumn{3}{l}{tilde-white/tilde-black $\mathbf{i}_x,\calC_{\mabox},\mabox_\mathbf{A},\mac_\mathbf{A}$}&\multicolumn{3}{l}{around \eqref{eq:def-K-A}}\\
$\bba(x,r)$& \textsc{Lem.} \ref{cutset}  &
$\bar \varphi,\varphi^\circ,\varphi^\circ_*,\bar \varphi^\circ_*,R_{\circ}$& \textsc{Def.} \ref{def-phic}& 
$g_x,Q_x$& \textsc{Def.} \ref{def-gx}\\
$h$ &\eqref{eq:def-h} &
$s_x,G_x,\Delta_x,D_x$ &\eqref{eq:def-alex}&
$\gpt$ &\textsc{Def.} \ref{def-gpt} \\
$(\epl r,\rho)$-empty$,\emp$& \textsc{Def.} \ref{def-smallbox}&
$f, \Omega_\epsilon$ & \textsc{Def.} \ref{def-eign-f} &
$\Omega_\epsilon^+$ &\eqref{eq:def-Omega+}\\
$B_\epsilon$ & \textsc{Lem.} \ref{Asymptotic shape - Ball}& 
$\tOmega, \hbn$&\eqref{eq:def-hbn} &
 & \\
\hline
\end{tabular}
\end{center}

\smallskip

\noindent {\bf Acknowledgment.} We warmly thank Ryoki Fukushima, Rongfeng Sun and  Alain-Sol Sznitman for many helpful discussions.

\section{Preliminaries}\label{sec:prelim}
As described in Section~\ref{sec-poly-to-sharp}, it was proved in \cite{DX17} that the random walk will be localized in poly-logarithmic in $n$ many balls of radius $(\log n)^\kappa$ (see Theorem \ref{Polylog thm}), which we refer to as pocket islands (see Lemma~\ref{locV-def} and the discussions that follows for a more formal definition for pocket islands). In this subsection, we will describe the main result of \cite{DX17} in more detail and record a number of useful lemmas. 

\noindent\textbf{Pocket Islands.} Let us first recall some notations and definitions from \cite{DX17}. 
Let $k_n = (\log n)^{4 - 2/d } (\log \log n)^{2\11_{d=2}}$. Write  $R = k_n (\log n)^2$ and denote by $\mathcal C_R(v)$ the connected component in $B_{R}(v) \backslash \ob$ that contains $v$. Let $\lambda(v)  = \lambda_{\mathcal C_R(v)}$ be the principal eigenvalue of the transition matrix $P|_{\mathcal C_R(v)}$--- we note that $(1-\lambda(v))$ is the discrete analogue the first eigenvalue of Dirichlet-Laplacian on $\mathcal C_R(v)$ defined in \eqref{eq-def-first-eigenvalue}. We call for the attention of the reader that the notation of $\lambda(v)$ and $\lambda_{\{v\}}$ have complete different meanings. 

Set \begin{equation}
\label{eq:lambda-*-def}
	\lambda_* = p_{\alpha_1}^{1/k_n}
\end{equation} where $p_{\alpha_1}$ (defined in \cite[(3.1)]{DX17}) is appropriately chosen according to some large quantile of the distribution of survival probability up to $k_n$ steps.  Denote $\D_* = \{ v\in \calC(0):\lambda(v) \geq \lambda_*\}$. We have that (\cite[Corollary 3.7]{DX17})
\begin{equation}
\label{eq:lambda-*-property}
  k_n^{2d} n^{-d} \leq \P(v \in \D_{*}) \leq k_n^{8d} n^{-d}\,.
\end{equation}
Note that the events $
\{v \in \D_{*}\}$ for $v \in \Z^d$ are rare (c.f.\ \eqref{eq:lambda-*-property}) and are only locally dependent. Thus, the set $\D_{*}$ can be divided into many isolated islands as incorporated in the next lemma. 

\begin{lemma}
\emph{(\cite[Lemma 3.8]{DX17})}~\label{locV-def}
For every constant $C_0>0$, with $\P$-probability tending to one, there exists a skeletal set $\locV \subseteq \D_* \cap \calC(0) \cap B_{\cucc n(\log n)^{-2/d}}(0)$ such that 
	\begin{equation}
\label{eq:locV-def}
  \begin{split}
  &\lambda(v) = \max \{\lambda(u):{u \in B_{3R}(v)} \}, \quad \D_{*} \cap \calC(0) \cap B_{\cucc n(\log n)^{-2/d}}(0) \subseteq \cup_{v \in \locV}B_{3R}(v);\\
  &B_{nk_n^{-100d}}(v)\text{ for }v \in \locV \cup \{ 0\}\text{ are disjoint}.
\end{split}
\end{equation}
\end{lemma}
We will fix the values of $\kappa,C_0$ in Theorem \ref{Polylog thm}. And the balls $B_{(\log n)^\kappa}(v)$ for $v \in \locV$ will be referred to as \emph{pocket islands}.

\noindent \textbf{Path Localization.} The following path localization has been proved in \cite{DX17}. Conditioned on survival, the random walk will  travel to one of the pocket islands (which we refer to as the target island) and then it will be confined in the target island afterwards. In addition, the random walk will avoid getting close to any region that is better than or almost as good as (i.e., has larger or nearly the same principal eigenvalue) the target island, and the random walk will reach the target island at a time at most linear in the distance between the target island and the origin. We next give a more formal statement (see Theorem~\ref{Polylog thm}) on the path localization.

\begin{defn}
\label{def-Ev}
For constant $\kappa,C_1>0$ to be determined and each $v \in\Z^d$, we define hitting time of the ball of radius $(\log n)^{\kappa/2}$ centered at $v$
$$ \ft_v = \tau_{B_{(\log n)^{\kappa/2}}(v)}\,,$$
and event
\begin{equation*}
\mathcal \locE_v = \left\{\tau_{\sd_{\lambda(v) }}>\ft_v, S_{[\ft_v,n]} \subseteq B_{(\log n)^{\kappa}}(v) , \tau > n 	\right\}\,,
\end{equation*}
where $\sd_\lambda = \{v\in \Z^d: \Pr^v(\tau > (\log n)^{C_1}) > [(1 - k_n^{-21d})\lambda]^{ (\log n)^{C_1}}\}$.	
\end{defn}

\begin{thm}[\cite{DX17}]
\label{Polylog thm}
For sufficiently large constants $\kappa,C_0,C_1$, with $\P$-probability tending to one,
\begin{equation}
\label{eq:poly-loc}
  \Pr(\cup_{v \in \locV}(\locE_v \cap \{\ft_v \leq C|S_{\ft_v}|\} )\mid  \tau > n) \geq 1 - e^{-(\log n)^c}\,.
\end{equation}
\end{thm}
\begin{proof}
It can be found in the proof of \cite[Proposition 4.3 and (5.19)]{DX17} that for sufficiently large $\kappa,C_0$, there is a random site $x_n \in \locV$ (depending on  $\ob$ and potentially also depending on $(S_t)_{t=0}^n$) such that the following hold accordingly with $\P$-probability tending to one,
$$ \Pr(S_{[\ft_{x_n},n]} \subseteq B_{(\log n)^{\kappa}}(x_n) \mid  \tau > n) \leq e^{-(\log n)^c} \mbox{ and } \Pr(\ft_{x_n} \leq C|S_{\ft_{x_n}}|  \mid  \tau > n) \leq  e^{-(\log n)^c}\,.$$
For the same $x_n$, \cite[Lemma 5.6]{DX17} states that for any constant $q$,
\begin{equation*}
	\Pr(\lambda(x) \leq (1 - k_n^{-20d})\lambda(x_n),\forall x \in \cup_{0 \leq t \leq \ft_{x_n}} B_{(\log n)^q}(S_t)  \mid  \tau > n) \leq e^{-n/(\log n)^C}\,.
\end{equation*}
At the same time, \cite[Lemma 4.5]{DX17} yields if $\lambda(x) \leq (1 - k_n^{-20d})\lambda, \forall x\in B_{(\log n)^q}(y)$, then
\begin{equation*}
	\Pr^y(\tau > (\log n)^{C_1}) \leq (\log n)^C [ (1 - k_n^{-20d})\lambda]^{(\log n)^{C_1}} \leq [(1 - k_n^{-21d})\lambda]^{ (\log n)^{C_1}}\,,
\end{equation*}
where the last inequality holds if $C_1$ is sufficiently large. Combining above three results gives \eqref{eq:poly-loc}.
\end{proof}
We will choose $\kappa,C_0,C_1$ sufficiently large as in this theorem and will assume $\kappa \geq 10C_2 + C_1 + 10$ where $C_2$ is a constant to be selected in Lemma \ref{chemdist2}. 

\smallskip

\noindent {\bf A few useful lemmas.} We next record a few lemmas for later use.
\begin{lemma}
\label{FiniteClusterSize}
 \emph{(\cite[Theorem 3]{CCN87} and \cite{GM90b} (see also \cite[Corollary 3]{KZ90}))} For standard (Bernoulli) site percolation on $\Z^d$ with parameter $p > p_c(\Z^d)$,
 \begin{equation}
 \label{eq:FiniteClusterSize}
   \P(|\calC(0)| = m) \leq e^{-cm^{(d-1)/d}}.
 \end{equation}
\end{lemma}
\begin{lemma}
\label{lambda*}$
 	\lambda_{*} \geq 1-c_*(\log n)^{-2/d} -C(\log n)^{-3/d}$.
\end{lemma}
\begin{proof}
Recall that $\rad =  \lfloor (\omega_d^{-1}d\log_{1/p} n)^{1/d} \rfloor  $ as defined in \eqref{eq:def-r}. Let $r = \rad -q$. By \cite[(24)]{Kuttler70} and scaling we see that for any fixed $q>0$,
$$ 1 - \lambda_{B_r(0)} \leq c_*(\log n)^{-2/d}  + O((\log n)^{-3/d})\,.$$
In addition, $|B_r(v)| \leq d\log_p (n^{-1}) - \rad^{d-1}$ for sufficently large $q$, hence for all $v\in \Z^d$, $$\P(\lambda(v) \geq \lambda_{B_r(v)}) \geq \P(B_r(v) \text{ is open}) \geq  n^{-d} e^{c \rad^{d-1}}\,.$$
Fix a large $q$ and compare it with \eqref{eq:lambda-*-property}. We get $\lambda_* \geq 1 - c_*(\log n)^{-2/d} -C(\log n)^{-3/d}.$
\end{proof}

\begin{lemma}
	For $\lambda \geq \lambda_*$ and any $v \in \Z^d, t>0$
	\begin{equation}
		\label{eq:Prob-Dlambda}
		\P( v \in \sd_\lambda) \leq (\log n)^{C}n^{-d}\,,
	\end{equation}
	\begin{equation}
	\label{eq:SurvivalProb-SD}
		\Pr^v( \tau_{\sd_\lambda \cup \ob}>t) \leq (\log n)^{C}(1 - (\log n)^{-{100d}})^t \lambda^t\,.
	\end{equation}
\end{lemma}
\begin{proof}
If $x \in \sd_\lambda$, then there exists $y \in B_{(\log n)^{C_1}}(x)$ such that $\Pr^y(\tau > k_n) \geq [(1 - k_n^{-21d})\lambda_*]^{k_n+1} \geq \lambda_*^{k_n}/2$. Hence \cite[Lemma 3.3]{DX17} gives \eqref{eq:Prob-Dlambda}. The bound \eqref{eq:SurvivalProb-SD} is an analogue of \cite[Lemma 4.4]{DX17}. The adaption of the proof is straightforward: we can adapt the proof by just changing all occurrences of $k_n,R,\mathcal U_\alpha,p_\alpha$ to $(\log n)^{C_1},(\log n)^{C_1},\D_\lambda,\lambda^{(\log n)^{C_1}}$ respectively and noting that $k_n^{-21d}>(\log n)^{-100d}.$
\end{proof}

\section{One city theorem}\label{sec:one-city}
\subsection{Overview}
In this section, we will give the proof of Theorem~\ref{onecity}.  We first give a heuristic description. There are poly-logarithmic many pocket islands (see \eqref{eq:locV-def},\eqref{eq:poly-loc}). For each of them the probability for localizing in that island is \emph{roughly speaking} the product of the probability of reaching the island (which we refer to as searching probability) and the probability of staying in that island afterwards. We use the fluctuation of the searching probability to show that only a single one of them will be dominating. Below are the key ingredients for demonstrating that the fluctuation of searching probability is large:
\begin{itemize}
\item We expect that the searching probability to a far away vertex $v$ (this is close to the searching probability to a neighborhood around $v$) is exponentially small in $|v|$, where the rate of decay may depend on the direction $\frac{v}{|v|}$. 
\item The locations of these pocket islands are roughly independent and uniform in $B_n(0)$. 
\end{itemize}
In fact, we can have a quantitative version for the first ingredient which controls the rate of convergence for the logarithmic of the searching probability. To prove this, we can adapt methods discussed in Section~\ref{sec:ingredients} on the rate of convergence in first-passage percolation (in particular the method in \cite{Alexander97}). 

However, our situation is  more complicated as we have to keep track of the time spent on reaching an island. This is because, when we require the random walk to stay in the island after reaching it, the remaining amount of time is not fixed but depends on how much time the random walk has already spent on reaching the island. This motivates the following definition. 
\begin{defn}
\label{def-phistar}
	Recall that $ \ft_v = \tau_{B_{(\log n)^{\kappa/2}}(v)}$. For $\lambda > 0$ we define
\begin{equation}
	\varphi_\star(0,v;\lambda) =- \log \Ex \big[\lambda^{-\ft_v}  \11_{n \wedge \tau_{\sd_{\lambda }\cup \ob }\geq \ft_v }\big]\,.
\end{equation}
\end{defn}

We wish to make a couple of remarks on our definition of $\varphi_\star$. 
\begin{itemize}
\item We have a term of $ \lambda^{-\ft_{y}}$ in the definition. This is because after reaching the island, every step of survival costs roughly a probability of $\lambda$ (assuming the principal eigenvalue of the target island is $\lambda$), and thus for every step spent on reaching the island we give a reward of $\lambda^{-1}>1$ to account for the saving on future probability cost.
\item We do not allow the random walk to enter $\sd_\lambda$. Otherwise, the random walk may stay in a region of eigenvalue grater than $\lambda$ for excessively large  amount of time and lead to an excessively small value of  $\varphi_\star(x,y;\lambda)$ (since for every step the random walk gains a prize of $\lambda^{-1}$), and thus fails to serve its intended purpose. 
\end{itemize}

Next, we list three ingredients for the proof of Theorem~\ref{onecity}: Lemma~\ref{ev-phi} expresses $\Pr(\locE_v)$ as a combination of $\lambda(v)$ and $\varphi_\star(0,v;\lambda(v))$; in Proposition~\ref{phi-and-g} we approximate $\varphi_\star$ by a linear function $g$;  Lemma~\ref{topev-2} encapsulates our basic intuition that  the fluctuation of $g(v;\lambda(v))$ should be large since sites in $\locV$ are roughly uniform distributed in $B_n(0)$.

\begin{lemma}
\label{ev-phi}
With $\P$-probability tending to one,
	\begin{equation}
	\log \Pr(\locE_v) = n\log \lambda(v) - \varphi_\star(0,v;\lambda(v))  +O( (\log n)^{C})\,.
\end{equation}
\end{lemma}

\begin{proposition}
\label{phi-and-g}
There exists a deterministic nonnegative function $g$ such that for some constants $\cul,\Cbar$ only depending on $(d,p)$ and all $\lambda \in [\lambda_*,1]$, 
\begin{equation}
\label{eq:phi-g-g-g}
	g( \, \cdot \,  ;\lambda)	\text{ is convex, homogeneous of order 1 and }\cul|x|\leq g(x;\lambda) \leq  \Cbar|x|\,.
\end{equation}
Also, conditioned on $0 \in \calC(\infty)$, with $\P$-probability tending to one, for all $v \in \locV$
\begin{equation}
	\varphi_\star(0,v;\lambda(v)) = g(v;\lambda(v)) + O(n^{5/6})\,.
\end{equation}
\end{proposition}
\begin{remark}
For our purpose, we are interested in the case when $\lambda < 1$ since $\lambda$ will be the principal eigenvalue of a transition matrix of random walk with killing. Similar results has been proved when $\lambda \geq 1$ in \cite{Sznitman96b}, where one considers Brownian motion with Poissonian obstacles and the corresponding $g(\, \cdot \, ;\lambda)$ will be proportional to Euclidean distance. We note that our case when $\lambda<1$ is substantially more challenging to analyze (since for instance, we have to forbid the random walk to enter $\sd_\lambda$ which incurs a number of complications in the proof).
\end{remark}


\begin{lemma}
\label{topev-2}
With $\P$-probability tending to one, for any $u,v \in \locV$,
	$$ | (n\log \lambda(v) - g(v;\lambda(v))) - (n\log \lambda(u) - g(u;\lambda(u)))| \geq n (\log n )^{-C}\,.$$
\end{lemma}
Let $v_*$ be the maximizer of the following variational problem
\begin{equation}
\label{eq:variational}
  \max_{v \in \locV} n\log \lambda(v) - g(v;\lambda(v))\,.
\end{equation}
Combining preceding three results, we will show in \emph{Proof of Theorem \ref{onecity}}  that conditioned on survival, with $\P$-probability tending to one the random walk travels  to the pocket island around  $v_*$ and stays there afterwards.
We will call the pocket island around $v_*$ the \emph{optimal pocket island}.

\begin{remark}
By Theorem~\ref{onecity}, for a typical environment, with $\Pr$-probability tending to 1 the target island as described in Section~\ref{sec:prelim} coincides with the optimal pocket island.
\end{remark}

\subsection{Proof of Lemmas \ref{ev-phi}, \ref{topev-2} and Theorem \ref{onecity}}

In this subsection, we prove Lemmas \ref{ev-phi}, \ref{topev-2} and then prove Theorem \ref{onecity} by combining  Lemmas \ref{ev-phi}, \ref{topev-2} and Proposition~\ref{phi-and-g}. The proof of Proposition~\ref{phi-and-g} is postponed to Section~\ref{sec:prop3.3} and occupies the entire section.

\begin{proof}[Proof of Lemma \ref{ev-phi}]
	By strong Markov Property, we get that 
\begin{equation}
\label{eq:evstep1-1}
	\Pr(\locE_v)= \Ex \left[ \11_{\tau_{\sd_{\lambda(v) }\cup \ob}>\ft_v}\Pr^{S_{\ft_v}}(\xi_{B_{(\log n)^{\kappa}}(v) \setminus \ob} > n - \ft_v)\right]\,.
\end{equation}
Note that for $m \in (0,n)$, by \eqref{eq:locV-def} and \cite[Lemma 4.5]{DX17},  
\begin{align*}
	\Pr^{S_{\ft_v}}(\xi_{B_{(\log n)^{\kappa}}(v) \setminus \ob} > m) \leq \Pr^{S_{\ft_v}}(\lambda(x) \leq \lambda(v) \text{ for } x \in S_{[0,m]},\tau > m) \leq(\log n)^{C} \lambda(v)^{m}\,.
\end{align*}
At the same time, since $v, S_{\ft_v} \in \calC(0)$ and $|S_{\ft_v} - v| \leq (\log n)^{\kappa/2}$, by \cite[Theorem 1.1]{AP96} (or \cite[(3.8)]{DX17}) we see that any point in $\calC_R(v)$ (defined in Section \ref{sec:prelim}) and $S_{\ft_v}$ can be connected by an open path of length at most $C(\log n)^{\kappa/2}$ (Thus the open path is also inside $B_{(\log n)^\kappa}(v)$). Combined with $\max_x\Pr^{x}(\xi_{A} > t) \geq \lambda_{A}^t$ (see \cite[Lemma 3.2]{DX17}), it gives that for $m \in (0,n)$
\begin{equation}
\label{eq:kappa-connect-v}
	\Pr^{S_{\ft_v}}(\xi_{B_{(\log n)^{\kappa}}(v) \setminus \ob}  > m) \geq e^{-(\log n)^{C}} \lambda(v)^{m}\,.
\end{equation}
Combining preceding three displays completes the proof of the lemma.
\end{proof}

\label{section:one city proof}

\begin{proof}[Proof of Lemma \ref{topev-2}]It suffices to prove the following:
with $\P$-probability tending to one, for any $u,v \in \D_{*}$ such that $|u - v| \geq (\log n)^6$, 
	$$ | (n\log \lambda(v) - g(v;\lambda(v))) - (n\log \lambda(u) - g(u;\lambda(u)))| \geq n (\log n )^{-C}\,.$$
Now we verify this statement. Let $A$ be the set of all the possible values for the random variable $\log \lambda(v)$ that are greater than or equal to $ \log \lambda_*$ and $C'>0$ be a large constant to be selected.
We have
\begin{align*}
	\sum_{\stackrel{u,v \in B_n(0)}{|u-v| \geq (\log n)^4}}&\P( | (n\log \lambda(v) + g(v;\lambda(v))) - (n\log \lambda(u) + g(u;\lambda(u)))| \leq n (\log n )^{-C'}, u,v \in \D_{*}) \\&= \sum_{u,v \in B_n(0),|u-v| \geq (\log n)^4}\sum_{a_1,a_2 \in A}\P(\log \lambda(v) = a_1, \log \lambda(u) = a_2) I_{a_1, a_2}\,,
\end{align*}
where $I_{a_1, a_2} = \11_{| (na_1 + g(v;e^{a_1})) - (na_2 + g(u;e^{a_2}))| \leq n (\log n )^{-C'}}$.
Noting that \eqref{eq:phi-g-g-g} holds for all $\lambda \in [\lambda_*,1]$, we let
\begin{equation}
	\mathcal H = \{ h: \R^d \to \R^+:\cul|v| \leq h(v) \leq \Cbar|v| , h \text{ convex, homogeneous of degree }1 \}.
\end{equation}
Then for all $h \in \mathcal H,x>0$, $h^{-1}(x)$ is a convex set in $B_{x/\cul}(0)$ and
\begin{align*}
	&\quad \sup_{a_1 \in A}\sup_{x \in \R}|\{v \in B_n(0): |g(v;e^{a_1}) -x | \leq n (\log n )^{-C'}\}|\\
	&\leq \sup_{h \in \mathcal H}\sup_{x \in \R}|\{v \in B_n(0): |h(v) -x | \leq n (\log n )^{-C'}\}|\\
  & =  \sup_{h \in \mathcal H}\sup_{x \in \R} | \{y \in \R^d : \exists \ v \in B_n(0) \ \mbox{ such that } \  |h(v) -x| \leq n (\log n )^{-C'}, |v-y|_\infty \leq 1/2 \} | \,,
   \end{align*}
  where in the last expression the outmost $|\cdot|$ stands for Lebesgue measure instead of cardinality. Note that $ |v-y|_\infty \leq 1/2$ implies $|h(v) - h(y)| \leq \Cbar d^{1/2}$ and $h(v) \leq \Cbar n$ for $v \in B_n(0).$ We thus obtain that
    \begin{align*}
  &\quad \sup_{a_1 \in A}\sup_{x \in \R}|\{v \in B_n(0): |g(v;e^{a_1}) -x | \leq n (\log n )^{-C'}\}|\\
	& \leq \sup_{h \in \mathcal H}\sup_{x \leq 2\Cbar n}\left|h^{-1}([x - n (\log n )^{-C'}-\Cbar d^{1/2},x + n (\log n )^{-C'} + \Cbar d^{1/2}]) \right|\\
	&\leq C \cul^{-1}(2\Cbar n)^{d-1} \cdot   (\cul^{-1}n (\log n )^{-C'} + \cul^{-1}2\Cbar d^{1/2})\leq Cn^{d}(\log n)^{-C'}\,,
\end{align*}
where the second inequality follows from the fact that the surface area of any convex set contained in a ball is less than the surface area of that ball (see, e.g., \cite[Page 48--50]{GW93}).
Since $\lambda(v)$ and $\lambda(u)$ are independent if $|u - v| \geq (\log n)^6$, we conclude that
\begin{align*}
\sum_{\stackrel{u,v \in B_n(0)}{|u-v| \geq (\log n)^4}}\sum_{a_1,a_2 \in A}&\P(\log \lambda(v) = a_1, \log \lambda(u) = a_2)\cdot I_{a_1, a_2}
	\leq  2^dn^{2d}(\log n)^{-C'}(\P(v \in \D_{*}))^2\,.
\end{align*}
Now, the results follows from \eqref{eq:lambda-*-property} and choosing $C'$ such that $(\log n)^{C'} \geq k_n^{20d}$. 
\end{proof}

\begin{proof}[Proof of Theorem \ref{onecity}]
Recall that $v_*$ is the maximizer of the function $ v \mapsto n\log \lambda(v) - g(v;\lambda(v))$ (c.f. \eqref{eq:variational}). Then combining Lemmas \ref{ev-phi}, \ref{topev-2} and Proposition \ref{phi-and-g}, we get that with $\P$-probability tending to one
$$\log \Pr(\locE_{v_*}) - \log\Pr(\locE_u) \geq 2^{-1}n (\log n )^{-C}\quad \forall u \in \locV \setminus \{{v_*}\}\,.$$
Hence, on the preceding event $ \Pr(\locE_{v_*})/\sum_{u \in \locV} \Pr(\locE_u) \geq 1 - e^{-n(\log n)^{-C}}$. Combined with \eqref{eq:poly-loc}, it follows that
$\Pr(\locE_{v_*} \mid \tau >n) \geq 1 - e^{-(\log n)^c}$.
Define
\begin{equation}
\label{eq:def-fregion}
	\fregion \text{ :=the connected component in $B_{(\log n)^\kappa}(v_*)$ that contains $v_*$}.
\end{equation}
And recall that $T = \ft_{v_*}$. On event $E_{v_*}$, the random walk stays in $B_{(\log n)^\kappa}(v_*)$ during $[T,n]$. The discussion before \eqref{eq:kappa-connect-v} yields $S_T \in \fregion$. Hence
   \begin{equation}
   \label{eq:INU}
\Pr\Big(T \leq C |v_*|, \{S_t: t \in [T,n]\} \subseteq \fregion \mid \tau > n\Big) \leq e^{-(\log n)^c}\,.\\
   \end{equation} 
This completes the proof of the theorem.
\end{proof}
\section{Approximation and concentration for $\varphi_\star$-function}
\label{sec:prop3.3}

This entire section is devoted to the proof of Proposition \ref{phi-and-g}. 
To this end, we first introduce a couple of definitions.
\begin{defn}
\label{def-green}
	For $\lambda>0,$ $A\subseteq \Z^d$ and $x,y \in \Z^d$,we define 
\begin{equation*}
  G_A(x,y;\lambda) = \sum_{n = 0}^\infty \lambda^{-n} \Pr^x(S_n = y,S_{[1,n-1]} \subseteq A)\,.
\end{equation*}
\end{defn}

Note that in the preceding definition, we have chosen $\lambda^{-n}$ as opposed to the more conventional $\lambda^n$ so that it will be consistent with the definition below.

\begin{defn}
\label{def-phi}
	We set
$r(x,y) = (\log|x-y|)^{2\kappa + 10d}$,$\cregion = (\sd_\lambda \cup \ob)^c$ and define log-weighted Green's functions (LWGF)
\begin{equation}
\label{eq:phi-def}
\begin{split}
    &\varphi(x,y;\lambda) = -\log \left(\Ex^{x} \left[ \lambda^{-\tau_{y}} \11_{\tau_{\sd_{\lambda} \cup \ob}>\tau_{y}} \right]  \right) \\
  &\varphi_*(x,y;\lambda) = \min_{x' \in B_{r(x,y)}(x)}\min_{y' \in B_{r(x,y)}(y)} \varphi(x',y';\lambda)\,.
\end{split}
\end{equation}
\end{defn}
Note that the name of log-weighted Green's functions came from the factor of $\lambda^{-\tau_y}$ in the preceding definition. Also,
\begin{equation}
\label{eq:phi=logG}
  \varphi(x,y;\lambda) = - \log G_{\cregion\setminus \{y\}}(x,y;\lambda) = - \log \big(G_{\cregion}(x,y;\lambda)/ G_{\cregion}(y,y;\lambda)\big)\,.
\end{equation}
The next result justifies the approximation of $\varphi_\star$ by the log-weighted Green's functions, whose proof can be found in Section~\ref{sec:Approximations of LWGF}.
\begin{lemma}
\label{ind-appr}
For all $v \in \Z^d$ with $|v| \in (n^{2/3},\cucc n(\log n)^{-2/d})$,
  \begin{equation*}
  \P\Big(\max_{\lambda \in [\lambda_*,1]}|\varphi_*(0,v;\lambda)- \varphi_\star(0,v;\lambda)  | \leq (\log n)^{C} \mid G_0\Big) \geq 1 - e^{-c(\log n)^2}\,.
\end{equation*}
where \begin{equation}
\label{eq:def-G-0}
	G_0 := \{\calC(0)\cap B^c_{n^{1/2}}(0) \not = \varnothing, \sd_{\lambda_*} \cap B_{n^{1/2}}(0) = \varnothing \}\,.
\end{equation}
\end{lemma}
Here the event $G_0$ is used as an approximation of  $\{0 \in \calC(\infty)\}$ (see Lemma \ref{FiniteClusterSize}, \eqref{eq:Prob-Dlambda}), and in \emph{Proof of Proposition} \ref{phi-and-g} we will take advantage of the fact that $G_0$ only depends on the local environment of $0$.
We also wish to make a couple of remarks on Definition~\ref{def-phi}:
\begin{itemize}
\item We first approximate $\varphi_\star$ by $\varphi$ where we replace $\ft_y$ by $\tau_y$ --- this is useful since later we will apply sub-additive arguments and it would be convenient to get rid of reference to $n$ in the definition (recall that  $\ft_y = \tau_{B_{(\log n)^{\kappa/2}}(y)}$). In addition, we do not restrict $\tau_y <n$, which will be justified in Lemma \ref{Linear}.
\item We further approximate $\varphi$ by $\varphi_*$ by allowing to minimize over the starting and ending points in some local ball around $x$ and $y$, for the purpose of getting around the complication when $x$ and $y$ are disconnected by obstacles (which occurs with positive $\P$-probability).
\end{itemize}
In addition, we note that in later subsections we will introduce more approximations of LWGFs to facilitate our analysis.

The rest of this section is organized as follows: In Section~\ref{section:Renormalization} we apply renormalization techniques to control the chemical distances on the cluster $(\sd_\lambda \cup \ob)^c$ as well as some refined geometric properties (c.f.\ Lemma \ref{cutset}) for later use. In Section \ref{sec:Approximations of LWGF}, we justify the approximation of $\varphi_*$ by the LWGF and prove a few technical lemmas about LWGF.
In Section \ref{section: Concentration, Sub-additivity}, we prove a concentration inequality for $\varphi_*$ and sub-additivity of $\E \varphi_*$. Then in Section \ref{section:Convergence rate}, we follow the framework developed in \cite{Alexander97} for first-passage percolation to prove that $\E \varphi_*$  is approximated by $g$ with approximation error bounded by $O(|v|^\alpha)$ for some $\alpha<1$. In this step, we have to address a number of challenges that are not seen in the first-passage percolation setup, due to the complication in the definition of our log-weighted Green's function $\varphi_*$. Finally, in Section~\ref{sec:phi-and-g} we prove Proposition \ref{phi-and-g}, by combining the ingredients in previous subsections.

\subsection{Percolation process avoiding high survival probability regions}
\label{section:Renormalization}
In this subsection, we study connectivity properties for the percolation process on $\ob^c\cap \sd_{\lambda_*}^c$ where $\lambda_*$ is defined in \eqref{eq:lambda-*-def} --- this will be useful later when analyzing LWGFs. In order to analyze this percolation process with (short-range) correlations, we employ the standard renormalization technique in percolation theory, and reduce it to the analysis of a certain independent percolation process (see Lemmas~\ref{kiprobability} and \ref{StochasticalDomination}, and discussions following Lemma~\ref{StochasticalDomination}). 

\begin{defn}
\label{def-box}
	Recall that $K_r(v) = \{x \in \Z^d : |x-v|_\infty \leq r\}$. Let $L = \lfloor \log n \rfloor^{2d}$ and consider disjoint boxes
\begin{equation}
	\mabox_\mathbf{i} := K_{L}((2L+1)\mathbf{i}) \for \mathbf{i}\in\Z^d\,.
\end{equation}
We define the renormalized lattice $\{\mabox_\mathbf{i}: \mathbf{i}\in\Z^d \}$ which inherits the graph structure from the bijection $\mabox_\mathbf{i} \mapsto \mathbf{i}$.
\end{defn}

\begin{defn}
\label{de-kwhite}
Let $C_D$ be a positive constant to be selected in Lemma \ref{kiprobability}.
	We say $\mabox_\mathbf{i}$ (or $\mathbf{i}$) is white (otherwise black), if the following hold:

\begin{itemize}
	\item[1.] There exists a unique open connected component $\mac_\mathbf{i}$ in $K_{C_D L}((2L+1)\mathbf{i})$, such that 
	\begin{equation} \label{eq-def-C-i}
		|\mac_\mathbf{i} \cap (\cup_{\mathbf{j} :\|\mathbf{j-i}\|_\infty \leq 1}\mabox_\mathbf{j})| \geq L/10\,.
	\end{equation}
	\item[2.] For all $u,v \in \mac_\mathbf{i} \cap (\cup_{\mathbf j:\|\mathbf{j-i}\|_\infty \leq 1}\mabox_\mathbf{j})$ (recalling definition of $D_\cdot(\cdot, \cdot)$ in Section~\ref{sec:notation}),
	\begin{equation}
		D_{\mac_\mathbf{i}}(u,v) \leq C_D L\,.
	\end{equation}
	\item [3.] $K_{C_D L}((2L+1)\mathbf{i}) \cap \sd_{\lambda_*} = \varnothing$.  In addition, for all $\mathbf j$ satisfying $\|\mathbf j-\mathbf i\|_\infty \leq 1$, one has
	\begin{equation} \label{eq-size-C-i}
		|\mac_\mathbf{i} \cap \mabox_\mathbf{j}| \geq L/10\,.
	\end{equation}
\end{itemize}
\end{defn}
\begin{remark}
The requirements in \eqref{eq-def-C-i} and \eqref{eq-size-C-i} look somewhat odd and repetitive at first glance. We present the conditions in this way since we wish that the component satisfying \eqref{eq-def-C-i} is unique, and in addition this component satisfies \eqref{eq-size-C-i}. This is stronger than the claim that there is a unique connected component satisfying \eqref{eq-size-C-i}.
\end{remark}
The collection of white vertices gives a dependent site percolation process on the renormalized lattice. In the next two lemmas, we will show that the white vertices dominates a supercritical Bernoulli percolation for an appropriate choice of $C_D$ and $n \geq n_0$ where $n_0$ is a fixed large constant.
\begin{lemma} 
\label{kiprobability}
There exists a constant $C_D\geq 4$ such that
for all $\mbfi \in \Z^d$
\begin{itemize}
	\item [(1)]The event  $\{\mabox_\mathbf{i}$ is white$\}$ is independent of $\sigma(\text{ $\{\mabox_\mathbf{j}$ is white}\} \ for \ \mbfj \ s.t.\ |\mbfj-\mbfi| \geq 4C_D)$.
	\item [(2)]$
		\P(\text{$\mabox_\mathbf{i}$ is white}) \geq  1- n^{-1}\,.$
\end{itemize}

\end{lemma}
\begin{proof}
The first item is a direct consequence of the definition. Now, we verify the second item. We first claim that  $\mabox_\mathbf{i}$ is white if all of the following hold:
\begin{itemize}
	\item[(a)] For all $u,v \in \cup_{\mathbf j:\|\mathbf{j-i}\|_\infty \leq 1}\mabox_\mathbf{j}$ such that $u,v$ are in the same open connected component,
	\begin{equation}
		 D_{\ob^c}(u,v) \leq C_D L\,.
	\end{equation}
	\item[(b)] For any $ v \in \cup_{\mathbf j:\|\mathbf{j-i}\|_\infty \leq 1}\mabox_\mathbf{j}$, either $\mathcal C (v) = \mathcal C(\infty)$ or $|\mathcal C (v)| < L/10$.
	\item [(c)]For all $\|\mathbf{j-i}\|_\infty \leq 1$,
	\begin{equation}
		|\mathcal C(\infty) \cap \mabox_\mathbf{j}| \geq L/10\,.
	\end{equation}
	\item [(d)]$K_{C_D L}((2L+1)\mathbf{i}) \cap \sd_{\lambda_*} = \varnothing$.
\end{itemize}
To verify this, we observe that (a) implies that 
\begin{equation}\label{eq-justify-uniqueness}
\mbox{ all vertices in }\mathcal C(\infty) \cap (\cup_{j:\|j-i\|_\infty \leq 1}\mabox_\mathbf{j})\mbox{ are connected in }K_{C_D L}((2L+1)\mathbf{i})\cap \calC(\infty)\,.
\end{equation} 
Combining with (b) and (c), we get Property 1 in Definition ~\ref{de-kwhite} where $\mac_\mathbf{i}$ is the connected component of $\mathcal C(\infty) \cap K_{C_D L}((2L+1)\mathbf{i})$ which has non-empty intersection with $\mabox_\mathbf{i}$ --- such connected component is unique by \eqref{eq-justify-uniqueness}. Combining this with (d) gives Property 3. Property 2 follows from (a).

Now, \cite[Theorems 1.1]{AP96} yields
\begin{equation*}
	\P(\text{(a) holds}) \geq 1 - (6L+3)^{2d} e^{-c L^{1/d}}.
\end{equation*}
By Lemma \ref{FiniteClusterSize}
\begin{equation*}
	\P(\text{(b) holds}) \geq 1 - (6L+3)^{2d} e^{-c L^{1/2}}.
\end{equation*}
Combined with \cite[Theorem 5]{PP96}, this implies
\begin{equation*}
	\P(\text{(b) and (c) holds}) \geq 1 - (6L+4)^{2d} e^{-c L^{1/2}}.
\end{equation*}
Finally, by \eqref{eq:Prob-Dlambda}
\begin{equation*}
	\P(\text{(d) holds}) \geq 1 - (2C_D L +1 )^d(\log n)^{C}n^{-d}\,.
\end{equation*}
Altogether, this completes the proof of the lemma.
\end{proof}

\begin{lemma}
 \label{StochasticalDomination}
For any $\epsilon>0$ the white vertices stochastically dominates a supercritical independent site percolation with parameter $1-\epsilon$ as long as $n$ is greater than a large constant depending on $(d, \epsilon)$. 
\end{lemma}
\begin{proof}
The lemma follows from Lemma \ref{kiprobability} and \cite[Theorem 0.0]{LSS97}.
\end{proof}
 \begin{wrapfigure}{R}{1.7in}
  \begin{center}
  \includegraphics[width=1.4in]{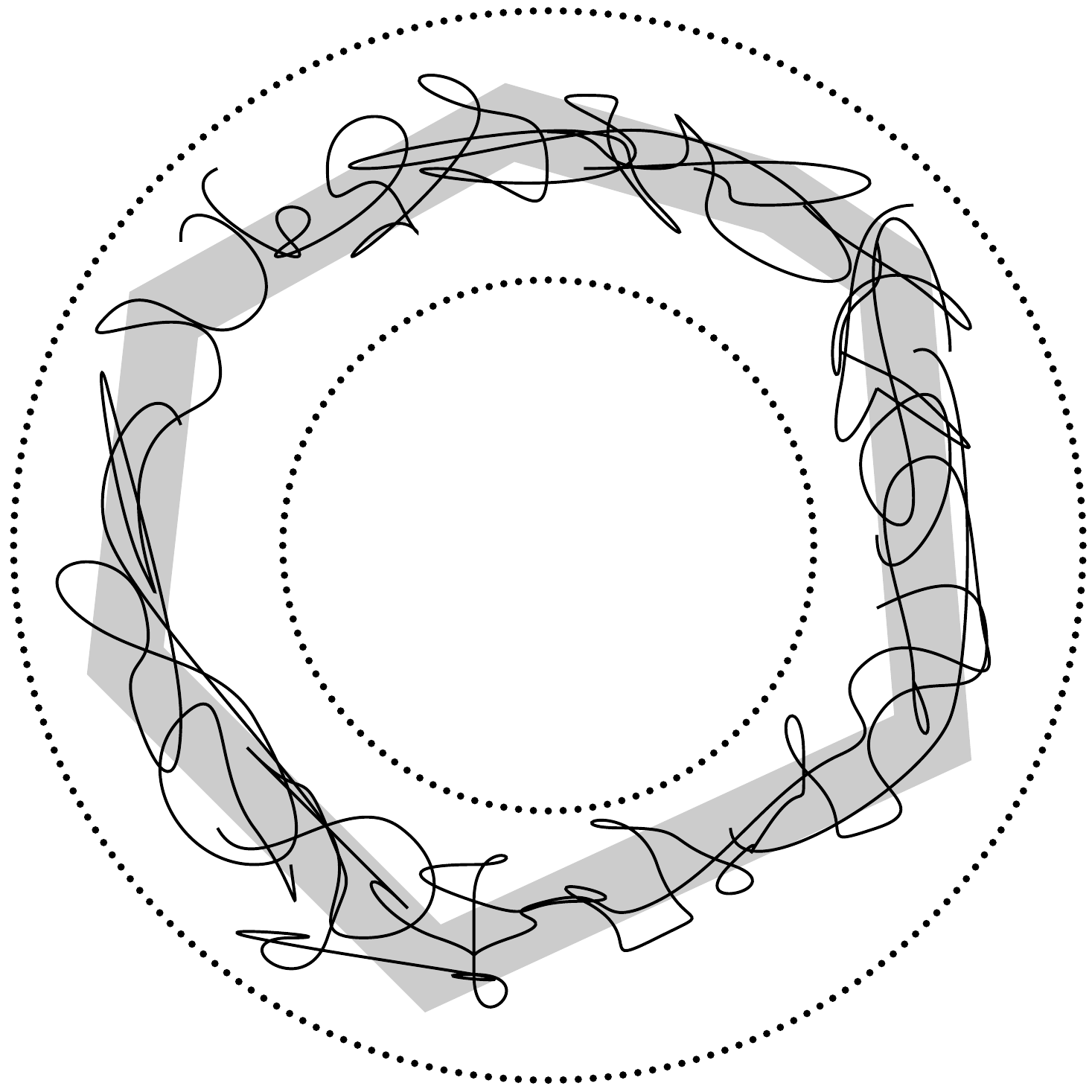}
	\begin{caption}
{Two dotted circles centered at $x$, with radius $r$ and $2r$ accordingly. Shaded area represents $\bba(x,r)$. Solid curves represent $\mac_{\bba(x,r)}$.}
\label{fig:cutset}
\end{caption}
  \end{center}
\end{wrapfigure}
In what follows, we will call vertices that are open in this Bernoulli($1-\epsilon$) percolation tilde-white. Thus, there is a coupling such that a tilde-white vertex is always white, and tilde-white vertices forms a Bernoulli percolation with parameter $1-\epsilon$. 
 In what follows, we will work with tilde-white vertices as opposed to white vertices. We will call the tilde-white percolation the  \textit{macroscopic} process and call the original site percolation (where open means free of obstacle) \textit{microscopic} process. For instance, we will refer to a microscopic path as a path that consists of vertices in the original lattice and a macroscopic path as a path that consists of vertices in the renormalized lattice. For any $x \in \Z^d$, we denote $\mbfi_x$ be such that $\mabox_{\mbfi_x}$ is the unique macroscopic box that contains $x$. For $\mathbf{A} \subseteq \Z^d$, define
\begin{equation}
\label{eq:def-K-A}
	\mabox_\mathbf{A}: = \cup_{\mbfi \in \mathbf{A}}\mabox_\mbfi \AND \mac_\mathbf{A}: = \cup_{\mbfi \in \mathbf{A}}\mac_\mbfi\,,
\end{equation}
where $\mac_{\mbfi}$ is defined as in Definition~\ref{de-kwhite} if $\mbfi$ is tilde-white and an empty set otherwise.
For $\mbfi \in \Z^d$, we denote by $\ck(\mbfi)$ the tilde-white cluster containing $\mabox_\mbfi$.
If $\mbfi$ is tilde-black (i.e., not tilde-white), then $\ck(\mbfi) = \varnothing$. 

For $A_1, A_2, A_3\subseteq\Z^d$, we say $A_1$ is a vertex cut that separates $A_2$ and $A_3$ if any path joining $A_2$ and $A_3$ has nonempty intersection with $A_1$.

\begin{lemma}
\label{cutset}
For any $x \in \Z^d$ and $r>(\log n)^{3d}$, let $\bba(x,r)$ be an arbitrary (macroscopic) tilde-white connected set such that $\mabox_{\bba(x,r)} $ is a vertex cut that separates $B_r(x)$ and $(B_{2r}(x))^c$ (See Figure \ref{fig:cutset}). If such cut does not exist, let $\bba(x,r) = \varnothing$. Then,
\begin{itemize}
	\item[(1)] With $\P$-probability at least $1 - e^{-rL^{-1}}$ we have $\bba(x,r)\neq \varnothing$.
	\item[(2)] For any $u \in \mabox_{\bba(x,r)}$ such that $|\mathcal C(u)|\geq L/10$, we have $u \in  \mac_{\bba(x,r)}$. Any $u,v \in \mac_{\bba(x,r)}$ can be join by a path in $\mac_{\bba(x,r)}$(thus in $(\ob \cup \sd_{\lambda_*})^c$) of length at most $r^{2d}$.
\end{itemize}
\end{lemma}

\begin{proof} 
(1) It suffices to prove that in the macroscopic lattice there exists a tilde-white connected set $\bba$ that separates $B_{R}(\mbfi_x)$ and $(B_{3R/2}(\mbfi_x))^c$ for $R = r(2L)^{-1}+2$, since each microscopic path not intersecting with $\mabox_\bba$ corresponds to a macroscopic path not intersecting with $\bba$. (Here, $B_{R}(\mbfi_x)$ is the discrete ball in the macroscopic lattice.)

	For $q>0$, we say a subset of $\Z^d$ is $\sharp_{q}$-connected, if it is connected with respect to the adjacency relation $$x \overset{\sharp_q}{\sim} y \iff  |x-y|_1 \leq q,.$$
	In addition, we say a subset of $\Z^d$ is $*$-connected, if it is connected with respect to the adjacency relation $$x \overset{*}{\sim} y \iff  |x-y|_\infty \leq 1.$$

	Now, let $A$ be the union of $B_R(\mbfi_x)$ and all tilde-black $\sharp_{3d}$-connected components that have nonempty intersection with $\cup_{\mathbf{j} \in B_R(\mbfi_x)}\{\mathbf{k}:|\mathbf{k} - \mathbf{j}|_1 \leq 3d\}$. 
	If $A \not \subseteq B_{3R/2-10d}(\mbfi_x)$, then there exists a $\sharp_{3d}$-connected tilde-black path connecting $B_{R+3d}(\mbfi_x)$ and $(B_{3R/2-10d}(\mbfi_x))^c$. By definition of tilde-white process, we get from a simple union bound over all $\sharp_{3d}$-connected tilde-black path connecting $B_{R+3d}(\mbfi_x)$ and $B_{3R/2-10d}(\mbfi_x)^c$ that
	\begin{align*}
		\P(A \not \subseteq B_{3R/2-10d}(\mbfi_x)) \leq (2R + 6d)^d (6d+1)^{d(3R/2-10d)}\epsilon^{3R/2-10d} \leq e^{-2R}\,,
	\end{align*}
	for sufficiently large $n$ (and thus $\epsilon$ is sufficiently small). 
	On the event $\{A\subseteq B_{3R/2-10d}(\mbfi_x)\}$, we let $A' = \cup_{\mathbf{j} \in A}\{\mathbf{k}:|\mathbf{k} - \mathbf{j}|_1 \leq d\}$ and $A_0$ be the connected component in $A'$ containing $B_R(\mbfi_x)$.
	Then by \cite[Lemma 2.1]{DP96} (which states that the external outer boundary of a $*$-connected set is $*$-connected), there is a subset of $\partial A_0$ which is $*$-connected and contains $A_0$ in its interior. We denote this subset by $\bba_0$ and let $\bba = \cup_{\mbfj \in \bba_0 }\{\mathbf{k}:|\mathbf{k} - \mathbf{j}|_\infty \leq 1\}
$. Then, the set $\bba$ is connected and is contained in $B_{3R/2}(\mbfi_x)$. It remains to prove that $\bba$ is tilde-white and contains $A = B_{R}(\mbfi_x)$ in its interior.

Note that for any $\mbfj \in \bba_0$, the $\ell_1$-distance between $\mbfj$ and $A$ is $d+1$. Hence, the set $\bba$ contains $A = B_{R}(\mbfi_x)$ in its interior.  By the $\sharp_{3d}$-connectivity, the $\ell_1$-distance between $\mbfj$ and any tilde-black point in  $A^c$ is at least $3d - (d+1) \geq d+1$. Hence, the set $\bba$ is also free of tilde-black points. 

\noindent (2) By Property 1 in Definition~\ref{de-kwhite},  $|\mathcal C(u)|\geq L/10$ implies that $u \in \mac_{\bba(x,r)}$. Also, since $\bba(x,r)$ is connected, by Properties 1 and 3 in Definition \ref{de-kwhite}, $ \mac_{\bba(w,r)} $ is connected. Then the last statement follows from $|\mac_{\bba(x,r)}| \leq r^{2d}.$
\end{proof}

\begin{lemma}
\label{chemdist2}
Recall $ \cregion = (\sd_\lambda \cup \ob)^c$ and let $C_2=6d^2$. Suppose that $\lambda \geq \lambda_*$ where $\lambda_*$ is defined in \eqref{eq:lambda-*-def}. For any $x,y \in \Z^d$ such that $|x-y|\geq n^\fuu$ with  $\P$-probability at least $1 - e^{- (\log |x-y|)^{3}}$ we have
\begin{itemize}
	\item [(1)] There is a unique connected component in $\cregion$ of diameter no less than $(\log |x-y|)^{C_2}$ that intersect with $B_{|x-y|^2}(x)$\,.
	\item[(2)] For any $u,v$ in this component, there exists a path in $\cregion$ that connects $u,v$ and which has length at most $\max (C|u-v|, (\log |x-y|)^{C}).$
\end{itemize}
\end{lemma}
\begin{proof}
Suppose $B_1,B_2 \subseteq \cregion$ are two connected set and both $B_1$ and $B_2$ have diameter at least $(\log |x-y|)^{C_2}$ and have nonempty intersection with $B_{|x-y|^2}(x)$. We will prove that with  $\P$-probability at least $1 - e^{- (\log |x-y|)^{3}}$ for any $v_1 \in B_1, v_2\in B_2$, there exists a path in $\cregion$ that connects $v_1,v_2$ and is of length less than $\max(C |v_1-v_2|,(\log |x-y|)^{C})$. Provided with this claim, the lemma follows immediately. Next we verify the claim.

Without loss of generality, we suppose $|B_1|, |B_2| \leq (2\log |x-y|)^{C_2d}$. We suppose $\bba(v,r)$ exists for $r = (\log |x-y|)^{5d}$ and all $v \in B_{2|x-y|^2}(x)$, which has probability at least $ 1 -e^{-(\log |x-y|)^4}$ according to Lemma \ref{cutset}. Since $C_2 = 6d^2$, it follows that $B_i \cap \mabox_{\bba(v_i,r)} \neq \varnothing$. By Property 1 in Definition \ref{de-kwhite} and that $B_i$ is open, it follows that $B_i \cap \mac_{\bba(v_i,r)} \neq \varnothing$.

Since tilde-white percolation is a supercritical independent site percolation, with probability 1 there exists a unique infinite (macroscopic) tilde-white cluster (c.f. \cite{AKN87}), which we denote as $\ck(\infty)$. Then the following holds with probability at least $1 - e^{-(\log |x-y|)^4}$ (by \cite[Theorems 1.1]{AP96} (or \cite[(3.8)]{DX17}) and Lemma \ref{FiniteClusterSize})
\begin{itemize}
	\item[(a)] For all $\mathbf{i,j} \in \ck(\infty)\cap B_{|x-y|^2}(\mbfi_x)$, $
		 D_{\ck(\infty)}(\mathbf{i,j}) \leq \max(C |\mathbf{i-j}|,( \log |x-y|)^5)$.
	\item [(b)] For all $\mbfi \in \Z^d \cap B_{|x-y|^2}(\mbfi_x)$, either $\ck(\mbfi) = \ck(\infty)$ or $|\ck(\mbfi)| \leq (\log |x-y|)^9$.
\end{itemize}
Then it follows from (b) that $\bba(v_i,r) \subseteq \ck(\infty)$ for $i=1,2$. Combined with (a) and the Property 2 in Definition \ref{de-kwhite}, it yields the claim we described at the beginning of the proof and thus complete the proof of the lemma.
\end{proof}

\subsection{Approximations of LWGF}
\label{sec:Approximations of LWGF}

In order to prove concentration inequality for $\varphi_*$ and sub-additivity for $\E \varphi_*$ in Section~\ref{section: Concentration, Sub-additivity}, we shall introduce $\bar \varphi_*$, and $\varphi^{\circ}$ as follows. For notation convenience, we will omit $\lambda$ in $\varphi_*$ (recalling that $\lambda \geq \lambda_*$). 

\begin{defn}
\label{def-phic}
	Recall \eqref{eq:phi-def}. We define $\bar \varphi$, the truncation of $\varphi$, as
\begin{align}
\label{eq:def-barphi}
	&\bar \varphi_*(x,y)  = \cul|x-y|\vee(\varphi_*(x,y) \wedge \Cbar|x-y|)\,,
\end{align}
and define  $\varphi^\circ$'s as ($\varphi^\circ$ is $\varphi$ with an additional restriction that RW does not exit a big ball centered at $x$ ---  the reader may remember the notation by interpret the superscript $\circ$ as ``restriction to a ball'') 
\begin{align}
\label{eq:def-phi[1]}
	&\varphi^{\circ}(x,y) = -\log \big(\Ex^{x}\big[ \lambda^{-\tau_{y}} ;\tau_{\sd_{\lambda} \cup \ob}>\tau_{y} , \xi_{B_{R_{\circ}}(x)} > \tau_{y}\big]  \big)\,, \nonumber \\
	&\varphi^{\circ}_*(x,y) = \min_{x' \in B_{r(x,y)}(x)}\min_{y' \in B_{r(x,y)}(y)}  \varphi^{\circ}(x',y') \,,\\
	&\bar \varphi^{\circ}_*(x,y)  = \cul|x-y|\vee(\varphi^{\circ}_*(x,y) \wedge \Cbar|x-y|)\,,\nonumber
\end{align}
where $\cul, \Cbar>0$ are two constants to be selected (the underline and bar decorations in the notation correspond to lower and upper truncations), $R_{\circ} = |x-y|(\log |x-y|)^{C_{\circ}}$, 
$C_{\circ}= C_1 + 3$.
(Recall $r(x,y) = (\log |x-y|)^{2\kappa + 10d}$).
\end{defn} 
It follows directly from definition that 
\begin{equation}
\label{eq:phiorder}
	\varphi^{\circ}(x,y) \geq \varphi(x,y) \geq \varphi(x,y)\,.
\end{equation}
Some remarks are in order for Definition~\ref{def-phic}.
\begin{itemize}
\item As we will show in Lemma~\ref{phic-app} $\varphi_*(x,y)$ is linear in $|x-y|$. Then we can (and we will) choose $\cul,\Cbar$ such that these $\varphi_*$ (respectively $\varphi_*^\circ$) and $\bar\varphi_*$ (respectively $\bar \varphi_*^\circ$) are close to each other. Such truncations are useful as they allow us to work with functions that are deterministically bounded from below and above.
\item We will show in Lemma~\ref{diamonddiff}  that $\varphi$ and $\varphi^\circ$ are close to each other and show in Lemma~\ref{con-con} that concentration of $\bar \varphi_*^{\circ}(x,y)$ around its expectation is sufficient to guarantee the concentration of $\bar \varphi_*(x,y)$. This is useful since it is more convenient to prove concentration for $\bar \varphi_*^{\circ}(x,y)$ for the reason that it only depends on a finite number of random variables.
\end{itemize}

\begin{lemma}
\label{phic-app}
	There exist constants $\cul,\Cbar>0$ only depending on $(d,p)$ such that for all $x,y \in \Z^d$ with $|x-y|>n^{\fuu}$, the following holds with probability at least $1-e^{-(\log |x-y|)^{2}}$: for all $\lambda \in [\lambda_*,1]$
	\begin{align}
	&\label{eq:phi-app-upper}\varphi^{\circ}_*(x,y) \leq \tfrac{1}{4}\Cbar|x-y|\,, \\
	&\label{eq:phi-app-lower}\varphi_*(x,y) \geq 4\cul|x-y|\,.
	\end{align}
As a result, we have for all $\lambda \in [\lambda_*,1]$
\begin{equation}
\label{eq:phidev}
\varphi_*(x,y) = \bar \varphi_*(x,y), \varphi^{\circ}_*(x,y)  =  \bar \varphi^{\circ}_*(x,y)
\end{equation}
\end{lemma}
\begin{lemma}
\label{diamonddiff}
	For all $x,y \in \Z^d$ such that $|x-y| > n^\fuu$. If $\varphi(x,y) \leq 2\Cbar|x-y|$, then
	\begin{align}
	&\varphi(x,y) \geq \varphi^{\circ} (x,y) - e^{-|x-y|}\,, \label{eq:diamonddiff}\\
	\Ex^x\big[\tau_{y}\lambda^{-\tau_{y}};&\tau_{\sd_{\lambda} \cup \ob}>\tau_{y} \big]/\Ex^x\big[\lambda^{-\tau_{y}};\tau_{\sd_{\lambda} \cup \ob}>\tau_{y} \big] \leq |x-y|(\log |x-y|)^{C_{\circ}} \,. \label{eq:ex-taulambda-tau}
\end{align}
\end{lemma}

\begin{lemma}
\label{con-con}
For all $x,y \in \Z^d$ such that $|x-y| > n^\fuu$,\begin{align*}
    \{&\E\left[\bar \varphi_*(x,y)\right] - \bar \varphi_*(x,y)\leq 2^{-1}|x-y|^{\tuu}, 4 \cul |x-y| \leq \bar \varphi_*(x,y)  \leq 4^{-1}\Cbar|x-y|\}\\
  &\subseteq \{ \E\left[\bar \varphi_*^{\circ}(x,y)\right] - \bar \varphi^{\circ}_*(x,y) \leq |x-y|^{\tuu} ,2\cul |x-y| \leq \bar \varphi^{\circ}_*(x,y) \leq 2^{-1}\Cbar|x-y| \}\\
  &\quad \subseteq \{\E\left[\bar \varphi_*(x,y)\right] - \bar \varphi_*(x,y) \leq 2|x-y|^{\tuu},  \cul |x-y|<\bar \varphi_*(x,y) < \Cbar|x-y|\}\,.
\end{align*}
\
\end{lemma}


Now, we prove the aforementioned results as well as Lemma~\ref{ind-appr}. We first record the following corollary, which is an immediate consequence of \eqref{eq:SurvivalProb-SD}.
\begin{cor}
\label{G-polylogn}
There exists $C_3 = C_3(d,p)$ such that for all $x,y \in \Z^d$ and $0<\lambda<1$
\begin{equation}
    G_{\cregion}(x,y;\lambda) \leq (\log n)^{C_3}\,.
\end{equation}
\end{cor}

\begin{proof}[Proof of  Lemma~\ref{phic-app}: \eqref{eq:phi-app-upper}]
We will work on the event that both $\bba(x,r(x,y)/2)$ and $\bba(y,r(x,y)/2)$ are non-empty and the properties described in Lemma \ref{chemdist2} holds for ($x,y,\lambda_*$). By Lemmas \ref{cutset} and \ref{chemdist2}, this event has probability at least $1 - 2e^{(\log |x-y|)^{3}}$. And we note that such a event does not depend on $\lambda$. 

On such a event, since $\mac_{\bba(x,r(x,y)/2)}$ and $\mac_{\bba(y,r(x,y)/2)}$ are in $(\sd_{\lambda_*} \cup \ob)^c$ and of diameter at least $(\log |x-y|)^{C_2}$,Lemma~\ref{chemdist2} (2) implies there is a path in $\mathcal V_{\lambda_*}$ of lenth at most $C|x-y|$ that connects $B_{r(x,y)}(x)$ and $B_{r(x,y)}(y)$. Therefore, by forcing the random walk to go along such a path we get
\begin{equation*}
	\max_{\lambda \in [\lambda_*,1]}\varphi^\circ_*(x,y;\lambda) \leq C \log (2d) |x-y|\,.\qedhere
\end{equation*}
\end{proof}

The next result will be useful in proving \eqref{eq:phi-app-lower}.
\begin{lemma}
\label{Linear}
For all $u \in \mathcal C(0)$ with $|u| \geq n^{1/2}$, with $\P$-probability at least $1 - 2e^{-(\log |u|)^{2d}}$, for all $\lambda \in [\lambda_*,1]$
  \begin{equation*}
  \Ex \left[ \lambda^{-\tau_{u}} \11_{\tau_{\sd_{\lambda} \cup \ob}>\tau_{u}} \11_{\tau_u < C|u|}\right] \geq \Ex \left[ \lambda^{-\tau_{u}} \11_{\tau_{\sd_{\lambda} \cup \ob}>\tau_{u}} \right](1 - e^{-|u|^{1/101}})\,.
\end{equation*}
\end{lemma}
\begin{proof}
By setting $\lambda = 1$ and $u \in B_n(0)$, this lemma reduce to the path localization result $\ft_u \leq C|S_{\ft_u}|$ (see \eqref{eq:poly-loc}, \cite[(5.18)]{DX17}). The proof of the lemma is a straightforward adaption of arguments in \cite{DX17}, and here we only describe how to implement such an adaption. 

For $u, v\in \mathbb Z^d$, $A\subseteq \Z^d$ and $r\geq 1$, we define $\K_{A, r}(u, v)$ as in \cite[Equation (2.16)]{DX17}.  For a random walk path $\omega$, we define its unique loop erasure decomposition as in \cite[Equation (5.1)]{DX17} where $\eta = \eta(\omega)$ is the loop erasure and we define $\phi(\omega)$ as in \cite[Equation (5.10)]{DX17}. For $t \geq 1$, as in \cite{DX17}, we let $\mathcal{M}(t)$ be the collection of sites $v$ such that $\Pr^v(\tau > t) \geq e^{-t/(\log t)^2}$. We also define
$$A_t(\omega) = \{0\leq i \leq |\eta|: |l_i| = t,\eta_i \not \in \mathcal M(t)\}\,.$$
We will work on the following event (which do not depends on $\lambda$) 
\begin{equation}
\label{eq:lin-dis-event}
	\big \{ |\gamma \cap \mathcal M(t)| \leq e^{-(\log t)^{3/2}} |\gamma|, \quad \forall \gamma \in \cup_{m \geq |u|} \mathcal W_m(0), t \geq t_1 \big \}
\end{equation}
where $\calW_{m}(0)$ is the set of self-avoiding paths in $\Z^d$ of length $m$ with initial point $0$ and $t_1$ is a constant only depending on $(d,p)$. By \cite[Lemma 5.3]{DX17}, the event described in \eqref{eq:lin-dis-event} has $\P$-probability at least $1 - e^{-(\log |u|)^{2d}}$. (Note that we should replace $``n(\log n)^{-100d^2}$'' in the proof of \cite[Lemma 5.3]{DX17} by ``$|u|$''.) 
By $\lambda \geq \lambda_*$ and \eqref{eq:SurvivalProb-SD} we could see that $$\max_{u \not \in \mathcal M(t)}\Pr^u(\tau_{\sd_{\lambda} \cup \ob} > t) \leq \begin{cases}
  \lambda^{t}e^{-2^{-1}t/(\log t)^2}, &t \leq k_n^{50d}\,,\\
  \lambda^{t}e^{-2^{-1}tk_n^{-21d}}, &t \geq k_n^{50d}\,.
\end{cases}$$
As in the proof of \cite[Lemma 5.5]{DX17}, there exists a constant $t_2 = t_2(d,p)$ such that for $m\geq 1$, $t \geq t_2$, $\gamma \in \phi( \K_{(\sd_\lambda \cup\ob)^c  \setminus \{u\},m}(0,u))$ with $ m -|\gamma| \geq |\eta|t^{-9}$ we have
  \begin{align*}
    \Pr(\{\omega \in \K_{(\sd_\lambda \cup\ob)^c  \setminus \{u\},m}(0,u) : \phi(\omega) = \gamma\})&\leq \Pr(\gamma) \binom{|\eta|}{\frac{m-|\gamma|}{t}} \big(\max_{u \not \in \mathcal M(t)}\Pr^u(\tau_{\sd_{\lambda} \cup \ob} > t)\big)^{\frac{m-|\gamma|}{t}}\,.
  \end{align*}
Substituting the bound on $\max_{u \not \in \mathcal M(t)}\Pr^u(\tau_{\sd_{\lambda} \cup \ob} > t)$ and using $\binom{M}{N} \leq (eM/N)^N$ gives 
$$ \Pr(\{\omega \in \K_{(\sd_\lambda \cup\ob)^c  \setminus \{u\},m}(0,u) : \phi(\omega) = \gamma\})\leq \left( \Pr(\gamma) \lambda^{m - |\gamma|}\right) \cdot e^{-3^{-1}(m-|\gamma|)k_n^{-21d}}\,.$$
For $t \leq |u|^{1/20}$, since $|\eta| \geq |u| \geq n^{1/2}$, we have $(m-|\gamma|)k_n^{-21d} \geq |\eta|t^{-9}k_n^{-21d} \geq |u|^{1/2}$. For $t \geq |u|^{1/20}$, since $m-|\gamma|$ is a multiple of $t$, $(m-|\gamma|)k_n^{-21d} \geq t \cdot k_n^{-21d} \geq |u|^{1/50}t^{1/2}$.
Then summing over all $m,\gamma$ such that $ m -|\gamma| \geq |\eta|t^{-9}$ and $t \geq t_2$ we get
  \begin{align}
  \label{eq:linear - 1}
   \sum_{t \geq t_2}&\sum_{m \geq 0} \lambda^{-m}\Pr(\{\omega \in \K_{(\ob \cup \sd_{\lambda})^c \setminus \{u\},m}(0,u) : |A_t(\omega)| \geq |\eta| t^{-10}\}) \nonumber \\
   &\leq \Ex \left[ \lambda^{-\tau_{u}} \11_{\tau_{\sd_{\lambda} \cup \ob}>\tau_{u}} \right] e^{-|u|^{1/100}}\,.
  \end{align}
Next, as it was treated in \cite[Corollary 5.9]{DX17}, we combine \eqref{eq:linear - 1} and \eqref{eq:lin-dis-event} to deduce that 
and on event \eqref{eq:lin-dis-event}, we have
  \begin{equation*}
  \Ex \left[ \lambda^{-\tau_{u}} \11_{\tau_{\sd_{\lambda} \cup \ob}>\tau_{u}} \11_{\tau_u \geq C'\cdot \eta(S_{[0,\tau_u]})}\right]
      \leq \Ex \left[ \lambda^{-\tau_{u}} \11_{\tau_{\sd_{\lambda} \cup \ob}>\tau_{u}} \right] e^{-|u|^{1/100}} 
  \end{equation*}
  where $C'$ is an appropriately chosen constant. 
Finally, by \cite[Lemma 5.4]{DX17}, we see that
\begin{equation*}
\Ex \left[ \lambda^{-\tau_{u}} \11_{\tau_{\sd_{\lambda} \cup \ob}>\tau_{u}} \11_{\tau_u < C'\cdot \eta(S_{[0,\tau_u]})}
\11_{\eta(S_{[0,\tau_u]}) \geq C'' |u|}\right]
 \leq \sum_{j \geq C''|u|} (c' \lambda_*^{-C'} )^j\,,
\end{equation*}
for some constant $c'= c(d,p) \in (0,1)$ and sufficiently large $C''$.
Combining the preceding two  inequalities and \eqref{eq:phi-app-upper} completes the proof of the lemma.
\end{proof}

\begin{proof}[Proof of  Lemma~\ref{phic-app}: \eqref{eq:phi-app-lower} and \eqref{eq:phidev}]
By Lemma \ref{Linear}, for any $u \in B_{r(x,y)}(x) $, 
$v \in B_{r(x,y)}(y)$ with $\P$-probability at least $1 - 2e^{-(\log |u-v|)^{2d}}$ we have that for all $\lambda \in [\lambda_*,1]$
\begin{align*}
  \varphi(u,v) &= -\log\Ex^{u} \left[ \lambda^{-\tau_{v}} \11_{\tau_{\sd_{\lambda} \cup \ob}>\tau_{v}} \right]\geq - \log \Ex^{u} \left[ \lambda^{-\tau_{v}} \11_{\tau_{\sd_{\lambda} \cup \ob}>\tau_{v}} \11_{\tau_{v} < C|u-v|}\right]-1\\
  & \geq -\log \Pr^{u}(\tau>\tau_{v}) - (1+C|u-v| \log \lambda )\,.
\end{align*}
At the same time, since any path connecting $u$ and $v$ has a loop erasure of length at least $|u-v|$, \cite[Lemma 5.4]{DX17} implies that there exists $c' = c'(d,p) \in (0,1)$ such that
$$\P(\Pr^u(\tau > \tau_v) \geq (c')^{|v-u|}) \leq e^{-c|v-u|}\,.$$
We complete the proof of \eqref{eq:phi-app-lower} by combining preceding two inequalities and using a union bound over all $u \in B_{r(x,y)}(x) $, $v \in B_{r(x,y)}(y)$.

Since \eqref{eq:phidev} is an immediate consequence of \eqref{eq:phi-app-lower} and \eqref{eq:phi-app-upper}, we have completed the proof of the lemma.
\end{proof}

\begin{proof}[Proof of Lemma \ref{diamonddiff}]
We denote $j_* =|x-y|(\log |x-y|)^{C_{\circ}-1}$. 
By \eqref{eq:SurvivalProb-SD}, we have for any $x,y \in \Z^d$,
\begin{equation}
\label{eq:ex-taulambda-tau -1}
\begin{split}
    &  \Ex^x\left[\tau_{y}\lambda^{-\tau_{y}};\tau_{\sd_{\lambda} \cup \ob}>\tau_{y},\tau_y \geq j_*\right]\leq \sum_{j\geq j_*}j\lambda^{-j}\Pr^x(\tau_{\sd_\lambda \cup \ob}>j)\\
  & \leq \sum_{j \geq j_*}j(\log n)^{C}(1 - (\log n)^{-100d})^j\leq e^{-|x-y|\log|x-y|}\,.
\end{split}
\end{equation}
Since
$
\lambda^{-\tau_{y}} \leq \tau_y\lambda^{-\tau_{y}}
$, we have that $e^{-\varphi(x,y)} - e^{-\varphi^{\circ}(x,y)}$ equals to
\begin{equation*}
\Ex^{x} [ \lambda^{-\tau_{y}} ;\tau_{\sd_{\lambda} \cup \ob}>\tau_{y}\geq  \xi_{B_{j_*}(x)}]\leq \Ex^{x} [\lambda^{-\tau_{y}} ;\tau_{\sd_{\lambda} \cup \ob}>\tau_{y}\geq j_*] \leq e^{-|x-y|\log|x-y|}\,.\label{eq-phi-compare}
\end{equation*}
Combined with $\varphi(x,y) \leq 2\Cbar|x-y|$, it yields
\begin{equation*}
	\varphi^{\circ} (x,y) - \varphi(x,y) \leq -\log( 1 - e^{-|x-y|\log |x-y| + 2\Cbar|x-y|}) \leq e^{-|x-y|}\,,
\end{equation*}
verifying \eqref{eq:diamonddiff}. Next, we prove \eqref{eq:ex-taulambda-tau}. To this end, we observe
\begin{align*}
	\Ex^x\left[\tau_{y}\lambda^{-\tau_{y}};\tau_{\sd_{\lambda} \cup \ob}>\tau_{y},  \tau_{y} \leq j_*\right]\leq j_*\Ex^x\left[\lambda^{-\tau_{y}};\tau_{\sd_{\lambda} \cup \ob}>\tau_{y}\right]
	 =  j_*e^{-\varphi(x,y)}\,.
\end{align*}
Combined with \eqref{eq:ex-taulambda-tau -1} and $\varphi(x,y) \leq 2\Cbar|x-y|$, it yields that
\begin{align*}
	\Ex^x\left[\tau_{y}\lambda^{-\tau_{y}};\tau_{\sd_{\lambda} \cup \ob}>\tau_{y} \right] 
	&\leq j_*e^{-\varphi(x,y)} + e^{-|x-y|\log|x-y|} \leq |x-y|(\log |x-y|)^{C_{\circ}}e^{-\varphi(x,y)}\,,
\end{align*}
which gives \eqref{eq:ex-taulambda-tau} as desired and thus completes the proof of the lemma.
\end{proof}

\begin{proof}[Proof of Lemma \ref{con-con}]
	Combining \eqref{eq:diamonddiff} and Definition \ref{def-phi}, we get that 
$$\varphi_*(x,y) \leq 2\Cbar|x-y| \implies \varphi_*(x,y) \geq \varphi_*^{\circ} (x,y) - e^{-|x-y|/2}\,. $$
Since $\varphi_*(x,y) \leq \varphi_*^{\circ} (x,y)$, it suffices to prove $\E\left[\bar \varphi_*^{\circ}(x,y)\right] \leq \E\left[\bar \varphi_*(x,y)\right] + e^{-(\log |x-y|)^2/2}$.
Now, denote event $E = \{\bar \varphi_*(x,y) = \varphi_*(x,y)$ and $\bar \varphi^{\circ}_*(x,y) =  \varphi^{\circ}_*(x,y)\}$. Then by \eqref{eq:phidev}, we have
  $$  \E\left[\bar \varphi_*(x,y)\right] -  \E\left[\varphi_*(x,y) \11_{E}\right] \leq \Cbar|x-y|e^{-(\log|x-y|)^2}\,,$$
    $$  \E\left[\bar \varphi^{\circ}_*(x,y)\right] -  \E\left[\varphi^{\circ}_*(x,y) \11_{E}\right] \leq \Cbar|x-y|e^{-(\log|x-y|)^2}\,.$$
Suppose $\varphi_*(x,y) = \varphi(u,v)$, for $u \in B_{r(x,y)}(x)$ and 
$v \in B_{r(x,y)}(y) $. By \eqref{eq:diamonddiff}, we have $$\varphi_*(x,y)\11_{E} = \varphi(u,v)\11_{E} \geq \varphi^{\circ} (u,v)\11_{E} - e^{-|u-v|} \geq \varphi_*^{\circ}(x,y)\11_{E} - e^{-|x-y|/2}\,.$$
We complete the proof by combining preceding three inequalities.
\end{proof}

\begin{proof}[Proof of Lemma \ref{ind-appr}]
We first note that combining Lemma \ref{FiniteClusterSize} and \eqref{eq:Prob-Dlambda} gives \begin{equation}
\label{eq:G-0-0-percolate}
	\P(G_0\bigtriangleup \{ 0\in \calC(\infty)\}) \to 0\,.
\end{equation}
So $\P(G_0)$ is bounded away from $0$ for large $n$.
Let $\tempset =B_{(\log n)^{\kappa/2}}(v)$ where $\kappa$ is chosen in Theorem \ref{Polylog thm}. Since $\varphi(0, z)\geq \varphi_*(0, v)$ for all $z\in \partial_i A$, we have the upper bound
\begin{equation}\label{eq-4.3-upper-bound}
  \Ex \left[\lambda^{-\ft_v}  \11_{\tau_{\sd_{\lambda }\cup \ob}>\ft_v}\right]  \leq \sum_{z \in \partial_i \tempset }e^{- \varphi(0,z)} \leq 2^d(\log n)^{d\kappa/2} e^{-\varphi_*(0,v) }\,.
\end{equation}
For the lower bound, we will work on the event such that the following holds (which does not depend on $\lambda$).
\begin{itemize}
	\item $\bba(v,(\log n)^{C_2}) \not = \varnothing$, 
	\item Properties discribed in Lemma \ref{chemdist2} hold for $(0,n\mathbf e_1,\lambda_*)$,
	\item For all $x \in B_{r(0,v)}(0) $, $y \in  B_{r(0,v)}(v)$ and $\lambda \in [\lambda_*,1]$,
  \begin{equation}
  \label{eq:phistar-cut}
  \Ex^x \left[ \lambda^{-\tau_{y}} \11_{\tau_{\sd_{\lambda} \cup \ob}>\tau_{y}} \11_{\tau_y < C|y|}\right] \geq \Ex^x \left[ \lambda^{-\tau_{y}} \11_{\tau_{\sd_{\lambda} \cup \ob}>\tau_{y}} \right](1 - e^{-|y|^{1/101}})\,.
\end{equation}
\end{itemize}
Combining Lemmas \ref{cutset}, \ref{chemdist2} and Lemma \ref{Linear} (since $|v| \geq n^{2/3}$ ensures $|y| \geq n^{1/2}$), this event has probability at least $1 - e^{-c(\log n)^2}$.

Now, we choose $x \in B_{r(0,v)}(0) $ and $y \in  B_{r(0,v)}(v)$ such that $\varphi_*(0,v) = \varphi(x,y)$ and we claim that 
\begin{equation}
\label{eq:two path}
  D_{\mathcal V_{\lambda_*}}(0,x) \leq (\log n)^C \text{   and   there exists }  z_0 \in A \ \mbox{ such that } \ D_{\mathcal V_{\lambda_*}}(y,z_0) \leq (\log n)^C\,.
\end{equation}
Provided with this, by forcing random walk to travel from $0$ to $x$ at the beginning and travel from $y$ to $z_0$ at the end (both along the geodesics), we have that for $i \geq 0$,
$$\Pr(S_{i + D_{\mathcal V_{\lambda_*}}(0,x) + D_{\mathcal V_{\lambda_*}}(y,z_0)} = z_0,S_{[1,{i + D_{\mathcal V_{\lambda_*}}(0,x) + D_{\mathcal V_{\lambda_*}}(y,z_0)}-1]} \subseteq \cregion)$$
is bounded below by $(2d)^{-2(\log n)^C}\Pr^x(S_i = y,S_{[1,i-1]} \subseteq \cregion)$.
Hence,
\begin{align}
		\Ex^x \left[ \lambda^{-\tau_{y}} \11_{\tau_{\sd_{\lambda} \cup \ob}>\tau_{y}} \11_{\tau_y < C|y|}\right] &\leq \sum_{i = 0}^{C|y|} \lambda^{-i} \Pr^x(S_i = y,S_{[1,i-1]} \subseteq \cregion) \nonumber\\
		& \leq(2d)^{2(\log n)^C}\sum_{i = 0}^{C|y|} \lambda^{-i} \Pr(S_i = z_0,S_{[1,i-1]} \subseteq \cregion)\,,\label{eq:lem4.3-1}
\end{align}
where in the last step, we have changed the index $i  + D_{\mathcal V_{\lambda_*}}(0,x) + D_{\mathcal V_{\lambda_*}}(y,z_0) \mapsto i$.  Decomposing the random walk path depending on the entrance point in $A$ and using strong Markov property, we get that
\begin{align*}
 \eqref{eq:lem4.3-1} &\leq (2d)^{2(\log n)^C} \sum_{z \in \partial_i A}\sum_{i = 0}^{C|y|} \lambda^{-i} \Pr(S_i = z,S_{[1,i-1]} \subseteq \cregion \setminus A) \cdot G_{\cregion}(z,z_0;\lambda) \\
 &\leq (2d)^{2(\log n)^C}(\log n)^{C_3}\Ex \big[\lambda^{-\ft_v}  \11_{n \wedge\tau_{\sd_{\lambda }\cup \ob} \geq\ft_v}\big]\,,
\end{align*}
where in the last step, we used Corollary \ref{G-polylogn} and $|y|\leq 2\cucc n(\log n)^{-2/d}$ (since $|v| \leq \cucc n(\log n)^{-2/d})$).
Combining it with \eqref{eq:phistar-cut} gives the lower bound
$$\Ex \big[\lambda^{-\ft_v}  \11_{n \wedge\tau_{\sd_{\lambda }\cup \ob} \geq\ft_v}\big] \geq e^{-(\log n)^C}\Ex^x \left[ \lambda^{-\tau_{y}} \11_{\tau_{\sd_{\lambda} \cup \ob}>\tau_{y}} \right]\,.$$
Combined with \eqref{eq-4.3-upper-bound} and \eqref{eq:G-0-0-percolate}, this yields the result of the lemma.

It remains to prove \eqref{eq:two path}. Since $x$ is connected to $y$ in $\cregion$ and (by \eqref{eq:def-G-0}) $0$ is connected to $B^c_{n^{1/2}}(0)$ in $\cregion$, Lemma~\ref{chemdist2} (2) yields the first part of \eqref{eq:two path}. The second part also follows directly if $y \in A$. If $y \not \in A$, since we assumed $\bba(v,(\log n)^{C_2}) \not = \varnothing$, applying Lemma \ref{chemdist2} to $\mac_{\bba(v,(\log n)^{C_2})}$ yields that the vertex $y$ is connected to $\mac_{\bba(v,(\log n)^{C_2})}$ in $\mathcal V_{\lambda_*}$. Since $A \supseteq B_{3(\log n)^{C_2}}(v)$, we get from Lemma~\ref{chemdist2} (2) that  $y$ is connected to some $z_0 \in \partial_i \tempset$ in $\mathcal V_{\lambda_*}$ by a path of length at most $(\log n)^{C}$. 
\end{proof}

\subsection{Concentration and Sub-additivity of LWGF}
\label{section: Concentration, Sub-additivity}
In this subsection, we will prove two key properties of LWGF: concentration as incorporated in Lemma~\ref{concentration} and sub-additivity as incorporated in Lemma~\ref{sub-additivity}.
\begin{lemma}
\label{concentration}
For any $q\geq 2$, $|x-y| > n^{\fuu}$ and sufficiently large $n$
	\begin{equation}
	\mathrm{Var}\left[ \bar \varphi_*(x,y) \right] \leq (\log |x-y|)^C \cdot |x-y|\,,
\end{equation}
\begin{equation}
	\E \left(\left(  \E \bar \varphi_*(x,y) - \bar \varphi_*(x,y) \right)_+\right)^q \leq (\log |x-y|)^{qC} \cdot |x-y|^{q/2}\,,
\end{equation}
where we used the notation $a_+ = a \11_{a\geq 0}$.
\end{lemma}

We will use the Efron-Stein inequality to bound the second and higher moments (see \cite[Theorem 2]{BBLM05}). To this end, we will control the increment of $ \bar \varphi_*(x,y)$ when resampling obstacles in $(B_{2R_{\circ}}(x))^c$ and resampling the obstacle at each $w \in B_{2R_{\circ}}(x)$. We see that Lemma \ref{diamonddiff}  implies resampling obstacles in $(B_{2R_{\circ}}(x))^c$ only has a very small effect on LWGFs (see \eqref{eq:concentration-bound-circ} below). It is a more complicated issue to control the influence from resampling $w \in B_{2R_{\circ}}(x)$. To this end, we will employ the following two types of bounds.
\begin{itemize}
\item We show in Lemma~\ref{lem-concentration-bound-w} that on a typical environment, the increment of  $\bar \varphi_*(x,y)$ due to the resampling at $w$ is at most poly-logarithmic in $|x-y|$. The proof is divided into two cases as follows.
\begin{itemize}
 \item  For $w$ near $x$ (or $y$), (recalling that in defining $\bar \varphi_*(x,y)$ we have optimized over starting and ending points that are near $x$ and $y$ respectively) we will  choose a different starting point (or end point) and consider the set of paths that do not get close to $w$.  See Lemma~\ref{lem-concentration-bound-w} \underline{\textsc{Case 1}} below.
 \item For $w$ away from $x$ and $y$, we note that the random walk can take a detour using $\mac_{\bba(w,(\log |x-y|)^C)}$ for some constant $C>0$ (recall $\mac$ as in Definition \ref{de-kwhite} and $\bba$ as in Lemma \ref{cutset}) to avoid getting close to $w$.  See Lemma~\ref{lem-concentration-bound-w} \underline{\textsc{Case 2}} below.
 \end{itemize}
 \item We prove in Lemma~\ref{lem-concentration-3-2} a bound on the increment by a direct computation which takes into account how likely the random walk will get close to $w$. 
\end{itemize}
\begin{remark}\label{rem-accessibility}
Due to the requirement of avoiding $\sd_{\lambda}$ in the definition of LWGFs, by resampling the obstacle at $w$, it is possible to change the accessibility for more than just the vertex $w$. However, by \eqref{eq-w-accessibility} below the change on accessibility is confined in a local ball around $w$.
\end{remark}
In order to implement the preceding outline in detail, we first introduce a few definitions.  For $w \in B_{2R_{\circ}}(x)$, we denote  $$\ob_w = (\ob \setminus \{w\}) \cup (\ob' \cap  \{w\}),\quad\ob_{\circ} = (\ob \setminus (B_{2R_{\circ}}(x))^c) \cup (\ob' \cap  (B_{2R_{\circ}}(x))^c)$$ where $\ob'$ is an independent copy of $\ob$ (that is, $\ob_w$ is obtained from $\ob$ by re-sampling the environment at $w$). Write $\varphi(x,y;\ob)$ to emphasize its dependence on the environment.
Let $E_{x, y}$ be the event such that the following hold:
\begin{itemize}
	\item For any $ v \in B_{|x-y|^2}(x)$, either $\mathcal C (v) = \mathcal C(\infty)$ or $|\mathcal C (v)| \leq (\log |x-y|)^{5}$.
	\item For all $w \in B_{|x-y|^2}(x)$, $\bba(w,r) \not = \varnothing$ for $r := (\log |x-y|)^{C_1+C_2 + 6d}$.
	\item $\bar\varphi_*(x,y;\ob) =\varphi_*(x,y;\ob)$, and for all $w \in B_{|x-y|^2}(x)$, 
  $\bar\varphi_*(x,y;\ob_w) =\varphi_*(x,y;\ob_w).$
\end{itemize}
Then by Lemmas \ref{FiniteClusterSize}, \ref{cutset}, \eqref{eq:phidev} and \eqref{eq:phiorder},
\begin{align}
\label{eq:con-e-prob}
	\P(E_{x,y}^c) \leq& (2|x-y|^2)^de^{-c(\log|x-y|)^{5/2}} + (2|x-y|^2)^de^{-r (\log n)^{-2d}} + 3(2|x-y|^2)^d e^{-(\log|x-y|)^{2}} \nonumber\\
	\leq& e^{-(\log|x-y|)^{3/2}}\,.
\end{align}
On the event $E_{x, y}$, by \eqref{eq:phi-def} we can choose $u \in B_{r(x,y)}(x) $ and 
$v \in B_{r(x,y)}(y)$ such that 
\begin{equation}\label{eq-phi-phi-*} 
\bar\varphi_*(x,y;\ob) =\varphi_*(x,y;\ob) = \varphi(u,v;\ob)\,.
\end{equation}
 Let $$ \cregion_{,w} =(\sd_{\lambda}(\ob_w) \cup \ob_w)^c\,.$$ Then $$\varphi(x,y,\ob_w) = -\log \big(\Ex^{x}\big[ \lambda^{-\tau_{y}} ;\xi_{\cregion_{,w}} > \tau_{y}\big]  \big)\,.$$

Now, we claim that on the event $E_{x, y}$, 
\begin{align}
	\varphi(u,v;\ob)\geq \varphi_*(x,y; \ob_\circ) - e^{-|x-y|/2} \,.\label{eq:concentration-bound-circ}
\end{align}
We first see that \eqref{eq:def-phi[1]} implies $\varphi(u,v;\ob_\circ) \leq \varphi^{\circ}(u,v;\ob)$ and Lemma \ref{diamonddiff}  yields $\varphi^{\circ}(u,v;\ob) \leq \varphi(u,v;\ob) + e^{-|u-v|}$. Then  \eqref{eq:concentration-bound-circ} follows from \eqref{eq:phi-def}.

Further, for all $w \in B_{|x-y|^2}(x)$ let $\tempset_w$ denote the union of $\mabox_{\bba(w,r)}$ and its interior region; i.e. $\tempset_w$ is the complement of the infinite connected component in $(\mabox_{\bba(w,r)})^c$. Since $ \cregion \bigtriangleup \cregion_{,w} \subseteq B_{(\log n)^{C_1}}(w)$ and $A_w \supseteq B_r(w) \supseteq B_{(\log n)^{C_1}}(w)$,
\begin{equation}\label{eq-w-accessibility}
 \cregion \setminus \tempset_w=  \cregion_{,w} \setminus \tempset_w\,.
\end{equation}
That is, removing or adding obstacle at $w$ can only change the accessibility of the sites in $A_w$ (actually, only in $B_r(w)$) for the random walk. 
\begin{lemma}\label{lem-concentration-bound-w}
For $x, y$ with $|x-y| > n^{\fuu}$, let $u,v$ be chosen as in \eqref{eq-phi-phi-*}. On the event $E_{x, y}$ we have
\begin{align}
	   \varphi(u,v;\ob)\geq \varphi_*(x,y; \ob_w) - \log(2d)r^{3d} \quad  \forall w \in B_{2R_{\circ}}(x)\,.\label{eq:concentration-bound-w}
\end{align}
\end{lemma}
\begin{proof}
We divide the proof into  two cases as follows. 
\begin{figure}
	\includegraphics[width=1.8in]{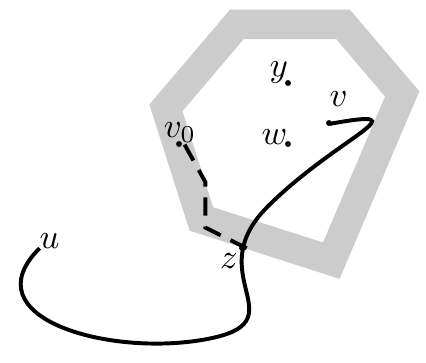}
	\includegraphics[width=1.8in]{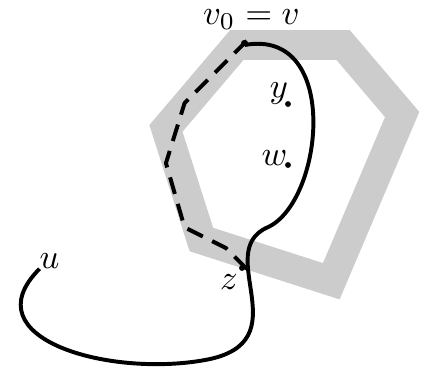}
	\includegraphics[width=2.2in]{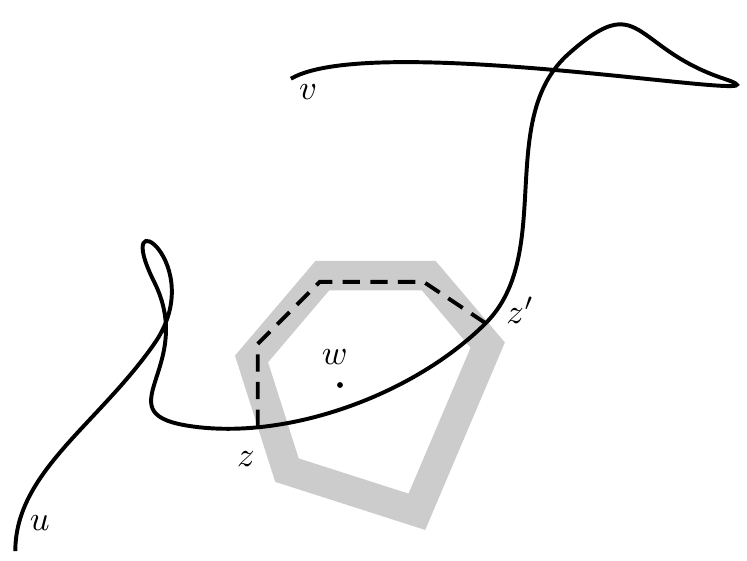}
	\begin{caption}
{From left to right: \underline{\textsc{Case 1}}(i), \underline{\textsc{Case 1}}(ii), \underline{\textsc{Case 2}}. The gray areas represent $\mabox_{\bba(w,r)}$'s. The solid curve are the original paths. We replace a segment of it by the dotted curve.}
\label{fig:Wx}
\end{caption}
\end{figure}

\noindent\underline{\textsc{Case 1:}} $u\in \tempset_w \cup \mac_{\bba(w,r)}$ or $v \in \tempset_w \cup \mac_{\bba(w,r)}$.

We assume that $v \in \tempset_w \cup \mac_{\bba(w,r)}$ (the case $u\in \tempset_w \cup \mac_{\bba(w,r)}$ can be treated similarly).  We first claim that there exists $v_0 \in \mac_{\bba(w,r)} \cap B_{r(x,y)}(y) $. To see this, we consider the following two scenarios.
\begin{enumerate}[(i)]
\item  $v \in \tempset_w \setminus \mabox_{\bba(w,r)}$: In this scenario, since 
$r(x,y) \geq 100r$, we know that $B_{r(x,y)}(y)  \cap \mabox_{\mathbf{i}}$ is non-empty for some $\mathbf{i} \in \bba(w,r)$. Then we choose an arbitrary $v_0$ in $\mabox_{\mathbf{i}} \cap \mac_{\bba(w,r)}$. Then Property 1 in Definition \ref{de-kwhite} gives $\calC(v_0) \geq L/10$ and the claim follows from Lemma \ref{cutset} (2).
\item $v \in \mabox_{\bba(w,r)}\cup \mac_{\bba(w,r)}$: In this scenario, choose $v_0=v$ and we only need to prove $v \in \mac_{\bba(w,r)}$ when $v \in \mabox_{\bba(w,r)}$. Since $v$ is connected to $u$ in $\ob^c$, by the first requirement of the event $E_{x, y}$ we see that $\calC(v) = \calC(\infty)$ and the desired claim follows from Lemma \ref{cutset} (2).
\end{enumerate}
Also, for the same reason as in (ii) above, we know that for any $z \in \partial_i (\tempset_w\cup \mac_{\bba(w,r)})$ with $G_{\cregion\setminus ((\tempset_w\cup \mac_{\bba(w,r)})\cup\{v\})}(u,z;\lambda) > 0$, we have $\calC(z) = \calC(\infty)$ and thus  by Lemma \ref{cutset} (2), we get $z \in \mac_{\bba(w,r)}$ (since $z \in \partial_i (\tempset_w\cup \mac_{\bba(w,r)})$). 

Next, we  will verify \eqref{eq:concentration-bound-w} in this case by combining the preceding discussions with the following decomposition of $\varphi$:
\begin{align}\label{eq-decomposition-phi}
	\varphi(u,v;\ob) = -\log \Big(\sum_{z \in \partial_i (\tempset_w\cup \mac_{\bba(w,r)})} G_{\cregion\setminus ((\tempset_w\cup \mac_{\bba(w,r)})\cup\{v\})}(u,z, \lambda)G_{\cregion\setminus  \{v\}}(z,v;\lambda)  \Big).
\end{align}
By Corollary \ref{G-polylogn}, $G_{\cregion \setminus \{v\}}(z,v;\lambda) \leq (\log n)^{C_3}$. Let $v_0$ be chosen as above and consider  any $z \in \partial_i (\tempset_w\cup \mac_{\bba(w,r)})$ with $G_{\cregion\setminus ((\tempset_w\cup \mac_{\bba(w,r)})\cup\{v\})}(u,z;\lambda) > 0$.  In light of scenario (i),(ii) and the discussion after (ii), we see that  $v_0,z \in \mac_{\bba(w,r)}$ and thus by Lemma \ref{cutset} (2) we get that $G_{\cregion_{,w} \setminus \{v\}}(v_0,z;\lambda) \geq (2d)^{-r^{2d}}$. Therefore, combined with \eqref{eq-decomposition-phi} it yields that $\varphi(u,v;\ob)$ is bounded below by
\begin{align*}
	 -\log \Big(&(2d)^{r^{2d}}(\log n)^{C_3}\cdot \sum_{z \in \partial_i (\tempset_w\cup \mac_{\bba(w,r)})} G_{\cregion\setminus ((\tempset_w\cup \mac_{\bba(w,r)})\cup\{v_0\})}(u,z;\lambda)G_{\cregion\setminus  \{v_0\}}(z,v_0;\lambda)  \Big)\\
	& = \varphi(u,v_0; \ob_w) - \log(2d)r^{2d} - C_3 \log\log n \geq \varphi_*(x,y; \ob_w) - \log(2d)r^{3d}\,,
\end{align*}
where we have used $\tempset_w\cup \mac_{\bba(w,r)}\cup\{v_0\} = \tempset_w\cup \mac_{\bba(w,r)}$. 

\smallskip

\noindent\underline{\textsc{Case 2:}} Neither $u$ nor $v$ is in $\tempset_w \cup \mac_{\bba(w,r)}$.

Let $I_{z,z'} = G_{\cregion\setminus (A_w \cup \{v\})}(u,z;\lambda) G_{\cregion\setminus (A_w \cup \{v\})}(z',v;\lambda)$, we have
\begin{align*}
G_{\cregion\setminus \{v\}}(u,v;\lambda) &= \sum_{z,z' \in \partial_i A_w}  I_{z,z'} \cdot G_{\cregion\setminus \{v\}}(z,z';\lambda)+G_{\cregion\setminus (A_w \cup \{v\})}(u,v;\lambda)\,, \\
G_{\cregion_{,w}\setminus \{v\}}(u,v;\lambda) &= \sum_{z,z' \in \partial_i A_w}  I_{z,z'} \cdot G_{\cregion_{,w}\setminus \{v\}}(z,z';\lambda)+G_{\cregion\setminus (A_w \cup \{v\})}(u,v;\lambda) \,.
\end{align*} 
Wen claim that that for $z,z' \in \partial_i A_w$ with $I_{z,z'}>0$,
\begin{equation}\label{eq-ratio-z-z'}
G_{\cregion\setminus \{v\}}(z,z';\lambda) /G_{\cregion_{, w}\setminus \{v\}}(z,z';\lambda)  \leq  (2d)^{r^{3d}}\,.
\end{equation}
Provided with \eqref{eq-ratio-z-z'}, we see that $G_{\cregion\setminus \{v\}}(u,v;\lambda)/G_{\cregion_{,w}\setminus \{v\}}(u,v;\lambda)  \leq (2d)^{r^{3d}}$, which then yields \eqref{eq:concentration-bound-w}. It remains to prove \eqref{eq-ratio-z-z'}.
We first note that $G_{\cregion}(z,z';\lambda) \leq (\log n)^{C_3}$ due to Corollary \ref{G-polylogn}. 
At the same time, for the same reason as in {\underline{Case 1}}, we know $I_{z,z'}>0$ implies $\calC(z)=\calC(z')=\calC(\infty)$. Combining this with $z,z' \in \partial_i A_w$ and Lemma \ref{cutset} (2) gives $z,z' \in \mac_{\bba(w,r)}$. Then Lemma \ref{cutset} (2) implies $G_{\mac_{\bba(w,r)}}(z,z';\lambda) \geq (2d)^{-r^{2d}}$. Thus \eqref{eq-ratio-z-z'} would follow once we prove $\mac_{\bba(w,r)} \subseteq \cregion_{,w} \setminus \{v\}$. To this end, note that $\mac_{\bba(w,r)} \subseteq \cregion\setminus \{v\}$ (by Definition \ref{de-kwhite} and the assumption of this case) and that $\mac_{\bba(w,r)} \cap B_{r - 100L}(w) = \varnothing$ (by the definition of $\mac_{\bba(w,r)}$ in Lemma \ref{cutset}). Since $\cregion \bigtriangleup \cregion_{,w} \subseteq B_{(\log n)^4}(w) \subseteq B_{r - 100L}(w)$, it yields that $ \mac_{\bba(w,r)} \subseteq \cregion_{,w} \setminus \{v\}$. At this point, we complete the verification of  \eqref{eq-ratio-z-z'}.

Combining the two cases above completes the proof of the lemma.
\end{proof}
 

We remark that  \eqref{eq:concentration-bound-w} holds for general $w$ provided that the event $E_{x, y}$ occurs, but it is suboptimal for typical $w$. For a typical $w$, the influence of resampling the environment at $w$ is much smaller than the bound given in \eqref{eq:concentration-bound-w}, as incorporated in the next lemma.
\begin{lemma}\label{lem-concentration-3-2}
For $x, y$ with $|x-y| > n^{\fuu}$, let $u,v$ be chosen as in \eqref{eq-phi-phi-*}.  For $w \in B_{2R_{\circ}}(x)$, we have
\begin{equation}
\label{eq:concentration-3-2}
   \frac{\Ex^{u} \left[ \lambda^{-\tau_{v}} ;{\xi_{\cregion}>\tau_{v}} ,{\tau_{\tempset_w} \leq \tau_{v}} \right] }{\Ex^{u} \left[ \lambda^{-\tau_{v}} ;{\xi_{\cregion}>\tau_{v}}\right] }\geq \frac{1}{2} \min(1, \varphi(u,v;\ob_w) -\varphi(u,v;\ob))\,.
 \end{equation} 
 \end{lemma}
\begin{proof}
We note that 
\begin{align*}
\varphi(u,v;\ob_w) &= -\log \left(\Ex^{u} \left[ \lambda^{-\tau_{v}} ;{\xi_{\cregion_w}>\tau_{v}} ,{\tau_{\tempset_w} \leq \tau_{v}} \right] +\Ex^{u} \left[ \lambda^{-\tau_{v}} ;{\xi_{\cregion}>\tau_{v}}  ,{\tau_{\tempset_w} > \tau_{v}}\right]  \right)\\
& \leq -\log \left(\Ex^{u} \left[ \lambda^{-\tau_{v}} ;{\xi_{\cregion}>\tau_{v}}  ,{\tau_{\tempset_w} > \tau_{v}}\right]  \right)\,.
\end{align*}
Thus, we have
\begin{equation*}
  \varphi(u,v;\ob_w) -\varphi(u,v;\ob) 
   \leq - \log \left(\frac{\Ex^{u} \left[ \lambda^{-\tau_{v}} ;{\xi_{\cregion}>\tau_{v}} ,{\tau_{
  \tempset_w} > \tau_{v}} \right] }{\Ex^{u} \left[ \lambda^{-\tau_{v}} ;{\xi_{\cregion}>\tau_{v}}\right] } \right)\,.
\end{equation*}
Therefore, using $ - \log t \leq (1 - t)$ for $1/2 \leq t \leq 1$ gives \eqref{eq:concentration-3-2}.
\end{proof}

\begin{proof}[Proof of Lemma~\ref{concentration}]
It follows from \eqref{eq-phi-phi-*}, \eqref{eq:phi-def} and the third requirement of $E_{x,y}$ that
\begin{align}
\label{eq:con-bound-w-0}
     \big(\big(\bar \varphi_*(x,y;\ob_w) - \bar\varphi_*(x,y;\ob) \big)_+\big)^2& \leq \left(\left(\varphi(u,v;\ob_w) - \varphi(u,v;\ob)\right)_+\right)^2  \nonumber \\
    & \leq(\log(2d))^2r^{6d}
 \times \frac{2\Ex^{u} \left[ \lambda^{-\tau_{v}} ;{\xi_{\cregion}>\tau_{v}} ,{\tau_{\tempset_w} \leq \tau_{v}} \right] }{\Ex^{u} \left[ \lambda^{-\tau_{v}} ;{\xi_{\cregion}>\tau_{v}}\right] } \nonumber\\
  & \leq Cr^{6d} \sum_{\mathclap{z \in B_{2r}(w)}} \ \frac{\Ex^{u} \left[ \lambda^{-\tau_{v}} ;{\xi_{\cregion}>\tau_{v}} ,{\tau_{z} \leq \tau_{v}} \right] }{\Ex^{u} \left[ \lambda^{-\tau_{v}} ;{\xi_{\cregion}>\tau_{v}}\right] }\,,
\end{align}
where in the second inequality, we used Lemmas~\ref{lem-concentration-bound-w} and \ref{lem-concentration-3-2}. Now, we define
\begin{equation}
	V_+ = \Big(\big(\bar \varphi_*(x,y;\ob_\circ) - \bar\varphi_*(x,y;\ob)   \big)_+\Big)^2 + \sum_{w \in B_{2R_{\circ}}(x)} \Big(\big(\bar \varphi_*(x,y;\ob_w) - \bar\varphi_*(x,y;\ob)   \big)_+\Big)^2\,.
\end{equation}
Combining \eqref{eq:concentration-bound-circ} and \eqref{eq:con-bound-w-0}, we have that on event $E_{x, y}$ 
\begin{align*}
	V_+ &\leq \sum_{w \in B_{2R_{\circ}}(x)}Cr^{6d} \sum_{z \in B_{2r}(w)}\frac{\Ex^{u} \left[ \lambda^{-\tau_{v}} ;{\xi_{\cregion}>\tau_{v}} ,{\tau_{z} \leq \tau_{v}} \right] }{\Ex^{u} \left[ \lambda^{-\tau_{v}} ;{\xi_{\cregion}>\tau_{v}}\right] }  + 1\\
	&\leq Cr^{7d}\sum_{z \in \Z^d} \frac{\Ex^{u} \left[ \lambda^{-\tau_{v}} ;{\xi_{\cregion}>\tau_{v}} ,{\tau_{z} \leq \tau_{v}} \right] }{\Ex^{u} \left[ \lambda^{-\tau_{v}} ;{\xi_{\cregion}>\tau_{v}}\right] } + 1\leq Cr^{7d}\frac{\Ex^{u} \left[ \tau_v\lambda^{-\tau_{v}} ;{\xi_{\cregion}>\tau_{v}} \right] }{\Ex^{u} \left[ \lambda^{-\tau_{v}} ;{\xi_{\cregion}>\tau_{v}}\right] } \,.
\end{align*}
Combining with \eqref{eq:ex-taulambda-tau}, we get that on the event $E_{x, y}$
$$ V_+   \leq  (\log |x-y|)^{C}|x-y|\,.  $$
Combined with \eqref{eq:con-e-prob} and $|\bar\varphi_*(x,y)|  \leq \Cbar|x-y|$, this yields for any $q\geq2$, and $|x-y| \geq n^\fuu$
\begin{equation*}
	\E\left[V_+^{q/2}\right] \leq \Big( (\log |x-y|)^{C} |x-y| \Big)^{q/2}\,.
\end{equation*}
In light of the preceding estimate, we complete the proof of the lemma by applying Efron-Stein inequality and \cite[Theorem 2]{BBLM05}.
\end{proof}

Next, we prove the sub-additivity of our LWGF. Note that the sub-additivity for \emph{unweighted} Green's function (i.e., when $\lambda = 1$) was proved in \cite[Chapter 5, Lemma 2.1]{Sznitman98}. The proof for our LWGF shares the same spirit but is a bit more complicated.
\begin{lemma}
	\label{sub-additivity} 
	For all $x,y,z \in \Z^d$ such that $|x-y| \geq n^{\fuu}$ and all $\lambda \in [\lambda_*,1]$ 
	\begin{equation}
	\label{eq:sub-additivity}
		\E \bar \varphi_*(x,y) \leq \E\bar\varphi_*(x,z) +\E \bar\varphi_*(z,y)  + e^{(\log|x-y|)^{2/3}}\,.
	\end{equation}
\end{lemma}
\begin{proof}
For any $u,v,w \in \Z^d$, by strong Markov property at $\tau_v$, we have
$$ G_{\cregion \setminus \{w\}}(u,w;\lambda) \geq G_{\cregion \setminus \{v,w\}}(u,v;\lambda)  G_{\cregion \setminus \{w\}}(v,w;\lambda)\,.  $$
Then \eqref{eq:phi=logG} and Lemma \ref{G-polylogn} implies
	\begin{equation}
	\label{eq:add-uvw}
		\varphi(u,w) \leq \varphi(u,v) + \varphi(v,w) + C_3 \log\log n\,.
	\end{equation}

To prove \eqref{eq:sub-additivity}, we have the complication that $\varphi_*(x,y)$ is the minimum over a neighborhood of $x$ to a neighborhood of $y$ (see \eqref{eq:phi-def}) and the size of the neighborhood depends on $|x-y|$. This requires a more careful treatment in the proof.

Without lose of generality, we assume $|x-z| \geq |y-z|$. By the definition of $\bar \varphi_*$, it suffices to consider the case when $|x-y| \geq \cul\Cbar^{-1}|x-z|$. The proof divides into the following two cases.
\begin{figure}
	\includegraphics[width=3in]{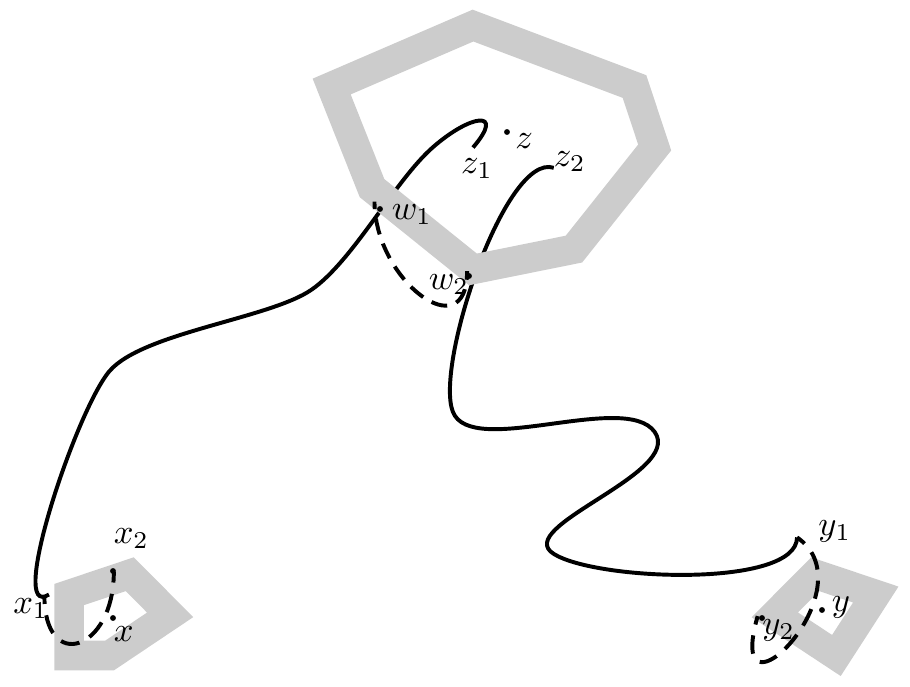}
	\includegraphics[width=2.2in]{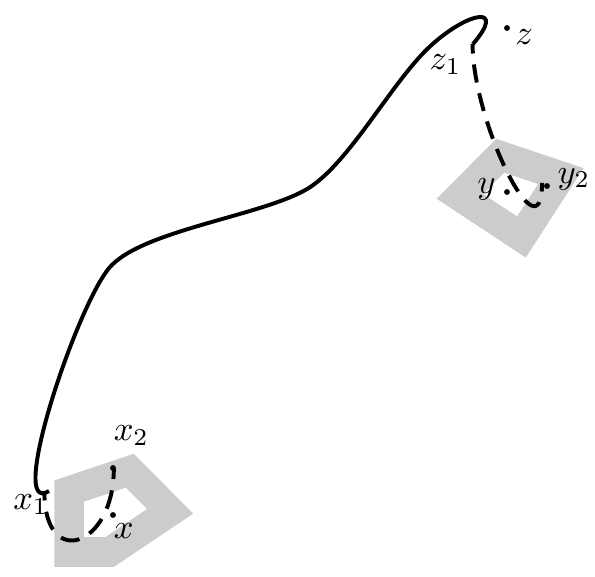}
	\begin{caption}
{Left: \textsc{Case 1}. Right: \textsc{Case 2}. The gray areas represent $\mabox_{\bba}$'s. The solid curve are the original paths. We use dotted curves to connect two paths or change the starting/ending points.}
\label{fig:add}
\end{caption}
\end{figure}
\medskip

\noindent\underline{\textsc{Case 1:}}  $|y-z| \geq e^{(\log |x-y|)^{1/2}}$.

In this case, we let $E_{x, y, z}$ be the event such that the following holds:
\begin{itemize}
	\item For any $ v \in B_{|x-y|^2}(x)$, either $\mathcal C (v) = \mathcal C(\infty)$ or $|\mathcal C (v)| \leq (\log |x-y|)^{5}$.
	\item $\bar \varphi_*(x,y) = \varphi_*(x,y),\bar \varphi_*(x,z) = \varphi_*(x,z)$ and $\bar \varphi_*(y,z) = \varphi_*(y,z)$,
	\item $\bba(x,r(x,y)/2),\bba(y,r(x,y)/2),\bba(z,2r(x,z))$ are not empty.
	\item Properties described in Lemma \ref{chemdist2}.
\end{itemize}
By Lemmas \ref{FiniteClusterSize}, \ref{cutset}, \eqref{eq:phidev} and \ref{chemdist2},
\begin{align}
\label{eq:tri-pro-1}
	\P(E_{x, y, z}^c) &\leq (2|x-y|)^de^{-c(\log |x-y|^{5/2})}+3e^{-(\log|x-y|)^2}+e^{-(\log|x-z|)^3}+e^{-(\log|z-y|)^3} \nonumber\\ &+ e^{-2^{-1}r(x,y) (\log n)^{-\cgpt}} + e^{-2r(x,z) (\log n)^{-\cgpt}} \leq 2 e^{-(\log|x-y|)^{3/2}}\,.
\end{align}
On event $E_{x, y,z}$, by \eqref{eq:phi-def} we choose $x_1 \in B_{r(x,z)}(x) $ and 
$z_1 \in B_{r(x,z)}(z) $, $z_2 \in B_{r(z,y)}(z)$ and 
$y_1 \in B_{r(z,y)}(y) $ such that 
$$\bar\varphi_*(x,z) = \varphi_*(x,z) = \varphi(x_1,z_1) \quad \text{and} \quad\bar \varphi_*(z,y)= \varphi_*(z,y) = \varphi(z_2,y_1).$$
Noting that $\bba(z,2r(x,z))$ is not empty, we denote by $\tempset$ the union of $ \mabox_{\bba(z,2r(x,z))}$ and its interior region; i.e., $\tempset$ is the complement of the infinite connected component in $ \mabox_{\bba(z,2r(x,z))}^c$. By decomposition of random walk paths, we have
$$G_{\cregion}(x_1,y_1;\lambda) \geq  \sum_{w_1,w_2 \in \partial_i \tempset}G_{\cregion \setminus \tempset}(x_1,w_1;\lambda)G_{\cregion }(w_1,w_2;\lambda)G_{\cregion \setminus \tempset}(w_2,y_1;\lambda)\,.$$
Also, for $w_1,w_2$ such that $G_{\cregion \setminus\tempset}(x_1,w_1;\lambda)G_{\cregion \setminus\tempset}(w_2,z_1;\lambda)>0$, by the first requirement of $E_{x, y, z}$, we have $\calC(w_1) = \calC(w_2)=\calC(\infty)$. Applying Lemma \ref{cutset} (2), we get
$$G_{\cregion }(w_1,w_2;\lambda) 
\geq (2d)^{-(2r(x,z))^{2d}}\,.$$
Therefore, combining the preceding two displayed inequalities gives
\begin{align*}
	G_{\cregion}(x_1,y_1;\lambda) \geq (2d)^{-(2r(x,z))^{2d}} \cdot \sum_{w_1,w_2 \in \partial_i \tempset}G_{\cregion \setminus\tempset}(x_1,w_1;\lambda)G_{\cregion \setminus\tempset}(w_2,y_1;\lambda)\,.
 \end{align*}
Also, by Corollary \ref{G-polylogn}, 
 \begin{align*}
G_{\cregion}(x_1,z_1;\lambda) & =  \sum_{w\in \partial_i \tempset}G_{\cregion \setminus\tempset}(x_1,w;\lambda)G_\cregion(w,z_1;\lambda) \leq (\log n)^{C_3} \sum_{w\in \partial_i \tempset}G_{\cregion \setminus\tempset}(x_1,w;\lambda)\,,
\end{align*}
and similarly $G_{\cregion}(z_2,y_1;\lambda) \leq (\log n)^{C_3} \sum_{w\in \partial_i \tempset}G_{\cregion \setminus  \tempset}(w,z_1;\lambda)$.
Combining preceding three inequalities and that $r(x,z) \geq \log (|x-y|/2) \geq c\log n$ gives
$$ G_{\cregion}(x_1,y_1;\lambda) \geq e^{-C(r(x,z))^{2d}}G_{\cregion}(x_1,z_1;\lambda) G_{\cregion}(z_2,y_1;\lambda)\,.$$
Then by \eqref{eq:phi=logG} and Corollary \ref{G-polylogn}, we deduce that
$$ \varphi(x_1,y_1) \leq \bar \varphi_*(x,z) + \bar \varphi_*(z,y)  + C(r(x,z))^{2d}\,.$$
Since $\bba(x,r(x,y)/2),\bba(y,r(x,y)/2) \not = \varnothing$  and $r(x,y)/2>(\log |x-y|)^{C_2}$, we choose $x_2 \in \mac_{\bba(x,r(x,y)/2)}$ and $y_2 \in \mac_{\bba(y,r(x,y)/2)} $ arbitrarily. Then by Lemma \ref{chemdist2}, $x_1$ and $x_2$ is connected by path in $\cregion$ of length at most $(\log |x-y|)^{C}$. Then $\varphi(x_1,x_2) \leq \log (2d)(\log |x-y|)^C$ (and the same holds for $y_1$ and $y_2$).
Therefore, on the event $E_{x, y, z}$, \eqref{eq:add-uvw} gives
$$ \varphi_*(x,y) \leq \varphi (x_2, y_2) \leq \varphi(x_1, y_1) +2 \log (2d) (\log |x-y|)^{C}  \leq 
\bar \varphi_*(x,z) + \bar \varphi_*(z,y)  + (\log |x-y|)^{C}\,.$$
Combined with \eqref{eq:tri-pro-1} and $\bar \varphi_*(x,y)  \leq \Cbar|x-y|$, this implies
$$ \E\bar\varphi_*(x,y)   \leq \E\bar\varphi_*(x,z) +\E \bar\varphi_*(z,y)  + (\log |x-y|)^{C}\,.$$

\noindent\underline{\textsc{Case 2}:} $|y-z| \leq e^{(\log |x-y|)^{1/2}}$. 

In this case, $|x-z| \geq |x-y|/2$. We let $E_{x,y, z}$ be the event such that the following hold:
\begin{itemize}
	\item For any $ v \in B_{|x-y|^2}(x)$, either $\mathcal C (v) = \mathcal C(\infty)$ or $|\mathcal C (v)| \leq (\log |x-y|)^{5}$.
	\item $\bar \varphi_*(x,y) = \varphi_*(x,y)$ and $\bar \varphi_*(x,z) = \varphi_*(x,z)$,
	\item $\bba(x,r(x,y)/2),\bba(y,r(x,y)/2)$ are not empty.
	\item Properties described in Lemma \ref{chemdist2}.
\end{itemize}
By Lemmas \ref{FiniteClusterSize}, \ref{cutset}, \eqref{eq:phidev} and \ref{chemdist2},
\begin{align}
\label{eq:tri-prob-2}
	\P(E_{x, y, z}^c) &\leq(2|x-y|)^de^{-c(\log |x-y|^{5/2})} +  3e^{-(\log|x-y|)^2}\nonumber\\
&+2e^{-(\log|x-z|)^2} + 2e^{-2^{-1}r(x,y) (\log n)^{-\cgpt}} \leq e^{-(\log|x-y|)^{3/2}}\,.
\end{align}
On the event $E_{x, y, z}$, by \eqref{eq:phi-def} we let $x_1 \in B_{r(x,z)}(x) $ and $z_1 \in B_{r(x,z)}(z) $ such that
$$\bar \varphi_*(x,z)= \varphi_*(x,z) = \varphi(x_1,z_1)\,.$$
Since $\bba(x,r(x,y)/2),\bba(y,r(x,y)/2)$ are not empty, by Lemma \ref{chemdist2}, there exists $x_2 \in B_{r(x,y)}(x) $ and $y_2 \in B_{r(x,y)}(y)  $ such that  $x_1$ is connected  to $x_2$ by path in $\cregion$ of length at most $(\log |x-y|)^{C}$ and $z_1$ is connected to $y_2$  by path in $\cregion$ of length at most $C e^{-(\log|x-y|)^{1/2}}$. Then $\varphi(x_2,x_1) \leq \log (2d)(\log |x-y|)^C$, $\varphi(z_1,y_2) \leq C\log (2d) e^{-(\log|x-y|)^{1/2}}$.
Therefore, on the event $E_{x, y, z}$, \eqref{eq:add-uvw} gives
\begin{align*}
 \varphi_*(x,y)& \leq \varphi(x_2,y_2)\leq \varphi(x_1,z_1)+ C e^{(\log|x-y|)^{1/2}}\log(2d) \\
 &=  \varphi_*(x,z)  + C e^{(\log|x-y|)^{1/2}}\log(2d)\,.
 \end{align*}
Combined with \eqref{eq:tri-prob-2} and $\bar \varphi_*(x,y)  \leq \Cbar|x-y|$, this implies
\begin{equation*}
 \E\bar\varphi_*(x,y) \leq \E\bar\varphi_*(x,z) +\E \bar\varphi_*(z,y)  + e^{2(\log|x-y|)^{1/2}}\,. \qedhere
 \end{equation*}
\end{proof}

\subsection{Rate of convergence for LWGF to linear functions}
\label{section:Convergence rate}
Let \begin{equation}
\label{eq:def-h}
	h(x) = \E \left[\bar \varphi_*(0,x)\right]\text{ for all }x \in \Z^d\,.
\end{equation}It has been proved in Lemma \ref{sub-additivity} that $\{h(mx)\}_{m \geq 1}$ is sub-additive (with some error term). Hence, one can prove that (see \cite[Theorem 23]{BE52}) for all $x \in \Z^d$ the limit
\begin{equation}
\label{eq:g-def-lim}
	g(x): = \lim_{m \to \infty} h(mx) /m 
\end{equation}
exists. Furthermore, $g$ could first extends to $\Q^d$ by restricting $m$ to such that $mx \in \Z^d$, and then extends to $\R^d$ by continuity (see \cite[Lemma 1.5]{Alexander97}). Moreover, it follows directly from definition and Lemma \ref{sub-additivity} that $g$ is homogeneous of order 1 and sub-additive, i.e.,
\begin{align}\label{eq:gsuba}
	g(\alpha x) = \alpha g(x) \mbox{ for }x\in \R^d, \alpha>0\,, \mbox{ and } g(x+y) \leq g(x) + g(y) \mbox{ for }x,y\in \R^d\,.
\end{align}
Thus, $g$ is convex (such properties of $g$ are useful, e.g., in the proof of Lemma \ref{topev-2}). Now, we can state the main conclusion of this subsection.
\begin{lemma}
\label{converrate}
For $|x| \geq n^{1/2}$, we have that
\begin{equation}
\label{eq:happ}
	h(x) \leq g(x) + |x|^{4/5}\,.
\end{equation}
\end{lemma}
The proof of Lemma~\ref{converrate} is a combination of the concentration results proved in Section~\ref{section: Concentration, Sub-additivity} and a nontrivial adaption of arguments in \cite{Alexander97}. We first record some straightforward properties of $h$. By Lemma \ref{sub-additivity} and \cite[Theorem 23]{BE52}, for all $|x| \geq n^{\fuu}$
\begin{equation}
\label{eq:gandh}
	g(x) \leq  h(x) + |x|^{1/10}\,,
\end{equation}
 \begin{equation}
\label{eq:function-g-bound}
	\cul|x|\leq g(x) \leq  \Cbar|x|\,,
\end{equation}
where $\cul, \Cbar$ are two constants as in \eqref{eq:def-barphi}  In addition, $g$ is continuous (see \cite[Lemma 1.5]{Alexander97}).

One of the main results in \cite{Alexander97} is that CHAP (as described in Lemma \ref{Alexlemma}) implies approximation property of the form \eqref{eq:happ}. Hence, it suffice to verify the CHAP condition. We will adapt the arguments in \cite{Alexander97} in order to verify CHAP (Lemma \ref{mleqan}). However, in our case $h$ is not the length of the shortest path (as in the case for the first-passage percolation), but in a vague sense $h$ is the ``average" length of all open paths (which are not necessarily self-avoiding). This incurs some nontrivial challenges, whose treatment requires some new technical ingredients including Lemma \ref{vpath} and some concentration results in our case (Lemmas \ref{good--v-prob}, \ref{concentration-path}).

We point out that Lemmas \ref{Alexlemma}, \ref{QDD} and the arguments in \textit{the proof of Lemma \ref{converrate}} are from \cite{Alexander97}. We omit the proof the Lemma \ref{Alexlemma} but present the proof the other two, as there is an error term in the sub-additivity of $h$ and hence a few (straightforward) modifications are required there. Also, Lemma \ref{mleqan} is similar to \cite[Proposition 3.4]{Alexander97}.

Now, we turn to the proof of the main result in this subsection. 
\begin{defn}\label{def-gx}
	For $x\in\R^d$, let $H_x$ denote a hyperplane tangent to $\partial \{y:g(y) \leq g(x)\}$ at $x$. Let $H_x^0$ denote the hyperplane through $0$ parallel to $H_x$. We define $ g_x(y)$ by the unique linear function such that
\begin{equation*}
	g_x(y) =0 \for y \in H_x^0, \AND g_x(x)=g(x)\,,
\end{equation*}
and define
\begin{equation}
\label{eq:def-Qx}
	Q_x: = \{y\in\Z^d:g_x(y) \leq g(x),h(y) \leq g_x(y) + 10C_Q|x|^{\tuu}\}\,,
\end{equation}
where $C_Q = 2\Cbar\cul^{-1} + 10.$
\end{defn}
Then one can prove that (see \cite[(1.9)]{Alexander97})
\begin{equation}
\label{eq:gx<g}
	 |g_x(y)| \leq g(y)\text{ for all }y \in \R^d\,.
\end{equation}
\begin{lemma}\emph{(\cite[Lemma 1.6]{Alexander97})} 
\label{Alexlemma}
		Let $M,a>1$. Suppose that for each $x \in \mathbb Q^d$ with $|x| \geq M$, there exists $N \geq 1$, a path $\gamma$ in $\Z^d$ from $0$ to $Nx$ and a sequence of sites $0 = v_0,v_1,...,v_m = Nx$ in $\gamma$ such that $m \leq aN$ and $v_i - v_{i-1} \in Q_x$ for all $1 \leq i \leq m$. Then $h$ satisfies the convex-hull approximation property (CHAP), meaning for all $x \in \mathbb Q^d$ with $|x| \geq M$, we have
		\begin{equation}
			x / \alpha \in \mathrm{Co}(Q_x)\quad \text{for some }\alpha \in [1,a]\,,
		\end{equation}
		where $\mathrm{Co}(\cdot)$ denotes the convex hull.
\end{lemma}

To verify the condition in Lemma \ref{Alexlemma}, we fix a large integer $N$ such that $Nx \in \Z^d$ and choose $v_i$'s depending on the environment and show that they satisfies the condition in Lemma \ref{Alexlemma} with positive $\P$-probability (which suffices as the existence of desired $v_i$'s as stated in Lemma~\ref{Alexlemma} is a deterministic event). 

We choose $u_1 \in B_{r(0,Nx)}(0) $ and $u_2 \in B_{r(0,Nx)}(Nx)$ such that $\varphi_*(0,Nx) = \varphi(u_1,u_2)$. Then we choose $v_i$ for $i = 1,2,...,\zeta_1,...,\zeta_2,...,\zeta_3$ as well as $\zeta_1, \zeta_2, \zeta_3$ inductively as follows. Set $v_0 = 0$ and for any $i \geq 1$

	\begin{equation}
	\label{eq:defin-v-i-1}
		v_i = \arg\min_{u :~u -  v_{i-1}\in\partial_iQ_x} |u - u_1|
	\end{equation}
until $u_1 \in v_{i-1}+ Q_x$, in which case we set $v_i = u_1$ and $\zeta_1 = i$. Next for $i \geq \zeta_1 + 1$ set
	\begin{equation}
		\label{eq:defin-v-i-2}
		v_i = \arg\max_{\substack{u :~u -  v_{i-1}\in\partial_iQ_x\\ u \not\in \{v_{\zeta_1},v_{\zeta_1+1},...,v_{i-1}\}} }G_{\cregion \setminus \{v_{\zeta_1},v_{\zeta_1+1},...,v_{i-1}\}}(v_{i-1},u;\lambda) \cdot G_{\cregion \setminus \{v_{\zeta_1},v_{\zeta_1+1},...,u\}}(u,u_2;\lambda) \,,
	\end{equation}
until $u_2 \in v_{i-1}+ Q_x$, in which case we set $v_i = u_2$ and $\zeta_2 = i$. Finally for $i \geq \zeta_2 + 1$ set
		\begin{equation}
			\label{eq:defin-v-i-3}
		v_i = \arg\min_{u :~u -  v_{i-1}\in\partial_iQ_x} |u - u_3|
	\end{equation}
until $Nx \in v_{i-1}+ Q_x$, in which case we set $v_i = Nx$ and $\zeta_3 = i$. We remark that we do not require either $0$ or $Nx$ is in $\calC(\infty)$ here.

We verify the condition in Lemma \ref{Alexlemma} for $M = n^{1/3}$ and $a = 4$ as a consequence of the following lemma, and then prove Lemma \ref{converrate}. 
\begin{lemma}
\label{mleqan}
For all $x \in \Q^d$ with $|x| \geq n^{1/3}$, there exists an integer $N \geq 1$ such that $$\P( \zeta_3 \leq 4N)>0\,.$$
\end{lemma}
\begin{proof}[Proof of Lemma \ref{converrate}]
 We recall that the conditions in Lemma \ref{Alexlemma} are on $h$, which is deterministic. By choosing $v_i$'s depending on environment as in \eqref{eq:defin-v-i-1}-\eqref{eq:defin-v-i-3}, in light of Lemma \ref{mleqan}, we know $v_i$ satisfies conditions in Lemma \ref{Alexlemma} for $M = n^{1/3}$ and $a = 4$ with positive $\P$-probability. Hence the conditions in Lemma \ref{Alexlemma} hold, because it only requires the existence of such $v_i$'s. Now, choose $q \in [|x|^{\frac{1}{4}}/2,|x|^{\frac{1}{4}}] \cap \Q$. Since $|x/q| \geq |x|^{\frac{3}{4}} \geq n^{1/3}$, we use Lemma \ref{Alexlemma} (with $x/q$ in replace of $x$) and Carath\'{e}odory's theorem (which states that every points in a convex hull is a convex combination of at most $d+1$ fixed points in this convex hull. See \cite[Theorem 3.10]{Cunha13}.) to conclude that
	\begin{equation}
		x/q =\alpha_{1}y_{1} + \alpha_{2}y_{2} + \ldots + \alpha_{d+1}y_{d+1}
	\end{equation}
with $\alpha_{i} \geq 0, \sum_{i = 1}^{d+1}\alpha_{i} \in [1,4]$ and $y_{i} \in Q_{x/q}.$
Let $ x^* = \sum_{i = 1}^{d+1}\lfloor q\alpha_{i} \rfloor y_{i}$.
Note that $y_{i} \in Q_{x/q}, g_{x/q} = g_x$ (since $g_x$ only depends on $x/|x|$). By Lemma \ref{sub-additivity},
\begin{align*}
	h(x^*) &\leq \sum_{i = 1}^{d+1}\lfloor q\alpha_{i} \rfloor (h(y_{i}) +e^{(\log 4|x|)^{2/3}})\\
	&\leq \sum_{i = 1}^{d+1}\lfloor q\alpha_{i} \rfloor \big( g_x(y_{i}) + 10 C_Q|x/q|^\tuu\big) + 4(d+1)qe^{(\log 4|x|)^{2/3}}\\
	&\leq g_x(x^*) +40C_Qq|x/q|^\tuu+ 4(d+1)qe^{(\log 4|x|)^{2/3}}\,,
\end{align*}
where in the last step we used that $g_x(\cdot)$ is a linear function.
In addition, since $|x - x^*| \leq \sum_{i=1}^{d+1}|y_i| \leq (d+1)C_Q |x/q|$, by \eqref{eq:gx<g} and \eqref{eq:function-g-bound}, $$ h(x-x^*) - g_x(x - x^*) \leq 2\Cbar C_Q(d+1)|x/q|\,.$$
Combining the preceding two inequalities, we conclude that
\begin{align*}
	h(x) &\leq g(x) + 40C_Qq|x/q|^\tuu+ 4(d+1)qe^{(\log 4|x|)^{2/3}} +2\Cbar C_Q(d+1)|x/q|\\
		& \leq g(x) + C|x|^{3/4} \,. \qedhere
\end{align*}
\end{proof}

The rest of this subsection is devoted to the proof of Lemma \ref{mleqan}.
We need a few more definitions: For $y \in \Z^d$, we define 
\begin{equation}
\label{eq:def-alex}
\begin{split}
	s_x(y)&:=h(y) - g_x(y),\\
	G_x&:=\{ y \in \Z^d : g_x(y)>g(x) \},\\
	\Delta_x&:=\{ y \in Q_x : y\text{ adjacent to } \Z^d \setminus Q_x, y\text{ not adjacent to } G_x \},\\
	D_x&:=\{ y \in Q_x : y\text{ adjacent to } G_x \}.
\end{split}
\end{equation}

\begin{lemma}\emph{(\cite[Lemma 3.3]{Alexander97})}
\label{QDD}
Recall that $C_Q = 2\Cbar\cul^{-1} + 10$. For all $x \in \Q^d$ with $|x| \geq n^{1/3}$
	\begin{itemize}
		\item [(1)]If $y \in Q_x$, then $g(y) \leq 2g(x)$, $|y| \leq C_Q|x|$ and $s_x(y) \geq -|y|^{1/10}$.
		\item [(2)]If $y \in \Delta_x$, then $s_x(y) \geq 8C_Q|x|^{\tuu}$.
		\item [(3)]If $y \in D_x$, then $g_x(y) \geq 5g(x)/6$.
	\end{itemize}
\end{lemma}

\begin{proof}
	\noindent (1) For $y\in Q_x$, we have $s_x(y) \leq 10C_Q|x|^\tuu$ and $g_x(y) \leq g(x)$. Then by \eqref{eq:gandh}
	$$g(y) \leq h(y) + |y|^{1/10} = g_x(y) + s_x(y) + |y|^{1/10} \leq g(x) + 10C_Q|x|^\tuu + |y|^{1/10}\,.$$
	Hence, by \eqref{eq:function-g-bound}, we have $g(y) \leq 2g(x)$ for large $x$ and $y$. Then by \eqref{eq:function-g-bound}, $|y| \leq 2\Cbar\cul^{-1}|x|<C_Q|x|$.
Also, \eqref{eq:gandh} implies $s_x(y) \geq -|y|^{1/10}$.

	\noindent(2) If $y \in \Delta_x$, then $y = z + e$ for some $z \in \Z^d  \setminus (Q_x \cup G_x)$ and $e \in \{\pm e_1,...,\pm e_d\}$ where $\{e_i\}_{i=1}^d$ is the standard basis in $\R^d$. Since $s_x(z) \geq 10C_Q|x|^{\tuu}$ and $|s_x(e)| \leq 2\Cbar$, using sub-additivity of $h$ (Lemma \ref{sub-additivity}), linearity of $g_x$  and $|z| \leq |y| +1 \leq2C_Q|x|$ (which follows from $y\in \Delta_x\subseteq Q_x$ and the first item of the current lemma) we get
	\begin{align*}
		s_x(y) &= h(y) - g_x(y) \geq h(z) - h(e) -e^{(\log |z|)^{2/3}}- (g_x(z) + g_x(e))\\
		 &= s_x(z)-s_x(e) -e^{(\log |z|)^{2/3}} \geq 8C_Q|x|^{\tuu}\,.
	\end{align*}
	\noindent (3) If $y \in D_x$, then $y = z + e$ for some $z \in \Z^d \cap G_x$ and $e \in \{\pm e_1,...,\pm e_d\}$. Hence
	\begin{equation*}
		 g_x(y) = g_x(z) + g_x(e) \geq g(x) - \Cbar \geq 5g(x)/6.\qedhere
	\end{equation*}
\end{proof}

The main goal of Lemmas \ref{vpath}-\ref{concentration-path} is to show that $\varphi_*(0,Nx)$ can be well approximated by $\sum_{i = 0}^{\zeta_3-1} \E \left[\bar \varphi_*(v_i,v_{i+1})  \right]$ with high probability, which is needed to prove Lemma \ref{mleqan}. The analog of such approximation is used in FPP setup to serve a similar purpose in \cite{Alexander97}. In FPP, the total length of the shortest path from $0$ to $Nx$ is simply equals to the sum of that of non-intersecting path segments that assembles the geodesic. And the concentration can be obtained by using BK's inequality and an exponential concentration bound for each segment. However, our case is a bit more complicated and our proof is different from the arguments in \cite{Alexander97}. In particular, we have applied results on lattice greedy animals (\cite{Lee97,Martin02}). 

\begin{lemma}
\label{vpath}
The following holds for all environments. For all $x \in \Q^d$ with $|x| \geq n^{1/3}$,
	\begin{itemize}
		\item [(1)]$ \zeta_2 - \zeta_1 \geq  N/(2C_Q) \quad \text{and}\quad \zeta_1 + \zeta_3 - \zeta_2\leq 4|x|^{-\tuuovertwo}r(0,Nx) \,.$
		\item [(2)]$ \sum_{i = \zeta_1}^{\zeta_2-1}  \varphi_*(v_i,v_{i+1}) \leq \varphi_*(0,Nx) +3d(\zeta_2-\zeta_1)\log|x|\,.$
	\end{itemize}
	
\end{lemma}

\begin{proof}
(1) By Lemma \ref{QDD}~(1), we have $|v_i - v_{i-1}| \leq C_Q|x|$ for all $\zeta_1 + 1\leq i\leq \zeta_2$. In addition, we see that $|v_{\zeta_1} - v_{\zeta_2}| \geq N |x| - 2 r(0, Nx) \geq  N|x|/2$.
This implies that  $\zeta_2 - \zeta_1 \geq |Nx|/(2C_Q|x|) = N/(2C_Q)$.

For $y \in \Z^d$ with $|y| \leq |x|^{1/3}$. By \eqref{eq:gx<g}, $g_x(y) \leq g(y) \leq \Cbar|x|^{1/3} \leq g(x)$ and $h(y) - g_x(y) \leq 2\Cbar|x|^{1/3} \leq 10C_Q|x|^{\tuu}$. Then by \eqref{eq:def-Qx}, \begin{equation}
	B_{|x|^{\tuuovertwo}}(0) \subseteq Q_x \,.
\end{equation}
Hence, by \eqref{eq:defin-v-i-1} we have that for $i \leq \zeta_1-1$
$$|u_1 - v_i| \leq |u_1 - v_{i-1}| - |x|^{\tuuovertwo}/2\,.$$
Since $|u_1| \leq r(0,Nx)$, we have $\zeta_1 \leq 2|x|^{-\tuuovertwo}r(0,Nx)$. Similarly $\zeta_3 - \zeta_2 \leq 2|x|^{-\tuuovertwo}r(0,Nx)$. Thus, we complete the proof of (1).

\noindent (2)  By considering the last visit to $v_{i-1} + Q_x= \{v_{i-1} + y: y\in Q_x\}$, we get for $\zeta_{1} +1 \leq i \leq \zeta_{2}-1$, $G_{\cregion \setminus \{v_{\zeta_1},v_{\zeta_1+1},...,v_{i-1}\}}(v_{i-1},u_2;\lambda)$ equals to 
	\begin{align*}
	\qquad &\sum_{\mathclap{\substack{u :~u -  v_{i-1}\in\partial_iQ_x\\ u \not\in \{v_{\zeta_1},v_{\zeta_1+1},...,v_{i-1}\}} }}\quad G_{\cregion \setminus \{v_{\zeta_1},v_{\zeta_1+1},...,v_{i-1}\}}(v_{i-1},u;\lambda) \cdot G_{\cregion \setminus (\{v_{\zeta_1},v_{\zeta_1+1},...,v_{i-1}\} \cup (v_{i-1}+Q_x))}(u,u_2;\lambda)\\
		\leq&\qquad \sum_{\mathclap{\substack{u :~u -  v_{i-1}\in\partial_iQ_x\\ u \not\in \{v_{\zeta_1},v_{\zeta_1+1},...,v_{i-1}\}}} }\quad  G_{\cregion \setminus \{v_{\zeta_1},v_{\zeta_1+1},...,v_{i-1}\}}(v_{i-1},u;\lambda) \cdot G_{\cregion \setminus \{v_{\zeta_1},v_{\zeta_1+1},...,v_{i-1},u\}}(u,u_2;\lambda)\,.
	\end{align*}
	 By Lemma \ref{QDD} (1), $|\partial_i Q_x| \leq (2C_Q|x|)^{d}<|x|^{2d}$. Then by \eqref{eq:defin-v-i-2} we get that (using the simple fact that the sum of non-negative numbers is bounded above by the product of the maximal term and the number of terms in the summation)
	\begin{align*}
	 &G_{\cregion \setminus \{v_{\zeta_1},v_{\zeta_1+1},...,v_{i-1}\}}(v_{i-1},u_2;\lambda)  \\
	&\leq G_{\cregion \setminus \{v_{\zeta_1},v_{\zeta_1+1},...,v_{i-1}\}}(v_{i-1},v_i;\lambda) G_{\cregion \setminus \{v_{\zeta_1},v_{\zeta_1+1},...,v_i\}}(v_i,u_2;\lambda) |x|^{2d} .
	\end{align*}
	Taking logarithm for the above inequality and summing over $\zeta_{1} +1 \leq i \leq \zeta_{2}-1$, we have
	\begin{equation*}
		\sum_{i = \zeta_{1}}^{\zeta_2-1} \varphi(v_i,v_{i+1}) \leq \varphi(u_{1},u_2) +3d(\zeta_2 - \zeta_{1}-1)\log|x|\,, 
	\end{equation*}
	where we used \eqref{eq:phi=logG} and Corollary \ref{G-polylogn}.
Since $\varphi_*(v_i,v_{i+1})\leq \varphi(v_i,v_{i+1}) $, combining with $\varphi_*(0,Nx) = \varphi(u_1,u_2)$, we completes the proof of (2).
\end{proof}
Next, we prove some concentration results for LWGFs. 
\begin{defn}
\label{def-gpt}
	For fixed $\lambda \in [\lambda_*,1]$ and $x \in \Q^d$ with $|x| \geq n^{1/3}$, let $\gpt$ be the collection of sites $u\in \Z^d$ such that for all $v$ satisfying $|x|^{\tuuovertwo} \leq |u-v| \leq C_Q|x|$,
\begin{equation}
\label{eq:goodv}
	 \E\left[\bar \varphi_*^{\circ}(u,v)\right] - \bar \varphi^{\circ}_*(u,v) \leq |u-v|^{\tuu}  \AND 2 \cul |x-y| \leq \bar \varphi^{\circ}_*(u,v) \leq 2^{-1}\Cbar|u-v| \,.
\end{equation}
\end{defn}
In the preceding definition we used $\bar \varphi^{\circ}_*(u,v)$ to take advantage of the fact that it only depends on local environment, and the fact that the concentration of $\bar \varphi_*^{\circ}(u,v)$ around its expectation is sufficient to guarantee the concentration of $\bar \varphi_*(u,v)$ (Lemma \ref{con-con}) (which plays an important role in proving the concentration of $\sum \bar\varphi_*(v_i,v_{i+1})$ in Lemma~\ref{concentration-path}).
\begin{lemma} 
\label{good--v-prob}
For all $u \in \Z^d$
	$$\P(u \not \in \gpt) \leq |x|^{-10d^2}\,.$$
\end{lemma}

\begin{proof}
For any $|u-v| \geq |x|^{\tuuovertwo}$, since $|x|^{1/3} \geq n^{\fuu}$, by Lemma \ref{concentration}
$$ \P(\E\left[\bar \varphi_*(u,v)\right] - \varphi_*(u,v) \geq 2^{-1}|u-v|^{2/3} ) \leq |u-v|^{- 60d^2}\leq |x|^{-20d^2}\,.$$
By Lemma \ref{phic-app},
$\P(\bar \varphi_*(u,v) \geq 4^{-1}\Cbar|u-v|  \ or \ \bar \varphi_*(u,v) \leq 4\cul|u-v|) \leq e^{-3^{-2}(\log |x|)^{2}}\,.$
Combined with Lemma \ref{con-con}, it yields that
\begin{equation*}
	\P(u \not \in \gpt)  \leq (2C_Q|x|)^d (|x|^{-20d^2} +  e^{-3^{-2}(\log |x|)^{2}}) \leq |x|^{-10d^2}\,. \qedhere
\end{equation*}

\end{proof}

\begin{lemma}
\label{concentration-path}
For any $m_0 \geq |x|^{10d}$ and $|x| \geq n^{1/3}$, the following holds with $\P$-probability at least $1 - e^{-m_0|x|^{-5d}}$: For all $m \geq m_0$ and $v_0,v_1,...,v_m  \in \Z^d$ such that $|v_0| \leq e^{(\log m)^{1/d}}$ and $|x|^{\tuuovertwo}\leq|v_i -v_{i-1}|\leq C_Q|x|\quad i = 1,2,...,m-1$, we have that
		\begin{align*}
		\sum_{i = 0}^{m-1} \E \left[\bar \varphi_*(v_i,v_{i+1})  \right] \leq \sum_{i = 0}^{m-1}  \varphi_*(v_i,v_{i+1}) +  C_Qm|x|^\tuu\,.
	\end{align*}
\end{lemma}

\begin{proof}
By Lemma \ref{QDD}~(1) and \eqref{eq:def-barphi}, $\bar \varphi_*(v_i,v_{i+1}) \leq \Cbar C_Q|x|$. By Corollary \ref{G-polylogn}, we have that $\varphi_*(v_i,v_{i+1}) \geq - C_3\log \log n$. Then by \eqref{eq:goodv} and Lemma \ref{con-con}, we get the following two inequalities:
	\begin{equation*}
		\sum_{i = 0}^{m-1} \bar \varphi_*(v_i,v_{i+1}) \leq \sum_{i = 0}^{m-1}  \varphi_*(v_i,v_{i+1})+|\{i\in\{0,1,...,m-1\}:v_i \not \in \gpt\}|\cdot (\Cbar C_Q|x| + C_3 \log \log n)\,,
	\end{equation*}
		\begin{equation*}
		\sum_{i = 0}^{m-1} \E \left[\bar \varphi_*(v_i,v_{i+1})  \right] - \bar \varphi_*(v_i,v_{i+1}) \leq 2m(C_Q|x|)^\tuu + |\{i\in \{0,1,...,m-1\}:v_i \not \in \gpt\}| \cdot \Cbar C_Q|x|\,.
	\end{equation*}
Since $C_Q>10, |x| \geq n^{1/3}$, it suffices to prove that  with $\P$-probability at least $1 - e^{-m_0|x|^{-5d}}$
\begin{equation} \label{eq-458}
	|\{i \in \{0,1,...,m-1\} :v_i \not \in \gpt\}| \leq m|x|^{-2}\,.
\end{equation}
To this end, we denote $\calV_i =i + |x|^2\Z^d$ for $i \in \{1,2,...,|x|^2\}^d$, where $\calV_i$ inherits the graph structure from the natural bijection which maps $v\in \Z^d$ to $i + |x|^2 v \in \calV_i$. 
For any integer $M>0$, 
  \begin{align*}
   & \P(\max_{\eta \in \calW_{M}(v_0)}|\eta \cap \calV_i \cap \gpt^c| \geq   M|x|^{-4d})\\
     \leq &\P(\max_{\eta \in \calW_{M}(v_0)}| \{ y \in \calV_i : \eta \cap K_{|x|^2}(y) \not = \varnothing  \}\cap \gpt^c| \geq  M|x|^{-4d})\,.
  \end{align*}
where $\calW_{M}(v_0)$ is the set of connected self-avoiding paths in the original lattice of length $M$ with initial point $v_0$. Since the event $v \in \gpt$ is independent of $\sigma(u \in \gpt, u\in\Z^d: |u-v|\geq |x|^2)$, we have that events $\{v \in \gpt^c\}$ for $v \in \calV_i$ are independent. Also, for any $\eta \in \calW_{M}(v_0)$, we know that $\{ y \in \calV_i: \eta \cap K_{|x|^2}(y)\not = \varnothing \}$ is a lattice animal in $\calV_i$ of size at most $ 3^d M/|x|^2$. By Lemma~\ref{good--v-prob}, $$M|x|^{-4d} \geq |x| \cdot (  (3^dM|x|^{-2})\cdot \P(v \in \gpt^c)^{1/d})\,.$$Then a result on greedy lattice animals proved in \cite[Page 281]{Lee97} (see also \cite{Martin02}) yields
\begin{equation}
\label{eq:greedyLA}
	\P(\max_{\eta \in \calW_{M}(v_0)}|\eta \cap \calV_i \cap \gpt^c| \geq  M|x|^{-4d}) \leq e^{-M|x|^{-4d}/2}\,.
\end{equation}
Note that for all $v_0,v_1,...,v_m$ such that $|v_i - v_{i-1}| \leq C_Q|x|$, there exists a self avoiding path $\eta$ that goes through $v_0,v_1,...,v_m$ and has length $dC_Qm|x|.$ Thus, \eqref{eq-458} follows by summing \eqref{eq:greedyLA} over $i \in \{1,2,...,|x|^2\}^d$, $M = dC_Qm|x|$, $m \geq m_0$ and $v_0 \in B_{ e^{(\log m)^{1/d}}}(0)$. 
\end{proof}

\begin{proof}[Proof of Lemma \ref{mleqan}]
By \eqref{eq:g-def-lim} and \eqref{eq:function-g-bound} and Markov's inequality, we get for sufficiently large $N$,
\begin{equation}
	\P(\bar \varphi_*(0,Nx) \leq Ng(x)+N) \geq 1 - \frac{\E \bar \varphi_*(0,Nx)}{Ng(x)+N} \geq \frac{1}{2\Cbar|x|+2}\,.
\end{equation}
At the same time, by Lemma \ref{vpath}~(1), we can apply Lemma \ref{concentration-path} to $v_{\zeta_1},...,v_{\zeta_2-1}$ and $m_0 = \lfloor N/(2C_Q) \rfloor-1$. Combining with \eqref{eq:phidev}, we get that the following two inequalities hold with positive $\P$-probability for sufficiently large $N$,
		\begin{align*}
				&\sum_{i = \zeta_1}^{\zeta_2-1} \E \left[\bar \varphi_*(v_i,v_{i+1})  \right] \leq \sum_{i = \zeta_1}^{\zeta_2-1}  \varphi_*(v_i,v_{i+1}) +  C_Q(\zeta_2 - \zeta_1)|x|^\tuu\\
	& \varphi_*(0,Nx) = \bar \varphi_*(0,Nx) \leq Ng(x)+N\,. 
	\end{align*}
On this event, combining preceding two inequality and Lemma \ref{vpath} (2), we get that
\begin{align}
\label{eq:mlegan-1}
	\sum_{i = \zeta_1}^{\zeta_2-1} \E \left[\bar \varphi_*(v_i,v_{i+1})  \right]
	\leq Ng(x) +N+ 2C_Q(\zeta_2 - \zeta_1)|x|^{\tuu}\,.
\end{align}
At the same time, by Lemma \ref{QDD}~(1), we have $|v_{i+1}-v_i| \leq C_Q|x|$. Combined with \eqref{eq:def-barphi} and Lemma \ref{vpath} (1), this implies
\begin{align}
\label{eq:mlegan-2}
	\sum_{i = 0}^{\zeta_1-1} \E \left[\bar \varphi_*(v_i,v_{i+1})  \right] + \sum_{i = \zeta_2}^{\zeta_3-1} \E \left[\bar \varphi_*(v_i,v_{i+1})  \right]
	&\leq(\zeta_1 + \zeta_3-\zeta_2)\Cbar C_Q|x|\nonumber\\
	& \leq 4\Cbar C_Q|x|^{2/3}r(0,Nx)\,.
\end{align}
Note that Lemma \ref{vpath}~(1) implies 
$$C_Q(\zeta_2 - \zeta_1)|x|^{\tuu} \geq N|x|^{\tuu} /2\mbox{ and }r(0,Nx) = (\log (Nx))^{2\kappa + 10d}\,.$$ 
It follows that \eqref{eq:mlegan-2} is upper bounded by $C_Q(\zeta_2 - \zeta_1)|x|^{\tuu}$. Then by \eqref{eq:mlegan-1}, we have
\begin{equation}
\label{eq:mlegan-666}
	\sum_{i = 0}^{\zeta_3-1} \E \left[\bar \varphi_*(v_i,v_{i+1})  \right]\leq Ng(x) + 3C_Q\zeta_3|x|^{\tuu}\,.
\end{equation}
Now, by Lemma \ref{QDD}~(1)(2) (recall that $s_x(y):=h(y) - g_x(y)$)
\begin{align*}
	\sum_{i = 0}^{\zeta_3-1} \E \left[\bar \varphi_*(v_i,v_{i+1})  \right] &= \sum_{i = 0}^{\zeta_3-1}g_x(v_i - v_{i+1}) +s_x(v_i - v_{i+1}) \\
	& \geq g_x(Nx) + |\{1\leq i \leq \zeta_3-1:v_i \in \Delta_x \}|\cdot 8C_Q|x|^{\tuu} - \zeta_3|x|^{1/10}\,.
\end{align*}
Combining this with \eqref{eq:mlegan-666} gives
\begin{equation}
\label{eq:m3-1}
	|\{1\leq i \leq \zeta_3-1:v_i \in \Delta_x \}| \leq \zeta_3/2\,.
\end{equation}
At the same time, by Lemma \ref{QDD}(3)
\begin{align*}
	\sum_{i = 0}^{\zeta_3-1} \E \left[\bar \varphi_*(v_i,v_{i+1})  \right] &= \sum_{i = 0}^{\zeta_3-1}g_x(v_i - v_{i+1}) +s_x(v_i - v_{i+1}) \\
	& \geq |\{1\leq i \leq \zeta_3-1:v_i \in D_x \}|\cdot 5g(x)/6 - \zeta_3|x|^{1/10}\,.
\end{align*}
Combining with the lower bound in \eqref{eq:function-g-bound} and \eqref{eq:mlegan-666}, we get
\begin{equation}
\label{eq:m3-2}
	|\{1\leq i \leq \zeta_3-1:v_i \in D_x \}| \leq 6N/5 + \zeta_3/8\,.
\end{equation}
Note that $v_i \in \Delta_x \cup D_x$ for $i \not= \zeta_1,\zeta_2,\zeta_3$; that is, 
\begin{equation}
\label{eq:m3-3}	
|\{1\leq i \leq \zeta_3-1:v_i \in \Delta_x \}| + |\{1\leq i \leq \zeta_3-1:v_i \in D_x \}| \geq \zeta_3-3\,.
\end{equation}
Combining \eqref{eq:m3-1}, \eqref{eq:m3-2} and \eqref{eq:m3-3}, we complete the proof of the lemma.
\end{proof}

\subsection{Proof of Proposition \ref{phi-and-g}} \label{sec:phi-and-g}
In this subsection we combine the ingredients in previous subsections and provide the proof for Proposition \ref{phi-and-g}. We will consider different $\lambda$'s. Thus, we will use the notation $g(v, \lambda)$ and $h(v, \lambda)$ to denote the functions $g(v)$ and $h(v)$ respectively (as defined at the beginning of Subsection~\ref{section:Convergence rate}), but with an explicit emphasis on the dependence of $\lambda$.

\begin{proof}[Proof of Proposition \ref{phi-and-g}]
Lemma \ref{converrate} and \eqref{eq:gandh} implies for all $\lambda \in [\lambda_*,1]$ and $v \in \locV \subseteq B_{\cucc n(\log n)^{-2/d}}(0) \setminus B_{n^{2/3}}(0)$ (c.f. \eqref{eq:locV-def})
$$g(v;\lambda) = h(v;\lambda) + O(n^{4/5})\,. $$
At the same time, we let $$E_v= \big\{\max_{\lambda \in [\lambda_*,1]}|\varphi_\star(0,v;\lambda) - \bar \varphi_*(0,v;\lambda)| \leq (\log n)^C \big\}\,.$$
Then Lemmas \ref{ind-appr} and \ref{phic-app} imply that
$$ \P(\cup_{v \in B_{\cucc n(\log n)^{-2/d}}(0) \setminus B_{n^{2/3}}(0)}~E_v \mid G_0) \geq 1- e^{-c(\log n)^2}\,.$$
Now by \eqref{eq:G-0-0-percolate}, it suffices to prove that conditioned on $G_0$, with $\P$-probability tending to $1$, for all $v \in \locV$
\begin{equation}
 	\label{eq:evstep3}|\bar\varphi_*(0,v;\lambda(v)) - h(v;\lambda(v))| \leq n^{2/3}\,.
 \end{equation}
Note that this does not follow directly from Lemma \ref{concentration} because Lemma \ref{concentration} can provide concentration neither uniformly on all $v \in B_n(0)$ nor on all $ \lambda \in [\lambda_*,1]$. To prove \eqref{eq:evstep3}, we first notice that for any $v \in B_{\cucc n(\log n)^{-2/d}}(0) \setminus B_{n^{2/3}}(0)$, $\{\lambda(v) \geq \lambda_*\}$ and $G_0$ are independent. Thus \eqref{eq:lambda-*-property} yields
$$\P(E_v \mid \lambda(v) \geq \lambda_*,G_0) \geq 1- e^{-c(\log n)^2}\,.$$
Hence,
\begin{align*}
	&\P(|\bar\varphi_*(0,v;\lambda(v)) - h(v;\lambda(v))| \leq n^{2/3} \mid \lambda(v) \geq \lambda_*,G_0)\\
 \geq  &\P(|\varphi_\star(0,v;\lambda(v)) - h(v;\lambda(v))| \leq n^{2/3}/2,E_v \mid \lambda(v) \geq \lambda_*,G_0)\\
 \geq &\P(|\varphi_\star(0,v;\lambda(v)) - h(v;\lambda(v))| \leq n^{2/3}/2 \mid\lambda(v) \geq \lambda_*, G_0) - e^{-c(\log n)^2}\,.
\end{align*}
Since $\sigma(G_0,\varphi_\star(0,v;\lambda), \lambda \in (0,1))$ is independent of $\sigma(\lambda(v))$, this is 
\begin{align*}
	\geq & \min_{\lambda \in [\lambda_*,1]}\P(|\varphi_\star(0,v;\lambda) - h(v;\lambda)| \leq n^{2/3}/2 \mid G_0) - e^{-c(\log n)^2}\\
	\geq & \min_{\lambda \in [\lambda_*,1]}\P(|\bar \varphi_*(0,v;\lambda) - h(v;\lambda)| \leq n^{2/3}/3 \mid G_0) - e^{-c(\log n)^2} - \P(E_v^c \mid G_0)\\
	\geq & 1- n^{1/4}\,.
\end{align*}
where in the last step, we used Lemma \ref{concentration}.
Combined with \eqref{eq:lambda-*-property}, it gives that
\begin{equation*}
  \P\big(|\bar\varphi_*(0,v;\lambda(v)) - h(v;\lambda(v))| > n^{2/3},\lambda(v) \geq \lambda_* \mid G_0\big)\end{equation*}
  is bounded from above by  $n^{-d-1/5}$.
Then using a union bound over $v$ yields that \eqref{eq:evstep3} holds for all $v \in \locV$ conditioned on $G_0$ and thus complete the proof of the lemma.
\end{proof}

\section{Asymptotic ball}\label{section:Asymptotic Ball}

From Theorem~\ref{onecity}, we know that conditioned on survival the random walk is localized in the optimal pocket island of volume poly-logarithmic in $n$. In this section, we will show that in fact, there is a region (which we refer to as \emph{intermittent island}) contained in the optimal pocket island such that the following holds:
\begin{itemize}
\item the intermittent island is asymptotically a discrete Euclidean ball of volume $d\log_{1/p} n$,
\item the principal eigenvalues of the intermittent island and the optimal pocket island are close to each other. 
\end{itemize}
The proof  consists of the following two steps.
\begin{itemize}
\item In Section~\ref{section:upper bound on the size}, we first notice that the region with low density of obstacles (which presumably forms the intermittent island) has low entropy, thus a sharp upper bound on the volume of the intermittent island can be derived.
Then we will show that the principal eigenvalues of the intermittent island and the optimal pocket island are close to each other. A key ingredient in this step is to show that the principal eigenfunction of the optimal pocket island is supported on the intermittent island.
\item  In Section~\ref{section:Asymptotic ball - Faber--Krahn}, we observe that the intermittent island achieves nearly largest eigenvalue (Lemma \ref{eigen-Omega}) over all set of the same volume (Lemma \ref{empty}). Thus, by Faber--Krahn inequality the intermittent island has to be a discrete ball asymptotically. The proof is carried out in Section \ref{section:Asymptotic ball - Faber--Krahn}, where we use a quantitative version of Faber--Krahn inequality as in \eqref{eq:FB}. Note that \eqref{eq:FB} is in the continuous setup but our problem is discrete. To address this, we use the relation between the continuous eigenvalue and its discrete approximation (\cite[(38)]{Kuttler70} (see also \cite[\S4]{Weinberger56}, \cite[\S6]{Weinberger58} and \cite{Polya52}).
\end{itemize}

\subsection{Intermittent island}
\label{section:upper bound on the size}
Recall that $\rad =  \lfloor (\omega_d^{-1}d\log_{1/p} n)^{1/d} \rfloor  $ as defined in \eqref{eq:def-r}.
For $\epl \in (0,1)$, we consider the following disjoint boxes that cover $\Z^d$: 
\begin{equation}
\label{def-smallbox}
	K_{\lfloor \epl\rad\rfloor}(x) = \{y\in \mathbb Z^d: |x-y|_\infty \leq \lfloor \epl\rad\rfloor \} \for x \in  (2\lfloor \epl\rad\rfloor+1)\Z^d\,.
\end{equation}
\begin{defn}
\label{empty-set}
	For $\rho \in (0,1)$, we say a box $K_{\lfloor \epl\rad\rfloor}(x)$ is $(\epl\rad,\rho)$-empty if $$|\ob \cap K_{\lfloor \epl\rad\rfloor}(x)| \leq \rho |K_{\lfloor \epl\rad\rfloor}(x)|\,.$$ Let $\emp(\epl,\rho)$ be the union of $(\epl\rad,\rho)$-empty boxes in \eqref{def-smallbox} that intersect with $B_{(\log n)^\kappa}(v_*)$.
\end{defn}


\begin{lemma}\label{empty}
With $\P$-probability tending to one, for any $ \epl,\rho \in(\rad^{-1/2}, p^{2}d^{-100})$

\begin{equation*}
\P(|\emp(\epl,\rho)| \leq d \log_{1/p}n + \rho^{1/2}\rad^{d}) \geq 1 - e^{-\rad}\,.
\end{equation*}
\end{lemma}
\begin{proof}
For all $x \in \Z^d$, a straightforward computation gives that
\begin{align*}
	\P(K_{\lfloor\epl\rad\rfloor}(x) \text{ is $(\epl\rad,\rho)$-empty}) &\leq \sum_{m = 0}^{\lfloor\rho |K_{\lfloor\epl\rad\rfloor}(x)| \rfloor}\binom{|K_{\lfloor\epl\rad\rfloor}(x)|}{m} p^{|K_{\lfloor\epl\rad\rfloor}(x)| - m}\\
	& \leq \rho |K_{\lfloor\epl\rad\rfloor}(x)| \rho^{-\rho |K_{\lfloor\epl\rad\rfloor}(x)|}p^{|K_{\lfloor\epl\rad\rfloor}(x)|(1 - \rho)}\,.
\end{align*}
Since $\rho \leq p$, we have
\begin{equation}
\label{eq:emptyprob}
	\P(K_{\lfloor\epl\rad\rfloor}(x) \text{ is $(\epl\rad,\rho)$-empty})  \leq \exp\{|K_{\lfloor\epl\rad\rfloor}(x)|(\log p  - 2\rho \log\rho)\}\,.
\end{equation}
Note that $B_{(\log n )^{\kappa }}(v)$ contains at most $(\log n)^{2\kappa d}$ boxes of the form \eqref{def-smallbox} for each $v \in B_{2n}(0)$. We let $q = \P(K_{\lfloor\epl\rad\rfloor}(x) \text{ is $(\epl\rad,\rho)$-empty})$. By \eqref{eq:emptyprob} and $\iota^{-d} =O ((\log n)^{1/2})	$(denoting by $\mathrm{Bin} ((\log n)^{2\kappa d},q)$ a Binomial variable with parameter $(\log n)^{2\kappa d}$ and $q$), 
\begin{align*}
	\P\left(\mathrm{Bin} \left((\log n)^{2\kappa d},q \right) \geq  \omega_d(2\epl)^{-d} (1+\frac{4\rho \log \rho}{\log (1/p)}) \right)\leq n^{-d(1 + \rho\log_p\rho)}\,.
\end{align*}
Therefore, by a union bound, with $\P$-probability at least $ 1 - Cn^{-d\rho\log_p\rho} \geq 1 - e^{-\rad}$, no ball in  $\{B_{(\log n )^{\kappa }}(v): v \in B_{2n}(0)\}$ contains more than $\omega_d(2\epl)^{-d} (1+\frac{4\rho \log \rho}{\log (1/p)})$ many $(\epl\rad,\rho)$-empty boxes. 
\end{proof}

 \begin{defn}
 \label{def-eign-f}
  	Recall that $\fregion$ is the connected component in $B_{(\log n)^\kappa}(v_*)$ that contains $v_*$. Let $f$ be the principal eigenfunction of $P|_{\fregion}$ corresponding to $\lambda_\fregion$ such that $\sum_{v \in \fregion}f(v) = 1$. We extend $f$ to $\Z^d$ by letting $f(v) = 0 $ for $v \in  \fregion^c$. For $\epsilon \in (0,1)$, denote \begin{equation*}
	\Omega_\epsilon = \{ v \in \Z^d : f(v) \geq \epsilon \rad^{-d}\}\,.
\end{equation*}
  \end{defn} 
We will show in Lemma \ref{omegaep} that $\Omega_\epsilon$, the set of sites where the eigenfunction value is high, largely coincide with $\emp(\epsilon^2,\epsilon^2)$ defined in Definition \ref{empty-set} and carries most of the weight of $f$ --- this is due to a simple relation between $f$ and $\emp(\epl,\rho)$ as in Lemma \ref{emptyeig}. Combining with Lemma~\ref{empty}, we are then able to provide a lower bound of the principal eigenvalue of $P|_{\Omega_\epsilon}$ in Lemma~\ref{eigen-Omega}.

\begin{lemma}
\label{eigen-first-region}
With $\P$-probability tending to one
\begin{equation}
\label{eq:eigen-first-region}
 	\lambda_{\fregion} \geq e^{-c_*(\log n)^{-2/d} - C(\log n)^{-3/d}}\,.
 \end{equation}
 Then for any $x \in \fregion$ and $t>0$,
 \begin{equation}
 \label{eq:eigen-first-region-p}
   \Pr^{x}(\xi_{\fregion} > t) \geq e^{-(2\log n)^{\kappa d}} e^{-c_*(\log n)^{-2/d}t(1 - C(\log n)^{-1/d})}\,.
 \end{equation}
\end{lemma}
\begin{proof}
By definitions of $\fregion$ and $C_R(v_*)$, $C_R(v_*) \subseteq\fregion$. Recall \eqref{eq:locV-def} and that $v_* \in \locV$. Lemma \ref{lambda*} gives \eqref{eq:eigen-first-region}. Since $\fregion$ is connected, \eqref{eq:eigen-first-region-p} follows from
   $\max_x\Pr^{x}(\xi_{\fregion} > t) \geq \lambda_{\fregion}^t$ (which can be found in \cite[Lemma 3.2]{DX17}) and $|\fregion| \leq (2 \log n)^{\kappa d}$.
\end{proof}

\begin{lemma}
\label{emptyeig}
For any $\epl,\rho \in (0,1)$ we have
$\sum_{v \in  \emp(\epl,\rho)^c}f(v) \leq C \rho^{-1} \epl^2$.
\end{lemma}
\begin{proof}
By \eqref{eq:eigen-first-region} and $e^x \geq 1 + x$,
\begin{align}
\label{eq:emp-eig-1}
	\sum_{v \in \fregion \setminus \emp(\epl,\rho)}f(v)\Pr^v(\xi_{\fregion} \leq \lfloor\epl \rad\rfloor^2)& \leq \sum_{v \in \fregion}f(v)\Pr^v(\xi_{\fregion} \leq \lfloor\epl \rad\rfloor^2) \nonumber\\
	&= 1 - \lambda_{\fregion}^{\lfloor\epl \rad\rfloor^2} \leq C\epl^2\,.
\end{align}
At the same time, for $v \in \fregion \setminus\emp(\epl,\rho)$, there exists at least $\rho(\epl \rad)^2$ obstacles in $B_{\epl \rad}(v)$. Hence,
\begin{equation}
\label{eq:EcCost}
  \Pr^v(\tau \leq \lfloor \epl \rad \rfloor^2)\geq \max(\Pr^v(S_{\lfloor \epl \rad \rfloor^2} \in \ob),\Pr^v(S_{\lfloor \epl \rad \rfloor^2-1} \in \ob)) \geq   c\rho\,.
\end{equation}
Substituting this into \eqref{eq:emp-eig-1} yields (noting that $\xi_{\fregion} \leq \tau$)
\begin{equation*}
	\sum_{v \in \fregion \setminus \emp(\epl,\rho)}f(v) \leq (c\rho)^{-1}\sum_{v \in \fregion \setminus \emp(\epl,\rho)}f(v)\Pr^v(\xi_{\fregion} \leq \lfloor\epl \rad\rfloor^2) \leq c^{-1}\rho^{-1} C \epl^2\,. \qedhere
\end{equation*}
\end{proof}

\begin{lemma}
\label{omegaep}
Consider $\epsilon \in (\rad^{-c},c)$. The following holds with $\P$-probability at least $ 1 - e^{-\rad}$,
\begin{align}
	\sum_{v \in \fregion \setminus \Omega_\epsilon}f(v) & \leq C\epsilon\,, \label{eq:omegaLocaltime}\\
	|\Omega_\epsilon \setminus\emp(\epsilon^2,\epsilon^2)|  \leq C \epsilon \rad^{d}, \quad & |\Omega_\epsilon \cup \emp(\epsilon^2,\epsilon^2)|  \leq d \log_{1/p}n  + C \epsilon \rad^d \label{eq:omegaSizeeee}\,.
\end{align}
\end{lemma}
\begin{proof} 
Applying Lemmas \ref{empty} and \ref{emptyeig} with $\epl = \rho = \epsilon^2$, we get that with $\P$-probability tending to 1,
	\begin{equation}\label{eq4.12}
		\sum_{v \in \fregion \setminus \emp(\epsilon^2,\epsilon^2)}f(v) \leq C \epsilon^2 \AND
	|\emp(\epsilon^2,\epsilon^2)| - d \log_{1/p}n \leq \epsilon \rad^{d}\,.
\end{equation}
Hence,
\begin{align*}
	\sum_{v \in \fregion \setminus \Omega_\epsilon}f(v) & \leq \sum_{v \in \fregion \setminus \emp(\epsilon^2,\epsilon^2)}f(v) + \sum_{v \in \emp(\epsilon^2,\epsilon^2) \setminus \Omega_\epsilon}f(v)\\
	& \leq \sum_{v \in \fregion \setminus \emp(\epsilon^2,\epsilon^2)}f(v) + |\emp(\epsilon^2,\epsilon^2)|\epsilon \rad^{-d}  \leq C\epsilon\,,
\end{align*}
yielding \eqref{eq:omegaLocaltime}.
Combining \eqref{eq4.12} and the fact that $f(v)\geq \epsilon/\rad^d$ for $v\in \Omega_\epsilon$ gives that
 $$|\Omega_{\epsilon} \setminus\emp(\epsilon^2,\epsilon^2)| \leq \sum_{v \in \fregion \setminus \emp(\epsilon^2,\epsilon^2)}f(v) \epsilon^{-1}\rad^d \leq C \epsilon \rad^{d}\,.$$ Combined with \eqref{eq4.12}, it immediately implies that
$|\Omega_{\epsilon} \cup \emp(\epsilon^2,\epsilon^2)| 
	 					 \leq d \log_{1/p}n + C\epsilon \rad^{d}$. This completes the proof of \eqref{eq:omegaSizeeee} and thus the proof of the lemma.
\end{proof}

\begin{lemma}
\label{eigen-Omega}Consider $\epsilon \in (\rad^{-c},c)$. With $\P$-probability tending to one,
	\begin{equation*}
 	1 - \lambda_{\Omega_\epsilon} \leq c_* (\log n)^{-2/d}(1 + C\epsilon)\,.
 \end{equation*}\end{lemma}
\begin{proof}
	
	Let $\bar f = (f - \epsilon/\rad^d)_+$. (Recall that $a_+ = a \11_{a\geq 0}$.) Then $\bar f$ is supported on $\Omega_\epsilon$ and 
	\begin{equation}
	\label{eq:fbargradient}
		\sum_{x \sim y} (\bar f(x) - \bar f(y))^2 \leq \sum_{x \sim y} ( f(x) -  f(y))^2\,.
	\end{equation}
By Cauchy's inequality and \eqref{eq:omegaLocaltime}, \eqref{eq:omegaSizeeee},
$$|f|_2^2 \geq |\Omega_\epsilon|^{-1} (\sum_{x \in \Omega_\epsilon}f(x))^2 \geq c \rad^{-d}\,.$$
By \eqref{eq:omegaLocaltime},
	\begin{equation*}
		|\bar f|_2^2  = \sum_{x \in \Omega_\epsilon}(f(x) - \epsilon \rad^{-d})^2\geq |f|_2^2 - 2 \epsilon \rad^{-d} |f|_1 - \sum_{x \not \in \Omega_\epsilon}f^2(x) \geq |f|_2^2  - 2 \epsilon \rad^{-d} - \epsilon^2 \rad^{-d}\,.
	\end{equation*}
Hence,
$
	|\bar f|_2^2 \geq |f|_2^2 (1 - C \epsilon)
$.
 Combined with \eqref{eq:fbargradient} and the fact that
\begin{equation*}
  1 - \lambda_A = \min \big\{\frac{1}{4d}\sum_{x \sim y} ( g(x) -  g(y))^2 :  |g|^2_2 = 1,g(x) = 0 \  \forall x \not \in A \big \} \quad \forall A \subseteq \Z^d\,,
\end{equation*}
 it yields 
\begin{equation*} 
	(1 - \lambda_{\Omega_\epsilon}) \leq (1 - \lambda_{\fregion})(1+C\epsilon)\,.
\end{equation*}
We complete the proof of the lemma by \eqref{eq:eigen-first-region}.
\end{proof}

\subsection{Asymptotic ball}
\label{section:Asymptotic ball - Faber--Krahn}

As discussed earlier, in order to apply the quantitative Faber--Krahn inequality we need to relate the principal eigenvalue  in the continuous and discrete setup. To this end, we first provide Lemma \ref{tOmegaSize} which gives an upper bound on the size of the boundary of $\Omega_\epsilon$ --- this will be used in the proof of Lemma~\ref{Asymptotic shape - Ball}. Define
\begin{equation}
\label{eq:def-Omega+}
	\Omega_\epsilon^+ = \{ v \in \Z^d: \min_{x \in \Omega_\epsilon}|x-v|_\infty \leq 2\}\,.
\end{equation}

\begin{lemma}
\label{tOmegaSize} 
Consider $\epsilon \in (\rad^{-c},c)$. With $\P$-probability tending to one, we have
$|\Omega_{{\epsilon}}^+ \setminus  \Omega_{\epsilon^2}| \leq  C{\epsilon} \rad^d$.	
\end{lemma}
\begin{proof}
Let $t_1 = 2d$ and $t_2 = 2d + 1$. By Lemma \ref{eigen-first-region}, for $i = 1,2$ (recalling that $P$ is the transition kernel for the simple random walk with no killing)
\begin{equation}
\label{eq:sizeomega-1}
	\sum_{v \in \Omega_{{\epsilon^2}}}((P|_{\fregion})^{t_i}f)(v) = \lambda_{\fregion}^{t_i} \sum_{v \in \Omega_{{\epsilon^2}}}f(v) \geq1 - 3dc_* (\log n)^{-2/d}  - \sum_{v \not \in \Omega_{{\epsilon^2}}} f(v)\,.
\end{equation}
At the same time, we have
\begin{align}
\label{eq:sizeomega-2}
	\sum_{v \in \Omega_{{\epsilon^2}}}(P^{t_i}f)(v) &= \sum_{u \in \Z^d} (1 - \Pr^u(S_{t_i} \not \in \Omega_{{\epsilon^2}}))f(u) \nonumber\\
&=1 - \sum_{v \not \in \Omega_{{\epsilon^2}}} \sum_{u \in \Z^d} p_{t_i}(u,v)f(u)\,,
\end{align}
where $p_t(\cdot,\cdot)$ is the $t$-step transition probability for simple random walk on $\Z^d$.
Combining \eqref{eq:sizeomega-1} and \eqref{eq:sizeomega-2} (noting $\sum_{v \in \Omega_{{\epsilon^2}}}((P|_{\fregion})^{t_i}f)(v) \leq \sum_{v \in \Omega_{{\epsilon^2}}}(P^{t_i}f)(v)$), we get that
\begin{equation}
\label{eq:sizeomega-1+2}
	\sum_{v \not \in \Omega_{{\epsilon^2}}} \Big( \sum_{u \in \Z^d} (p_{t_1}(u,v) + p_{t_2}(u,v) )f(u) \Big) \leq 2(3dc_*  (\log n)^{-2/d}  +  \sum_{v \not \in \Omega_{{\epsilon^2}}} f(v)) \leq C {\epsilon^2}\,,
\end{equation}
where the second transition follows from \eqref{eq:omegaLocaltime}.
Now, if $ v \not \in \Omega_{{\epsilon^2}}$ and $|v - x|_\infty \leq 2$ for some $x \in \Omega_{{\epsilon}}$, then
$$ p_{t_1}(x,v) + p_{t_2}(x,v) \geq (2d)^{-2d} \AND f(x) \geq {\epsilon} \rad^{-d}\,.$$
Substituting these bounds in \eqref{eq:sizeomega-1+2} yields the desired result.
\end{proof}
\begin{lemma}
\label{Asymptotic shape - Ball}
Consider $\epsilon \in (\rad^{-c},c)$. With $\P$-probability tending to one, there exists a discrete ball $B_\epsilon$ such that 
\begin{equation}
\label{eq: Bn, Omega, E}
\begin{split}
  &|B_\epsilon \cup \emp(\epsilon^2,\epsilon^2) \cup \Omega_\epsilon| \leq d \log_{1/p} n (1+ C \epsilon^{1/2}) \,, \\
 \qquad  &|B_\epsilon \cap \emp(\epsilon^2,\epsilon^2) \cap \Omega_\epsilon| \geq d \log_{1/p} n (1- C \epsilon^{1/2})\,.
\end{split}
\end{equation}
\end{lemma}
\begin{proof}

The proof of the lemma crucially relies on the Faber--Krahn inequality, the application of which requires to approximate a discrete set in $\mathbb Z^d$ by a continuous set in $\mathbb R^d$. For notation clarity, in the proof we will use boldface to denote a subset in $\mathbb R^d$ (which typically has non-zero Lebesgue measure). Following the notation convention, we define
\begin{equation}
	\mathbf{\Omega}_\epsilon^+ = \{y \in \R^d: \min_{x\in \Omega_\epsilon}  \ |y-x|_{\infty} \leq 2 \}\,.
\end{equation}
Recalling \eqref{eq:def-Omega+}, we see that $\Omega_\epsilon^+ = \mathbf{\Omega}_\epsilon^+ \cap \mathbb Z^d$.
We will consider $\mu_{\mathbf{\Omega}_\epsilon^+}$ where $\mu_\cdot$ is defined as in \eqref{eq-def-first-eigenvalue}.
By \cite[(38)]{Kuttler70} (see also \cite[\S4]{Weinberger56},\cite[\S6]{Weinberger58} and \cite{Polya52}), if $1-\lambda_{\Omega_\epsilon}
$ is less than a sufficiently small constant depending on $d$, then
\begin{equation*}
	\mu_{\mathbf{\Omega}_\epsilon^+}\leq (1 - \lambda_{\Omega_\epsilon }) + C(1 - \lambda_{\Omega_\epsilon})^2\,.
\end{equation*}
Combined with Lemma \ref{eigen-Omega}, it yields that with $\P$-probability tending to 1
\begin{equation}
\label{eq:con-lower}
	\mu_{\mathbf{\Omega}_\epsilon^+}\leq c_*  (\log n)^{-2/d}(1 + C \epsilon)\,.
\end{equation}
At the same time, by \eqref{eq:FB}, we have that
\begin{equation}
\label{eq:FB'}
	\mu_{\mathbf{\Omega}_\epsilon^+}|\mathbf{\Omega}_\epsilon^+|^{2/d} - \mu_{\mathbf{B}} |\mathbf{B}|^{2/d} \geq  c_d \mathcal A(\mathbf{\Omega}_\epsilon^+)^2\,,
\end{equation}
where $ \mathcal A(\mathbf{\Omega}_\epsilon^+)$ is defined as in \eqref{eq-def-mathcal-A}, and $\mathbf B\subseteq \mathbb R^d$ is an arbitrary continuous ball.
Note that $|\Omega_\epsilon^+| = |\{y \in \R^d: \min_{x\in \Omega_\epsilon^+}  \ |y-x|_{\infty} \leq 1/2 \}| = |\{y \in \R^d: \min_{x\in \Omega_\epsilon}  \ |y-x|_{\infty} \leq 2 + 1/2 \}| \geq |\mathbf{\Omega}_\epsilon^+ |$.
By Lemma \ref{tOmegaSize} and \eqref{eq:omegaSizeeee}, we have with $\P$-probability tending to 1
\begin{equation*}
	|\mathbf{\Omega}_\epsilon^+ | \leq |\Omega_\epsilon^+ | \leq |\Omega_\epsilon^+ \setminus\Omega_{\epsilon^2} |+| \Omega_{\epsilon^2} | \leq d \log_{1/p} n (1 + C \epsilon).
\end{equation*}
Combined with \eqref{eq:con-lower}, \eqref{eq:FB'}, it gives that
$\mathcal A(\mathbf{\Omega}_\epsilon^+)^2 \leq C \epsilon$. We denote by $\mathbf{B}_\epsilon$ the ball in $\R^d$ that achieves the minimum in $\mathcal A(\mathbf{\Omega}_\epsilon^+)$. Note that $$|\{y \in \R^d: |y-x|_\infty<1/2\} \setminus \mathbf{B}_\epsilon| \geq c \quad \mbox{ for all } x \not \in \mathbf{B}_\epsilon\,.$$ Since $\{y \in \R^d: |y-x|_\infty<1/2\} \setminus \mathbf{B}_\epsilon$ for $x \in \Omega_\epsilon \setminus \mathbf{B}_\epsilon$ are disjoint subsets of $\mathbf{\Omega}_\epsilon^+ \setminus \mathbf{B}_\epsilon$, a volume calculation yields $|\Omega_\epsilon \setminus \mathbf{B}_\epsilon| \leq C\epsilon^{1/2}|\mathbf{B}_\epsilon|$. Now let $B_\epsilon = \mathbf B_\epsilon \cap \mathbb Z^d$ (so $B_\epsilon$ is a discrete ball). We have\begin{equation}
\label{eq:Om-Bn}
   |B_\epsilon| \leq d\log_{1/p} n (1 + C \epsilon), \quad |\Omega_\epsilon \setminus B_\epsilon| \leq C\epsilon^{1/2}|B_\epsilon|\,.
 \end{equation} 

Also, combining \eqref{eq:con-lower} and \eqref{eq:FB'} yields $ |\mathbf{\Omega}_\epsilon^+| \geq (\mu_{\mathbf{\Omega}_\epsilon^+}^{-1}\mu_{B})^{d/2} |B| \geq d \log_{1/p} n(1  -  C \epsilon)$. Combined with $|\mathbf{\Omega}_\epsilon^+| \leq |\Omega_\epsilon^+| $ and Lemma \ref{tOmegaSize}, it yields that with $\P$-probability tending to 1
$$|\Omega_{\epsilon^2}| \geq |\Omega_\epsilon^+| - |\Omega_\epsilon^+ \setminus \Omega_{\epsilon^2}| \geq d \log_{1/p} n(1  -  C \epsilon)\,.$$ 
We replace $\epsilon$ by $\sqrt{\epsilon}$ in the preceding display and get that with $\P$-probability tending to 1
\begin{equation}
\label{eq:Omega size lb}
|\Omega_\epsilon| \geq  d \log_{1/p} n(1  -  C \epsilon^{1/2})\,.
\end{equation}  
Combined with \eqref{eq:Om-Bn}, \eqref{eq:omegaSizeeee}, it yields that
\begin{align*}
|B_\epsilon \setminus \Omega_\epsilon| &\leq |B_\epsilon| + | \Omega_\epsilon \setminus B_\epsilon| - |\Omega_\epsilon| \,  \leq C \,\epsilon^{1/2} |B_\epsilon|\,,\\
|\emp(\epm^2,\epm^2) \bigtriangleup \Omega_\epsilon| &\leq |\emp(\epm^2,\epm^2)| + 2|\Omega_\epsilon \setminus \emp(\epm^2,\epm^2)| - |\Omega_\epsilon| \,  \leq C\, \epsilon^{1/2}\rad^d\,.
\end{align*}
Combined with \eqref{eq:Om-Bn} and \eqref{eq:Omega size lb}, this completes the proof of the lemma.
\end{proof}
\section{Localization on the intermittent island}
\label{section:Localization on intermittent island}
This section is devoted to the proof for localization on the intermittent island.
Denote 
\begin{equation}\label{eq:def-hbn}
	\tOmega := B_\epsilon \cap \Omega_\epsilon,\quad \hbn := \cup_{x \in \tOmega}B_{\epsilon^{1/10d}\rad}(x)
\end{equation}
where $B_\epsilon$ is the ball chosen in Lemma~\ref{Asymptotic shape - Ball}. 
We will prove that the random walk will be in $\hbn$ (a neighborhood of $\tOmega$)
with high probability at any given time $t$ after hitting $\tOmega$ and thus complete the proof of Theorem \ref{ballthm}. To this end, in Section \ref{sec:survival} we prove a couple of estimates on survival probabilities, building on which in Section~\ref{sec:localization-last} we provide proof for localization.

\subsection{Survival probability estimates}\label{sec:survival}
The main results of this subsection are the following two survival probability estimates: Lemma~\ref{rstepsnt} says the survival probability of random walk staying outside $\tOmega$ is very low, and Lemma~\ref{Survival probability lower bound by eigenvalue} gives a lower bound on the survival probability of the random walk staying in $\fregion$ depending on its starting point (via the principal eigenfunction).
\begin{lemma}
\label{rstepsnt}
Consider $\epsilon \in (\rad^{-c},c)$. With $\P$-probability tending to one, for all $x\in\Z^d$ and $t \geq C \epsilon^{-1/d}\rad^2$,
\begin{equation}
    \Pr^x(\xi_{\fregion \setminus \tOmega} > t) \leq 2\exp({-c \epsilon^{-1/d} \rad^{-2}t})\,.
  \end{equation}
\end{lemma}
\begin{remark}
\label{rmk:b near v-*}
 It follows from Lemma \ref{rstepsnt} that $B_\epsilon \subseteq B_{(\log n)^{\kappa/2}}(v_*)$. Because \eqref{eq:lambda-*-property} implies $\Pr^{v_*}(\xi_{\calC_R(v_*)}>t) \geq c(n) \lambda_*^t \gg 2\exp({-c \epsilon^{-1/d} \rad^{-2}t}) \geq \Pr^{v_*}(\xi_{\fregion \setminus \tOmega} > t) $ for sufficiently large $t$. Hence $\tOmega \cap \calC_R(v_*) \not = \varnothing$.
\end{remark}
\begin{lemma}
\label{Survival probability lower bound by eigenvalue}
With $\P$-probability tending to one, for any $v \in \fregion$, 
	\begin{equation}
		\Pr^v( \xi_{\fregion} > t) \geq c \rad^df(v) \lambda_{\fregion}^t\,.
	\end{equation}
\end{lemma}

\subsubsection{Proof of Lemma \ref{rstepsnt} and Lemma \ref{Survival probability lower bound by eigenvalue}}
Lemma \ref{rstepsnt} is a direct consequence of Lemmas \ref{Travel in Bad Region - 1} and \ref{Travel in Bad Region - 2} (below): Lemma \ref{Travel in Bad Region - 1}  states that the random walk can not spend two much time in $\emp(\epm^2,\epm^2) \setminus \tOmega$ due to its small size, and Lemma \ref{Travel in Bad Region - 2} states that the random walk can not spend two much time in $\fregion \setminus \emp(\epm^2,\epm^2)$ because in this region the density of the obstacles is too high.

\begin{lemma}
\label{Travel in Bad Region - 1}
Consider $\epsilon \in (\rad^{-c},c)$. With $\P$-probability tending to one, for any $v \in \fregion$, $m > C\epsilon^{1/d}\rad^2$,
  \begin{equation}
    \Pr^v(|\{0 \leq t \leq m : S_t \in  \emp(\epm^2,\epm^2) \setminus \tOmega\}| \geq m/3) \leq \exp({-cm \epsilon^{-1/d}\rad^{-2}})\,.
  \end{equation}
\end{lemma}
\begin{proof}
We denote $q = \lceil | \emp(\epm^2,\epm^2) \setminus \tOmega|^{2/d} \rceil$. Then for sufficiently small constant $\delta >0$ and any $t\geq  \delta^{-1} q$, $x \in \Z^d$
  $$\Pr^{x}(S_{t}  \in \emp(\epm^2,\epm^2)\setminus \tOmega) \leq C \delta^{d/2} \leq \delta\,.$$
Therefore for all $x \in \Z^d$
\begin{equation*}
  \Ex^x \left[ |\{0 \leq t \leq \lfloor \delta^{-2} q \rfloor: S_t \in  \emp(\epm^2,\epm^2) \setminus \tOmega\}|\right] \leq \delta^{-1}q \,.
\end{equation*}
We assume $m > \delta^{-2} q$ and $\delta$ is sufficiently small. For $k = 0,1,...,\lfloor m/\lfloor \delta^{-2} q \rfloor\rfloor$, we define $$\theta_k = \11_{|\{k\lfloor \delta^{-2} q \rfloor \leq t <(k+1)\lfloor \delta^{-2} q \rfloor : S_t \in  \emp(\epm^2,\epm^2) \setminus \tOmega\}| \geq \delta^{-3/2}q }\,.$$ Then by strong Markov property, $\theta_k$'s are dominated by i.i.d.\ Bernoulli random variable with parameter $C\delta^{1/2}$. 
A standard large deviation computation for Binomial random variables gives\begin{equation}
\label{eq:sumtheta - binom}
  \Pr(\sum_{k=0}^{\lfloor m/\lfloor \delta^{-2} q \rfloor\rfloor} \theta_k  \geq \delta^{1/4}\lfloor m/\lfloor \delta^{-2} q \rfloor\rfloor) \leq e^{-c\delta^{5/2}m/q}\,.
\end{equation}
Hence, with $\Pr$-probability at least $1 - e^{-c\delta^{5/2}m/q}$, for $m > \delta^{-2} q$,
\begin{align*}
|\{0 \leq t \leq m : S_t \in  \emp(\epm^2,\epm^2) \setminus \tOmega\}|\leq (1 + \sum_{k=0}^{\mathclap{\lfloor m/\lfloor \delta^{-2} q \rfloor\rfloor}} \theta_k)\cdot\delta^{-2}q +  \lfloor m/\lfloor \delta^{-2} q \rfloor\rfloor \cdot\delta^{-3/2}q \leq 3\delta^{1/4}m\,.
\end{align*}
Since \eqref{eq: Bn, Omega, E} yields $q \leq C \epsilon^{1/d} \rad^2$, we complete the proof of the lemma by choosing a sufficiently small constant $\delta$.
\end{proof}

\begin{lemma}
\label{Travel in Bad Region - 2}
Consider $\epsilon \in (\rad^{-c},c)$. With $\P$-probability tending to one, for any $v \in \fregion$, $m \geq C\epm ^2 \rad^2$,
	\begin{equation}
		\Pr^v(|\{0 \leq t \leq m : S_t \in \fregion \setminus  \emp(\epm^2,\epm^2)\}| \geq m/3,\xi_{\fregion}>m) \leq \exp(-c m\epm^{-2}\rad^{-2})\,.
	\end{equation}
\end{lemma}
\begin{proof}

Let $\iota = \rho = \epm^2$, $\zeta_{0} = 0$ and for $m \geq 1$, $$\zeta_m := \inf\{t \geq \xi_{m-1}+\lfloor \epl \rad \rfloor^2:S_t \in \fregion \setminus \emp(\iota,\rho)   \}\,.$$
Note that $|\{0 \leq t \leq m : S_t \in \fregion \setminus \emp(\iota,\rho)\}| \geq m/3$ implies $\zeta_{\lfloor 10^{-1} \iota^{-2} m\rad^{-2} \rfloor}<m$. By strong Markov property at $\zeta_1,\zeta_2,...$ and \eqref{eq:EcCost}
\begin{align*}
	\Pr&(\xi_{\fregion}>m,\zeta_{\lfloor 10^{-1} \iota^{-2} m\rad^{-2} \rfloor}<m) \\
	&\leq \Pr ( S_{[\zeta_{m},\zeta_{m}+ \lfloor \epl \rad \rfloor^2]} \subseteq \ob^c, S_{\zeta_{m}} \in  \fregion \setminus \emp(\iota,\rho)\mathrm{~for}~m = 1,2,..., \lfloor 10^{-1} \iota^{-2} m\rad^{-2} \rfloor-1) \\
	&\leq \exp\{-(10^{-1} \iota^{-2} m\rad^{-2}-2)c\rho\} \,. \qedhere
\end{align*}
\end{proof}

Now we prove Lemma \ref{Survival probability lower bound by eigenvalue}. 
\begin{lemma}
\label{MVP}
For all $v \in \fregion$ 
we have $f(v) \leq C \rad^{-d}$. More generally, for all $l \leq \rad$ we have
$f(v) \leq  Cl^{-d}\sum_{|u-v| \leq l}f(u)$\,.
\end{lemma}
\begin{proof}
For any $v \in \fregion$, we consider stopping time $T = \lfloor l^2\rfloor \wedge \zeta_{U}$, where for $k \geq 1$, $\zeta_k := \inf\{t: |S_t - v|_1 \geq k \}$ and $U$ is independent of both the random walk $(S_t)$ and the environment, and has a uniform distribution on $\{\lfloor l/2\rfloor,\lfloor l/2\rfloor+1,...,\lfloor l\rfloor -1\}$. Note that $|S_T -v| \leq |S_T -v|_1 \leq  l$ and that both $\Pr^v(S_{\lfloor l^2\rfloor} = u) \leq C l^{-d}$ and $\Pr^v(S_{\zeta_U} = u) \leq Cl^{-d}$ hold for all $u \in \Z^d$. So
\begin{equation}
\label{eq:unfffff}
	\Pr^v(S_T = u) \leq Cl^{-d} \for |u-v| \leq l\,. 
\end{equation}
By $P|_{\fregion}f = \lambda_{\fregion} f$, we have $\sum_{u:u \sim v}(2d)^{-1}f(u) = \lambda_{\fregion} f(v)$. Then it follows from Markov property that $\lambda_{\fregion}^{-t}f(S_{t\wedge \xi_{\fregion}})$ is a martingale. Then by \eqref{eq:unfffff} and optional sampling theorem,
\begin{equation*}
 	f(v) = \Ex[\lambda_{\fregion}^{-T}f(S_{T\wedge\xi_{\fregion}})] \leq  \lambda_{\fregion}^{-\rad^2}  \sum_{u} \P(S_T = u) f(u) + \Pr(\xi_\fregion \leq T) \cdot 0\leq C\lambda_{\fregion}^{-\rad^2}l^{-d} \sum_{|u-v| \leq l}f(u)\,.
 \end{equation*} 
 We complete the proof of the lemma by Lemma \ref{eigen-first-region}.
 \end{proof}

\begin{proof}[Proof of Lemma \ref{Survival probability lower bound by eigenvalue}]
	Since $f$ is the eigenfunction of eigenvalue $\lambda_{\fregion}$,
	\begin{equation}
	\label{eq:eigenfunction equation}
		\sum_{u \in \fregion}f(u) \Pr^u(S_t = v,\xi_{\fregion } > t) = \lambda_{\fregion}^t f(v)\,.
	\end{equation}
Applying Lemma \ref{MVP} with $l = \rad$ yields that $	f(u) \leq C \rad^{-d}$ for all $u \in \fregion$. 
Plugging this bound into \eqref{eq:eigenfunction equation} and using reversibility yields
\begin{equation*}
  \lambda_{\fregion}^t f(v) \leq C \rad^{-d} \sum_{u \in \fregion} \Pr^u(S_t = v,\xi_{\fregion } > t) =  C \rad^{-d} \Pr^v(\xi_{\fregion } > t)\,. \qedhere
  \end{equation*}
\end{proof}

\subsection{Proof of localization}\label{sec:localization-last}
In this subsection, we first prove the localization result for the end point (Lemma \ref{endpt-U}), which enables us to give an upper bound of the survival probability in $\fregion$ with a constant error factor as in Lemma \ref{sprbd-lpt}. Next, in Lemma \ref{gt-loc-U}, we prove that the random walk will hit $\tOmega$ in poly-$\log n$ steps conditioned on staying in $\fregion$ and prove a  localization result for any fix time point. Finally, we prove Theorem \ref{ballthm} by combining these ingredients in Lemma \ref{balllemma}.

To start, we consider a random walk starting from $z \in \fregion$ conditioned on staying in $\fregion$ up to time $t$. We assume
\begin{equation}
\label{eq:as-z-t}
	\text{either } z \in \tOmega,t \geq 0 \quad \text{   or   }\quad z \in \fregion, t \geq \rad^\iota\,,
\end{equation}
where $\iota$ is a large constant to be determined in the following lemma. The following lemma guarantees that under this assumption we either starting from $\tOmega$, or the time $t$ is long enough so that the random walk can reach $\tOmega$. 
\begin{lemma}
\label{hit tOmega-as}
Consider $\epsilon \in (\rad^{-c},c)$. There exists $\iota>0$ such that for all $z \in \fregion$ and $t \geq \rad^\iota$,
\begin{equation}
  \Pr^z( S_{[0,t]} \cap \tOmega \not = \varnothing \mid  \xi_{\fregion}>t) \leq  C \exp({-c \epsilon^{-1/d}t \rad^{-2}})\,.
\end{equation}
\end{lemma}
\begin{proof}
Lemma \ref{rstepsnt} implies
$ \Pr^z( S_{[0,t]} \subseteq \fregion \setminus \tOmega) \leq 2\exp({-c \epsilon^{-1/d}t \rad^{-2}})$. Comparing it with \eqref{eq:eigen-first-region-p} and choosing a sufficiently large $\iota$ yields the desired result.
\end{proof}
\def \expb {b}
By our convention of $c$, we can choose a sufficiently small constant $b>0$ such that $\epsilon := \rad^{-\expb}$ satisfies the condition for $\epsilon$ in all previous lemmas (from Lemma \ref{empty} to Lemma \ref{hit tOmega-as}). We will fix $\epsilon = \rad^{-\expb}$ henceforth.

\begin{lemma}
\label{endpt-U} Recall $\hbn$ as in \eqref{eq:def-hbn}. We assume \eqref{eq:as-z-t}. For $\epsilon>0$ sufficiently small,
  \begin{equation}
\label{eq:Endpoint Localization}
  \Pr^z( S_t \in \hbn \mid \xi_{\fregion}>t) \geq 1 - \exp({-c\rad^{\expb/(10d)}})\,. 
\end{equation}
\end{lemma}

\begin{proof}
The proof divides into two steps.
In {\bf Step 1} we consider the last visit to $\tOmega$ conditioned on survival: the last excursion should be short because the survival probability (if not coming back to $\tOmega$) decays very fast in light of Lemma \ref{rstepsnt} while Lemma \ref{Survival probability lower bound by eigenvalue} shows that the survival probability starting from $\tOmega$ is comparably large. Then in \textbf{Step 2}, we show that the random walk can not go too far in the last few steps.

\noindent \textbf{Step 1.}  We consider the last visit time $j$ to $\tOmega$ before time $t$. By  Lemma \ref{rstepsnt} and Markov property, we get
	\begin{equation*}
		\Pr^z(S_j \in \tOmega, S_{[j+1,t]} \subset \fregion \setminus \tOmega, \xi_\fregion>t) \leq \Pr^z(S_t \in \tOmega, \xi_\fregion >j) \cdot 2\exp({-c  \rad^{-2+\expb/d}(t-j-1)})\,.
	\end{equation*}
	In addition, we deduce from Lemma \ref{Survival probability lower bound by eigenvalue} and Lemma \ref{eigen-first-region} that
	\begin{align*}
		\Pr^z(S_j \in \tOmega, \xi_\fregion>t) &\geq \Pr^z(S_j \in \tOmega, \xi_\fregion >j) \cdot c \rad^{-b} \lambda_\fregion^{t-j}\\
		 &\geq \Pr^z(S_j \in \tOmega, \xi_\fregion >j) \cdot c \rad^{-b}  \exp({-C\rad^{-2}(t-j)})\,.
	\end{align*}
Combining preceding two inequalities, we see that for $ j \leq t - \rad^{2 - \expb/(2d)}$,
\begin{equation*}
	\Pr^z(S_j \in \tOmega, S_{[j+1,t]} \subset \fregion \setminus \tOmega \mid \xi_\fregion>t) \leq e^{- c\rad^{-\expb/(2d)}}\,.
\end{equation*}
Then a union bound over $0 \leq j \leq t - \rad^{2 - \expb/(2d)}$ and Lemma \ref{hit tOmega-as} gives
\begin{equation}
\label{eq:endpoint-step2}
	\Pr^z(S_{[t - \rad^{2 - \expb/(2d)},t]} \cap \tOmega = \varnothing \mid \xi_\fregion>t) \leq e^{- c\rad^{-\expb/(2d)}}\,.
\end{equation}
\smallskip
\noindent \textbf{Step 2.} We define stopping time $T_{\star} = \inf \{j\geq t - \rad^{2 - \expb/(2d)}: S_j \in \tOmega\}$. We claim that \begin{equation}\label{eq5.13}
  \Pr^z(T_{\star}<t,\max_{T_\star \leq j \leq \rad^{2 - \expb/(2d)} + T_\star}|S_{j} - S_{T_{\star}}| \leq  \rad^{1- \expb/(5d)}\mid \xi_{\fregion}>t) \geq 1 -\exp({-c\rad^{\expb/(10d)}})\,.
\end{equation}
We note that the desired result is a direct consequence of \eqref{eq5.13}. It remains to prove \eqref{eq5.13}. We first see that \eqref{eq:endpoint-step2} implies
\begin{equation*}
	\Pr^z(T_\star \geq t \mid \xi_\fregion>t) \leq e^{- c\rad^{-\expb/(2d)}}\,.
\end{equation*}
In addition, by strong Markov property at $T_\star$, we get that
\begin{align*}
  \Pr^z\Big(&T_{\star}<t,\max_{T_{\star} \leq j \leq T_{\star}+\rad^{2 - \expb/(2d)}}|S_{j} - S_{T_{\star}}| > \rad^{1- \expb/(5d)} , \xi_{\fregion}>t\Big)\\
   \leq &\Ex^z \Big[ \11_{T_{\star}<t,\xi_{\fregion}> T_{\star}} \Pr^{S_{T_{\star}}}\Big(\max_{0 \leq j \leq \rad^{2 - \expb/(2d)}}|S_{j} - S_{T_{\star}}| > \rad^{1- \expb/(5d)}\Big)\Big]\\
   \leq &C\exp({-c\rad^{\expb/(10d)}})\Pr^z(T_{\star}<t,\xi_{\fregion}> T_{\star}) \leq C\exp({-c\rad^{\expb/(10d)}})\Pr^z(\xi_{\fregion}> T_{\star})\,,
\end{align*}
where we used \cite[Proposition 2.4.5]{Lawler10} in the second inequality.
At the same time, by strong Markov property at $T_\star$ and Lemma \ref{Survival probability lower bound by eigenvalue} and \eqref{eq:eigen-first-region}
\begin{align*}
  \Pr^z(\xi_{\fregion}>t) &\geq \Pr^z(\xi_{\fregion}> T_{\star} + \rad^{2 - \expb/(2d)})\\
  &=\Ex^z \left[ \11_{\xi_{\fregion}> T_{\star}} \Pr^{S_{T_{\star}}}\left(\xi_{\fregion} > \rad^{2 - \expb/(2d)}\right)\right] \geq c\epsilon(1 - 2 c_* \rad^{ - \expb/(2d)})\Pr^z(\xi_{\fregion}> T_{\star})\,.
\end{align*}
Combining the last three inequalities, we deduce \eqref{eq5.13} as required, and thus complete the proof of the lemma.

\end{proof}

\begin{lemma}
\label{sprbd-lpt}
Recall $\hbn$ as in \eqref{eq:def-hbn}. For any $x \in \Z^d, t \geq 0$, $\Pr^x(\xi_{\fregion} > t, S_t \in \hbn) \leq C \lambda_{\fregion}^{t}$.
\end{lemma}
\begin{proof} 
By \eqref{eq:eigen-first-region}, it suffice to consider $t \geq \rad^2$. By Markov property at $\rad^2$,
\begin{align*}
  \Pr^x(\xi_{\fregion} > t, S_t \in \hbn) &\leq \sum_{y}p_{\rad^2}(x,y)\Pr^y(\xi_{\fregion} > t-\rad^2, S_{t-\rad^2} \in \hbn)\\
  & = p_{\rad^2}(x,\cdot)(P|_{\fregion})^{t- \rad^2}\11_{\hbn} \leq C\lambda_{\fregion}^{t-\rad^2}\,.
\end{align*}
where in the last inequality we used $|p_{\rad^2}(x,\cdot)|_2 \leq 1 \times |p_{\rad^2}(x,\cdot)|_\infty \leq C\rad^{-d}$ and $|\hbn| \leq C\rad^d$. We complete the proof of the lemma by \eqref{eq:eigen-first-region}.
\end{proof}

\begin{lemma}
\label{gt-loc-U}
We assume \eqref{eq:as-z-t} and $m\geq t$. Then
\begin{align}
	&\Pr^z( \tau_{\tOmega} \leq\rad^\iota \mid \xi_{\fregion}>m) \geq 1 - \exp(-\rad)\,,\label{eq:hit-tOmega-2}\\
	&\Pr^z(S_{t} \in \hbn\mid \xi_{\fregion}>m) \geq 1 - \exp({-c\rad^{\expb/(10d)}})\label{eq:loc-t}\,. 
\end{align}
\end{lemma}
\begin{proof}
We first give a lower bound on $\Pr^z(\xi_{\fregion}>m)$. Define stopping time $T'_{\star} = \inf \{j\geq t - \rad^2: S_j \in \tOmega\}$. By strong Markov property at $T'_{\star}$ and Lemma \ref{Survival probability lower bound by eigenvalue} and \eqref{eq:eigen-first-region}
\begin{align*}
  \Pr^z(\xi_{\fregion}>m) &\geq \Pr^z(\xi_{\fregion}>T'_{\star} + \rad^2 + m-t)\\ &= \Ex^z[\11_{\xi_{\fregion}>T'_{\star}}\Pr^{S_{T'_{\star}}}(\xi_{\fregion}>\rad^2 + m-t)]\\
&\geq c\epsilon \lambda_{\fregion}^{\rad^2 + m-t} \Pr^z(\xi_{\fregion}>T'_{\star}) \geq  c\epsilon \lambda_{\fregion}^{m-t} \Pr^z(\xi_{\fregion}>t,T'_{\star} \leq t)\,.
\end{align*}
Applying \eqref{eq:endpoint-step2} to $t$ gives $
  \Pr^z(T'_{\star}<t \mid \xi_{\fregion}>t)\geq 1/2$. Hence
\begin{equation}
\label{eq:gt-loc-U -lb}
  \Pr^z(\xi_{\fregion}>m) \geq c\epsilon \lambda_{\fregion}^{m-t} \Pr^z(\xi_{\fregion}>t)\,.
\end{equation}

We are now ready to prove \eqref{eq:hit-tOmega-2}. First, by Lemma \ref{endpt-U}, $
  \Pr^z(S_m \not \in \hbn \mid  \xi_{\fregion}>m) \leq \exp({-c\rad^{\expb/(10d)}}).$
Second, using Markov property at time $\rad^\iota$ and \eqref{sprbd-lpt}, 
$$ \Pr^z(\tau_{\tOmega} > \rad^\iota,S_m \in \hbn,\xi_{\fregion}>m)\leq \Pr^z(\tau_{\tOmega} > \rad^\iota, \xi_{\fregion}>\rad^\iota) C \lambda_\fregion^{m - \rad^\iota}\,.$$
Then by Lemma \ref{hit tOmega-as}, this is bounded above by $$ C \exp({-c \rad^{\expb/d}\rad^{-2} \rad^{\iota}}) \cdot \Pr^z(\xi_{\fregion}> \rad^\iota) \lambda_\fregion^{m - \rad^\iota}\,.$$ Now \eqref{eq:hit-tOmega-2} follows from setting $t = \rad^\iota$ in \eqref{eq:gt-loc-U -lb} and comparing it with this bound.

We next prove \eqref{eq:loc-t}. The case $t=m$ has been treated by Lemma \ref{endpt-U}. We assume $t < m$. By Markov property at time $t$ and Lemma \ref{sprbd-lpt}
\begin{equation*}
\label{eq:gt-loc-1}
   \begin{split}
\Pr^z( S_t \not \in \hbn, S_m \in \hbn,\xi_{\fregion}>m)
 &\leq C \Pr^z( S_t \not \in \hbn,\xi_{\fregion}>t) \lambda_{\fregion}^{m-t} \\
 & \leq \exp({-c\rad^{\expb/(10d)}}) \Pr^z(\xi_{\fregion}>t)\lambda_{\fregion}^{m-t}\,,
\end{split}
 \end{equation*} 
where in the last inequality we used Lemma \ref{endpt-U}. Combining with \eqref{eq:gt-loc-U -lb}, we get
$$ \Pr^z( S_t \not \in \hbn, S_m \in \hbn\mid \xi_{\fregion}>m) \leq \exp({-c\rad^{\expb/(10d)}})\,.$$
We complete the proof of \eqref{eq:loc-t} by Lemma \ref{endpt-U}.
\end{proof}
Theorem \ref{ballthm} follows from the following lemma.
\begin{lemma}
\label{balllemma}
Conditioned on the origin being in the infinite open cluster, with  $\P$-probability tending to one, we have that
\begin{align*}
	&\Pr(\tau_{\tOmega} \geq  C|v_*| \mid \tau>n) \leq e^{-\rad^c}\,,\\
	&\Pr(S_{t \vee \tau_{\tOmega}} \not \in \hbn \mid \tau>n) \leq e^{-\rad^c} ~~\text{for all}~t \leq n\,.
\end{align*}
\end{lemma}
\begin{proof}
Recall that $T = \ft_{v_*} = \tau_{B_{(\log n)^{\kappa/2}}(v_*)}$ as  in Theorem \ref{onecity} . 
By Remark \ref{rmk:b near v-*}, $\tau_\tOmega > T$. We first prove that $\tau_\tOmega \leq  T + \rad^\iota$. By strong Markov property at $T$,
  \begin{align*}
    \Pr(&\tau>T,T \leq C|S_T|, \tau_\tOmega >  T + \rad^\iota, S_{[T,n]} \subseteq \fregion)
    \\& = \Ex \big[\11_{\tau>T,T \leq C|S_T|}\Pr^{S_{T}}(\tau_\tOmega > \rad^\iota, \xi_{\fregion}> n-T)\big]\,.
  \end{align*}
By \eqref{eq:hit-tOmega-2} (since $n - C|S_T|>\rad^\iota$), this is bounded from above by
\begin{equation*}
\exp(-\rad)\Ex \big[\11_{\tau>T,T \leq C|S_T|}\Pr^{S_{T}}(\xi_{\fregion}> n-T)\big] = \exp(-\rad)\Pr(T \leq C|S_T|,S_{[T,n]} \subseteq \fregion, \tau > n)\,.
\end{equation*}
Combining with \eqref{eq:INU}, we have that $\Pr(\tau_\tOmega \leq  T + \rad^\iota,T \leq C|S_T| \mid \tau>n) \geq 1 - e^{-\rad^c}$. Hence $$ \Pr(\tau_{\tOmega} < C|v_*| \mid \tau>n) \geq 1 - e^{-\rad^c}\,.$$Next, by strong Markov property at $\tau_\tOmega$,
  \begin{align*}
    \Pr&(\tau>\tau_\tOmega,S_{t\vee \tau_\tOmega} \not\in \hbn, S_{[\tau_\tOmega,n]} \subseteq \fregion)
     = \Ex \big[\11_{\tau>\tau_\tOmega}\Pr^{S_{\tau_\tOmega}}(S_{(t-\tau_\tOmega)_+} \not \in \hbn, \xi_{\fregion}> n-\tau_\tOmega)\big]\,.
  \end{align*}
By \eqref{eq:loc-t} (Since $S_{\tau_\tOmega} \in \tOmega$), this is bounded from above by
\begin{equation*}
e^{-\rad^c}\Ex \big[\11_{\tau>\tau_\tOmega}\Pr^{S_{\tau_\tOmega}}(\xi_{\fregion}> n-\tau_\tOmega)\big] = e^{-\rad^c}\Pr(S_{[\tau_\tOmega,n]} \subseteq \fregion, \tau > n)\,.
\end{equation*}
We complete the proof of the lemma by \eqref{eq:INU}.
\end{proof}

\end{document}